\documentclass[oneside,english,12pt]{amsart}
\usepackage{hyperref}
\usepackage{txfonts}
\usepackage{amsfonts}
\usepackage{lmodern}

\usepackage[T1]{fontenc}
\usepackage[latin9]{inputenc}
\usepackage{geometry}
\usepackage{amsthm}
\usepackage{amssymb}
\usepackage{ amssymb,amsmath,amsthm,tikz}
\usepackage[all]{xy}
\usetikzlibrary{ matrix,arrows,decorations.pathmorphing}

\usepackage{pinlabel}

\title{The augmentation category map induced by exact Lagrangian cobordisms}
\author{Yu Pan}

\makeatletter
\numberwithin{equation}{section}
\numberwithin{figure}{section}
\theoremstyle{plain}
\newtheorem{thm}{Theorem}[section]
\newtheorem{lem}[thm]{Lemma}
\newtheorem{cor}[thm]{Corollary}
\newtheorem{prop}[thm]{Proposition}

\theoremstyle{definition}
\newtheorem{defn}[thm]{Definition}
\newtheorem{eg}[thm]{Example}

\theoremstyle{remark}
\newtheorem{rmk}[thm]{Remark}

\newcommand{\bbR}{{\mathbb{R}}}
\newcommand{\bbZ}{{\mathbb{Z}}}

\newcommand{\bbF}{{\mathbb{F}}}

\newcommand{\fraka}{\mathfrak{a}}

\newcommand{\calA}{{\mathcal{A}}}
\newcommand{\calB}{{\mathcal{B}}}

\newcommand{\calM}{{\mathcal{M}}}

\newcommand{\calR}{{\mathcal{R}}}

\newcommand{\calT}{{\mathcal{T}}}

\newcommand{\diag}[1]{\begin{center}\begin{minipage}{5in}\xymatrix{#1}\end{minipage}\end{center}}

\def\arr{\ar[r]}
\def\ard{\ar[d]}
\def\aru{\ar[u]}

\usepackage{babel}

\tikzset{node distance=1.5cm, auto}

\begin{document}
\maketitle
\abstract 
To a Legendrian knot, one can associate an $A_{\infty}$ category, the augmentation category.
An exact Lagrangian cobordism between  two Legendrian knots gives a functor of the augmentation categories of the two knots.
We study the functor and establish a long exact sequence relating the corresponding cohomology of morphisms of the two ends.
As applications, we prove that the functor between augmentation categories is injective on the level of equivalence classes of  objects and 
find new obstructions to the existence of exact Lagrangian cobordisms in terms of linearized contact homology and ruling polynomials.

\endabstract

\section{Introduction}
Let $\Lambda_+$ and $\Lambda_-$ be Legendrian submanifolds in the standard contact manifold 
$(\bbR^3, \xi= \ker \alpha)$, where $\alpha = dz - y dx$.
An exact Lagrangian cobordism $\Sigma$ from $\Lambda_-$ to $\Lambda_+$ is a 
$2$-dimensional surface in the symplectization of $\bbR^3$ that has cylindrical ends over $\Lambda_+$ and $\Lambda_-$ with some properties. 
See  Figure \ref{lagcob} for a schematic picture and Definition \ref{cobdef} for a detailed description.

\begin{figure}[!ht]
\labellist
\small

\pinlabel $t$ at -10 350
\pinlabel $\Lambda_+$  at 350 310
\pinlabel $\Sigma$  at 350 220
\pinlabel $\Lambda_-$ at 350 120

\pinlabel $N$ at -10 250
\pinlabel $-N$ at -20 50
\endlabellist

\includegraphics[width=2in]{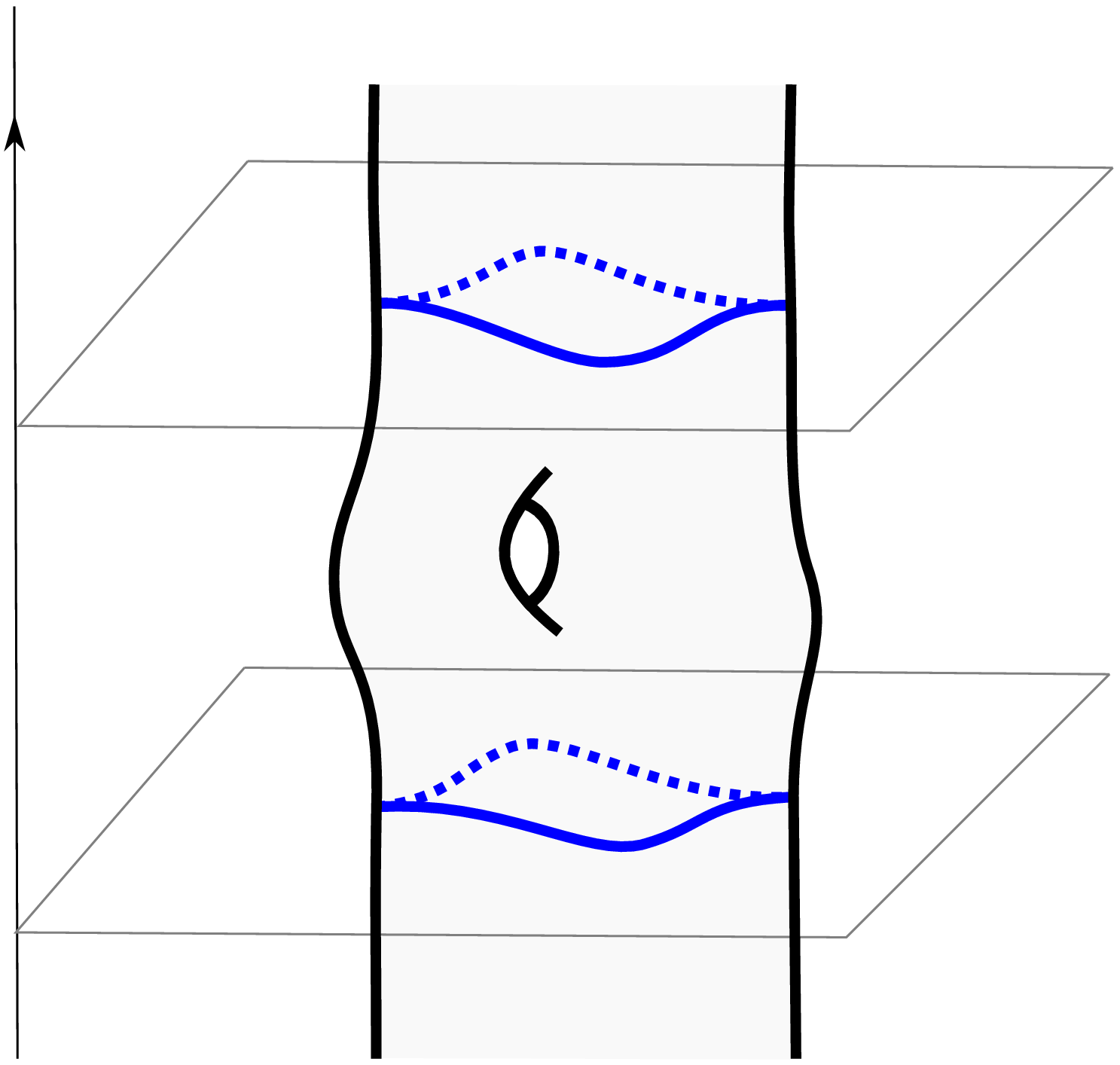}
\caption{A Lagrangian cobordism $\Sigma$ from $\Lambda_-$ to $\Lambda_+$ lies in the symplectization of $\bbR^3$, which is $\big(\bbR_t\times\bbR^3, d(e^t \alpha)\big)$.
The vertical direction is the $t$ direction and each horizontal plane is $\bbR^3$.
Two Legendrian submanifolds $\Lambda_+$ and $\Lambda_-$ sit inside different $\bbR^3$ with $t$ coordinates $N$ and $-N$, respectively.}
\label{lagcob}
\end{figure}

Lagrangian cobordism is a natural relation between Legendrian submanifolds 
and is crucial in the definition of the functorial property of the Legendrian contact homology differential graded algebra (DGA).
For a Legendrian knot $\Lambda$ in $(\bbR^3, \xi= \ker \alpha)$,
the Legendrian contact homology DGA
is a powerful invariant of $\Lambda$ that was 
introduced by Eliashberg \cite{Eli} and  Chekanov \cite{Che}
 in the spirit of symplectic field theory \cite{EGH}.
The underlying algebra  $\calA(\Lambda;\bbF[H_1(\Lambda)])$ is a unital graded algebra freely 
 generated by Reeb chords of $\Lambda$ and a basis of $H_1(\Lambda)$ over a field $\bbF$,
where $H_1(\Lambda)$ is  the singular homology of $\Lambda$ with $\bbZ$ coefficients.
The differential $\partial$ is defined by a count of rigid holomorphic disks in $\bbR \times \bbR^3$ with boundary on  the Lagrangian submanifold $\bbR\times \Lambda$.
The DGA $(\calA(\Lambda;\bbF[H_1(\Lambda)]), \partial)$ is invariant 
up to stable tame isomorphism
under Legendrian isotopy of $\Lambda$.
Ekholm, Honda and K{\'a}lm{\'a}n \cite{EHK} showed that
an exact Lagrangian cobordism  $\Sigma$ from $\Lambda_-$ to $\Lambda_+$ gives a DGA map
 $$\phi_{\Sigma}: (\calA(\Lambda_+;\bbF[H_1(\Sigma)]), \partial) \to (\calA(\Lambda_-; \bbF[H_1(\Sigma)]), \partial),$$
which is defined by a count of rigid holomorphic disks in $\bbR \times \bbR^3$ as well, but with boundary on 
$\Sigma$.
Here $\bbF$ can be any field if the cobordism $\Sigma$ is spin.
If the condition is not satisfied, the field $\bbF$ is assumed to be $\bbZ_2$.

\begin{rmk}

When $\Sigma$ is spin,
the boundary Legendrian knots $\Lambda_{+}$ and $\Lambda_-$ get induced spin structure from the spin structure of $\Sigma$.
This condition
makes the moduli spaces of the holomorphic disks used in the DGA differentials and the DGA map equipped with a coherent orientation (following \cite{EESorientation}). 
In particular, when the dimension of a moduli space is $0$,
one can associate each rigid holomorphic disk in the moduli space with a sign.
Therefore, we can count the disks with sign and  get coefficients in any field $\bbF$.
Otherwise, it is only reasonable to count the disks mod $2$, which means ignoring the orientation. 
For the rest of the paper, we  focus on the case where $\Sigma$ is spin.
If one is working on a non-spin cobordism,
one can omit our description of orientation and get the corresponding statements for $\bbF=\bbZ_2$.
\end{rmk}

A fundamental question about Lagrangian cobordisms is:  given two Legendrian knots $\Lambda_+$ and $\Lambda_-$, does there exist an exact Lagrangian cobordism $\Sigma$ between them?
In order to answer this question, we need to investigate the properties of Lagrangian cobordisms and 
obtain a relationship between Legendrian knots $\Lambda_+$ and $\Lambda_-$.
If the two given Legendrian knots do not satisfy the desired relationship, there does not exist a cobordism between them.
In this way, we can find obstructions to the existence of exact Lagrangian cobordisms.
Chantraine  \cite{Chantraine}  first gave a relationship between  Thurston-Bennequin numbers of the Legendrian knots
\begin{equation}\label{tb}
tb(\Lambda_+)-tb(\Lambda_-) = -\chi(\Sigma).
\end{equation}
This question was explored further in  a lot of work including \cite{BST}, \cite{ST}, \cite{BS}, \cite{CNS} and \cite{CDGG} using generating families, normal rulings and Floer theory.

We approach the question through studying the relationship between the augmentation category of the Legendrian knots that are connected by an exact Lagrangian cobordism.
Analogous to the derived Fukaya category of exact Lagrangian compact submanifolds 
introduced in \cite{NZ},
the augmentation category is an $A_{\infty}$ category of Legendrian knots in $(\bbR^3, \ker \alpha)$.
Bourgeois and Chantraine first introduced  a non-unital  $A_{\infty}$ category in \cite{BC}
and then 
 Ng, Rutherford, Sivek, Shende and Zaslow  introduced a unital version in \cite{NRSSZ}.
We will focus on the latter one.

For a fixed DGA $(\calA(\Lambda),\partial)$ of a Legendrian knot $\Lambda$, the augmentation category $\calA ug_+(\Lambda)$  
consists of objects, morphisms and $A_{\infty}$ operations.
The objects in the category are augmentations $\epsilon$ of the Legendrian contact homology DGA, i.e. DGA maps $\epsilon: (\calA (\Lambda) ,\partial) \to (\bbF, 0)$.
For any two objects $\epsilon_1$ and $\epsilon_2$, 
the morphism space  $Hom_+(\epsilon_1, \epsilon_2)$ is a vector space over the field $\bbF$
generated by Reeb chords
from $\Lambda$ to $\Lambda'$, where $\Lambda'$ is  a positive Morse perturbation of $\Lambda$.
The $A_{\infty}$ operations are composition maps $\{m_n |\  n\ge 1\}$ that satisfy certain relations. 
These relations allow us to take cohomology of the $Hom_+(\epsilon_1, \epsilon_2)$ space with respect to $m_1$, 
denoted by $H^*Hom_+(\epsilon_1, \epsilon_2)$.
From \cite{NRSSZ}, we know that
up to $A_{\infty}$ equivalence, the augmentation category $\calA ug_+(\Lambda)$  is an invariant of Legendrian knots under Legendrian isotopy.

We will show that an exact Lagrangian cobordism $\Sigma$ from a Legendrian knot $\Lambda_-$ to a Legendrian knot $\Lambda_+$ gives 
a DGA map $\phi_{\Sigma}$ from the DGA $\calA(\Lambda_+; \bbF[H_1(\Lambda_+)])$
to the DGA  $\calA(\Lambda_-; \bbF[H_1(\Lambda_-)])$.
By \cite{NRSSZ}, this DGA map induces
an $A_{\infty}$
category map $ f: \calA ug_+(\Lambda_-) \to \calA ug_+(\Lambda_+)$.
As a result, the augmentation category $\calA ug_+$ acts functorially under Lagrangian cobordisms as well.
For each augmentation $\epsilon_-$ of $\calA(\Lambda_-)$, the cobordism $\Sigma$ induces an augmentation $\epsilon_+$ of $\calA(\Lambda_+)$ by  composing with the DGA map $\phi_{\Sigma}$, i.e.,
$$\epsilon_+= \epsilon_- \circ \phi_{\Sigma}.$$
The augmentation category map $f$ sends an object $\epsilon_-$ of $\calA ug_+(\Lambda_-)$ to
the object $\epsilon_+$ of $\calA ug_+(\Lambda_+)$.
For any two objects $\epsilon_-^1$ and $\epsilon^2_-$ in $\calA ug_+(\Lambda_-)$,
the category map $f$ sends the morphism $Hom_+(\epsilon^1_-, \epsilon^2_-)$ to the morphism $Hom_+(\epsilon^1_+, \epsilon^2_+)$, where $\epsilon^1_+$ and $\epsilon^2_+$ are the augmentations induced by $\Sigma$.

We investigate properties of this $A_{\infty}$ category map
through the Floer theory of a pair of exact Lagrangian cobordisms \cite{CDGG},
which is an analog of Ekholm's construction for a pair of  Lagrangian fillings \cite{Ekh} in the spirit of 
symplectic field theory \cite{EGH}.
Let  $\Sigma$ be an exact Lagrangian cobordism from $\Lambda_-$ to $\Lambda_+$. 
Perturb $\Sigma$ using a positive Morse function $F$ and get a new exact Lagrangian cobordism $\Sigma'$.
In \cite{CDGG}, Chantraine, Dimitroglou Rizell, Ghiggini and Golovko constructed a chain complex for this pair of exact Lagrangian cobordisms $\Sigma \cup \Sigma'$, called the Cthulhu chain complex.
The generators of this chain complex are the union of  double points of $\Sigma\cup \Sigma'$ and
Reeb chords on the cylindrical ends from $\Sigma$ to $\Sigma'$.
Indeed,
the second part agrees with union of $Hom_+$ spaces  in the augmentation category of the Legendrian submanifolds on two ends.
The differential of this chain complex is defined by a count of rigid holomorphic disks with boundary on $\Sigma \cup \Sigma'$ as well.
From \cite{CDGG}, the Cthulhu chain complex is acyclic,
which implies the following  long exact sequence:
\begin{thm}[see Corollary \ref{lescor}]\label{thm1}
Let $\Sigma$ be an exact Lagrangian cobordism with Maslov number $0$ from $\Lambda_-$ to $\Lambda_+$.
If
$\epsilon^i_-$, for $i=1,2$, is an augmentation of $\calA(\Lambda_-)$
and $\epsilon^i_+$ is the augmentation of $\calA(\Lambda_+)$ induced by $\Sigma$,
then we have the following long exact sequence:
$$
{\xymatrixcolsep{1pc}
\xymatrix{
\cdots  \arr& H^{k}(\Sigma, \Lambda_-)  \arr& H^k Hom_+(\epsilon^1_+,\epsilon^2_+)  \arr& H^{k}Hom_+(\epsilon^1_-,\epsilon^2_-)  \arr& H^{k+1}(\Sigma, \Lambda_-) \arr& \cdots }}.
$$
\end{thm}

If $\epsilon^1_- = \epsilon^2_-= \epsilon_-$, we can identify
$H^k Hom_+(\epsilon, \epsilon)$ with the linearized contact homology $LCH^{\epsilon}_{1-k}(\Lambda)$ by \cite[Section 5.2]{NRSSZ}.
The long exact sequence above can be rewritten as:

\diag{\cdots \arr & H^{k}(\Sigma, \Lambda_-) \arr & LCH^{\epsilon_+}_{1-k}(\Lambda_+) \arr& LCH^{\epsilon_-}_{1-k}(\Lambda_-)  \arr& H^{k+1}(\Sigma, \Lambda_-)  \arr& \cdots .}
Computing the Euler characteristics of the exact triangle, we have
$$
tb(\Lambda_+)-tb(\Lambda_-) = - \chi(\Sigma),
$$
where $\chi(\Sigma)$ is the Euler characteristic of the surface $\Sigma$.
This result was previously shown by Chantraine in \cite{Chantraine}.

Combine Theorem \ref{thm1} with the augmentation category map induced by exact Lagrangian cobordisms and we have the following theorem.
\begin{thm}[see Theorem \ref{main1}]\label{thm2}
Let $\Sigma$ be an  exact Lagrangian cobordism  with Maslov number $0$  from a Legendrian knot $\Lambda_-$ to a Legendrian knot $\Lambda_+$. 
For $i=1,2$,
if $\epsilon^i_-$ is an augmentation of  $\calA(\Lambda_-)$
and $\epsilon^i_+$ is the  augmentation of $\calA(\Lambda_+)$ induced by $\Sigma$,
the map $$i^0: H^0Hom_+(\epsilon^1_+, \epsilon^2_+) \to H^0Hom_+(\epsilon^1_-, \epsilon^2_-)$$
in the long exact sequence in Theorem \ref{thm1} is an isomorphism.
Moreover, we have that
$$
H^*Hom_+(\epsilon^1_+,\epsilon^2_+) \cong H^*Hom_+(\epsilon^1_-,\epsilon^2_-) \oplus \bbF^{-\chi (\Sigma)}[1],
$$
where $\bbF^{-\chi (\Sigma)}[1]$ denotes the vector space $\bbF^{-\chi (\Sigma)}$ in degree $1$
and $\chi(\Sigma)$ is the Euler characteristic of the surface $\Sigma$.
\end{thm}

This relation was shown for positive braid closures in \cite{Menke}.
Theorem \ref{thm2} shows that this is true for general Legendrian knots.

When $\epsilon^1_-=\epsilon^2_-$, we 
restate Theorem \ref{thm2} in terms of linearized contact homology as follows:
\begin{cor}[see Corollary \ref{LCH}]
Let $\Sigma$ be an exact Lagrangian cobordism with Maslov number $0$ from $\Lambda_-$ to $\Lambda_+$.
If $\epsilon_-$  is an augmentation of  $\calA(\Lambda_-)$ and $\epsilon_+$ is the augmentation of $\calA(\Lambda_+)$ induced by $\Sigma$,
then
$$LCH_*^{\epsilon_+}(\Lambda_+)\cong LCH_*^{\epsilon_-}(\Lambda_-) \oplus \bbF^{-\chi (\Sigma)}[0],$$
where $\bbF^{-\chi (\Sigma)}[0]$ denotes the vector space $\bbF^{-\chi (\Sigma)}$ in degree $0$.
\end{cor}

Therefore, if there exists an exact Lagrangian cobordism $\Sigma$ from $\Lambda_-$ to $\Lambda_+$,
the Poincar{\'e} polynomials of linearized contact homology  of $\Lambda_+$ and  $\Lambda_-$ agree on all degrees except $0$.
In degree $0$ their coefficients  differ by $-\chi(\Sigma)$.
This is a stronger obstruction to the existence of the exact Lagrangian cobordism than the relation (\ref{tb}) of 
the Thurston-Bennequin number.
\begin{figure}[!ht]
\begin{minipage}{2in}
\labellist
\pinlabel $\Lambda_1$ at 80 -5
\endlabellist
\includegraphics[width=2in]{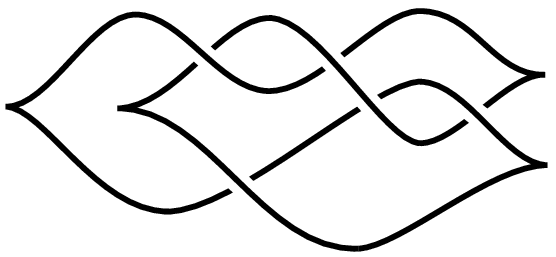}
\end{minipage}
\hspace{0.2in}
\begin{minipage}{2in}
\labellist
\pinlabel $\Lambda_2$ at 240 -30
\endlabellist
\includegraphics[width=2in]{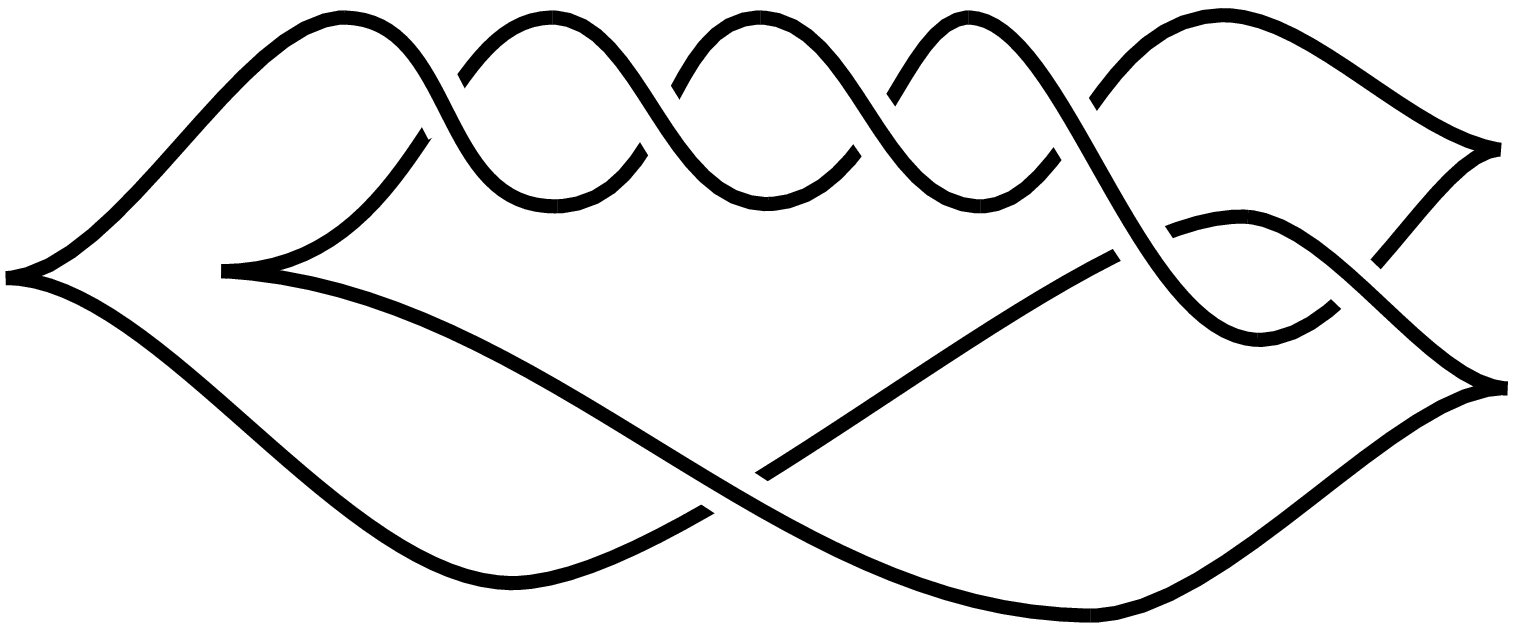}

\end{minipage}
\vspace{0.1in}
\caption{Legendrian knots of knot type $4_1$ and $6_1$.}
\label{4_1vs6_1}
\end{figure}

For instance, 
Figure \ref{4_1vs6_1} shows two Legendrian knots  $\Lambda_1$ and $\Lambda_2$ of smooth knot type $4_1$ and $6_1$, respectively.
There is a topological cobordism between $4_1$ and $6_1$ with genus $1$.
The Thurston-Bennequin numbers of $\Lambda_1$ and $\Lambda_2$ are $-3$ and $-5$, respectively,
and thus satisfy the Thurston-Bennequin number relation given in (\ref{tb}).
Therefore,
 there is a possibility to exist an exact Lagrangian cobordism from $\Lambda_2$ to $\Lambda_1$ with genus $1$.
However, the Poincar{\'e} polynomials of linearized contact homology for $\Lambda_1$ and $\Lambda_2$ are $t^{-1}+2t$ and $2t^{-1} +3t$, respectively.
Thus, we have the following proposition.
\begin{prop}[see Proposition \ref{ex}]
There does not exist  an exact Lagrangian cobordism with Maslov number $0$  from $\Lambda_2$ to $\Lambda_1$, where $\Lambda_1$ and $\Lambda_2$ are shown in Figure \ref{4_1vs6_1}.
\end{prop}

Various long exact sequences that are similar to Theorem \ref{thm1} were explored by people.
Sabloff and Traynor \cite{ST} gave a long exact sequence using generating families.
Chantraine, Dimitroglou Rizell, Ghiggini and Golovko  gave three long exact sequences in \cite{CDGG} in the same spirit as this paper but use different Morse functions to perturb the cobordism.
The way we construct the pair of cobordisms allows us to have more control of the behavior of the Morse function.
This turns out to be a key point towards proving the following surprising theorems.

\begin{thm}[see Theorem \ref{injectivity} and \ref{cocat}]\label{thm5}
Let  $\Sigma$ be an exact Lagrangian cobordism  with Maslov number $0$ from a Legendrian knot $\Lambda_-$ to a Legendrian  knot $\Lambda_+$.
Then the $A_{\infty}$ category map $f: \calA ug_+(\Lambda_-) \to \calA ug_+(\Lambda_+)$ induced by the exact Lagrangian cobordism $\Sigma$  is injective on the level of equivalence classes of objects. 
In addition, the corresponding cohomology category map $ \tilde{f}: H^* \calA ug_+(\Lambda_-) \to H^* \calA ug_+(\Lambda_+)$ is faithful.
In particular, when $\chi(\Sigma)=0$, the functor $\tilde{f}$ is fully faithful.
\end{thm}

By \cite{NRSSZ}, for a Legendrian knot with a single base point, two augmentations  are equivalent if  and only if they are isomorphic as DGA maps.
This theorem  tells us that the number of augmentations of $\Lambda_-$ is smaller than or equal to  the 
the number  of augmentations of $\Lambda_+$ up to  equivalence.
However, in general, it is hard to count the number of augmentations of $\Lambda$ up to equivalence.
Ng, Rutherford, Shende and Sivek \cite{NRSS} introduced a new way to count the augmentations, called the homotopy cardinality,
which is related to the ruling polynomial.
Recall that the ruling polynomial is defined by $R_{\Lambda}(z)=\displaystyle{\sum_{R} z^{-\chi(R)}},$ 
where the sum is over all normal rulings $R$ of $\Lambda$ (see  \cite{Cheruiling} for the detailed definition).
This invariant is much easier to compute than the augmentation equivalence class. 
Using Theorem \ref{thm5}, we have the following corollary.

\begin{cor}[see Corollary \ref{ruling}]
Suppose there exists a spin exact Lagrangian cobordism with Maslov number $0$ from a Legendrian knot $\Lambda_-$ to  a Legendrian knot $\Lambda_+$. 
Then the ruling polynomials satisfy:
$$R_{\Lambda_-}(q^{1/2} - q^{-1/2}) \leq q^{-\chi(\Sigma)/2}R_{\Lambda_+}(q^{1/2} - q^{-1/2})$$
for any $q$ that is a power of a prime number.
\end{cor}

This corollary gives a new obstruction to the existence of exact Lagrangian cobordisms.
In particular, we have a new and simpler proof to the fact given by \cite{Chasymmetric} that
there does not exist an exact Lagrangian cobordism from the Legendrian $m(9_{46})$ knot shown in Figure \ref{9_46} to the Legendrian unknot. 
This fact is crucial to prove that Lagrangian concordance is not  a symmetric relation.

Another important step towards proving the injectivity and faithfulness in Theorem \ref{thm5} is
to understand the differential map of the Cthulhu chain complex better.
Analogous to a result for Legendrian submanifolds in \cite{EESduality},
we give a bijective correspondence between rigid holomorphic disks with boundary on a $2$-copy of $
\Sigma$ and rigid holomorphic disks with boundary on $\Sigma$ together with Morse flow lines.
With this in hand, we can decompose the Cthulhu chain complex in various ways and 
recover the three long exact sequences in \cite{CDGG}.

{\bf Outline. }In Section 2, we  review the Chekanov-Eliashberg DGA of a Legendrian submanifold 
and the DGA map induced by a Lagrangian cobordism.
In Section 3, we   introduce the augmentation category for a Legendrian submanifold  and describe the $A_{\infty}$ category map induced by an exact Lagrangian cobordism.
 In Section 4, we  review the Floer theory of Lagrangian cobordisms.
Finally, using the techniques in Section 4, we  prove the main  result Theorem \ref{thm1}  in Section 5
and discuss its applications.

{\bf Acknowledgments:}
The author would like to thank Lenhard Ng for introducing the problem and many enlightening discussions. The author also thanks Baptiste Chantraine, John Etnyre and Michael Abel for helpful conversations, the referee for pointing out Theorem \ref{thm 1}, Corollary \ref{iso} and Theorem \ref{cocat}, and Caitlin Leverson for comments on an earlier draft.
 This work was partially supported by NSF  grants
DMS-0846346 and DMS-1406371.

\vspace{0.5in}

\section{Legendrian contact homology DGA and exact Lagrangian cobordisms}
\subsection{The Legendrian contact homology DGA}\label{DGA}
In this section, we  review the Legendrian contact homology DGA
from the geometric perspective of \cite{EHK} and the combinatorial perspective of \cite[Section 2.2.1]{NRSSZ}. 
We refer readers to \cite{Che,ENS, Ng} for a more detailed introduction.

Let $\Lambda$ be a Legendrian submanifold in the standard contact space $(\bbR^3, \xi=\ker \alpha)$, where $\alpha = dz -ydx$. 
For simplicity when defining the degree, we assume throughout the paper that $\Lambda$ has rotation number $0$.

Let $(\calA(\Lambda; \bbF[H_1(\Lambda)]), \partial)$ denote the Legendrian contact homology DGA of $\Lambda$,  which is also called  Chekanov-Eliashberg DGA.
The underlying algebra $\calA(\Lambda; \bbF[H_1(\Lambda)])$ is a non-commutative unital graded algebra 
over a field $\bbF$ generated by 
$$\{c_1,\dots , c_m, t_1, t_1^{-1}, \dots , t_M, t_M^{-1}\}$$
 with relations 
$t_i t_i^{-1}=1$ for $ i =1,\dots, M$.
Here $c_1,\dots , c_m$ are Reeb chords of $\Lambda$ 
and $\{t_1, \dots,t_M\} $ is a basis of the singular homology $H_1(\Lambda)$.
The grading of a Reeb chord $c$ is defined as $$|c|= CZ(\gamma_c)-1,$$ 
where $\gamma_c$ is a capping path for $c$
and $CZ$ is the Conley-Zehnder index introduced  by \cite{EES}.
See \cite[Section 4.1]{DR} for the way to choose a capping path $\gamma_c$ for a Reeb chord of a Legendrian link.
The grading of a Reeb chord depends on the choice of capping paths,
but the difference between  two Reeb chords' gradings is independent of the choice of capping paths.
Furthermore, set the grading of $t_i$ to be zero for $ i =1,\dots, M$, and then
extend the definition of degree to $ \calA(\Lambda; \bbF[H_1(\Lambda)])$ through the relation $|ab|= |a|+|b|$.

To define the differential $\partial$, we need to a \textbf{cylindrical almost complex structure} $J$ on $\big(\bbR \times \bbR^3, d(e^t \alpha)\big)$,
i.e.,
\begin{itemize}
\item $J$ is compatible with the symplectic form $d(e^t \alpha)$;
\item $J$ is invariant under the action of $\bbR_t$;
\item $J(\partial_t)= \partial_z$ and $J(\xi)=\xi$.
\end{itemize}

For a generic choice of cylindrical almost complex structure $J$,
the differential $\partial$ is defined by counting rigid $J$-holomorphic disks in $\big(\bbR_t \times \bbR^3, d(e^t \alpha)\big)$ with boundary on $\bbR \times \Lambda$. 
See Figure \ref{differential} for an example.
 For Reeb chords $a, b_1,\dots ,  b_m$ of  $\Lambda$, let
$\calM(a;b_1,\dots ,  b_m)$ denote the moduli space of $J$-holomorphic disks:
$$u: (D_{m+1}, \partial D_{m+1}) \to (\bbR\times \bbR^3, \bbR\times \Lambda)$$
such that
\begin{itemize}
\item $D_{m+1}$ is a $2$-dimensional unit disk with $m+1$ boundary points $p, q_1, \dots , q_m$ removed
and the points $p, q_1, \dots , q_m$ are labeled in a counterclockwise order;
 \item  $u$ is asymptotic to $[0, \infty) \times a$ at $p$;
 \item  $u$ is asymptotic to $( -\infty, 0] \times b_i$ at $q_i$.
\end{itemize} 

\begin{figure}[!ht]
\begin{minipage}{2in}
\begin{center}
\labellist
\small

\pinlabel $\Lambda$  at 140 200
\pinlabel $\bbR \times \Lambda$  at 150 120
\pinlabel $\Lambda$ at 140 50

\pinlabel $a$ at 58 230
\pinlabel $b_1$ at 23 25
\pinlabel $b_2$ at 43 15
\pinlabel $b_3$ at 70 15

\endlabellist

\includegraphics[width=1.4in]{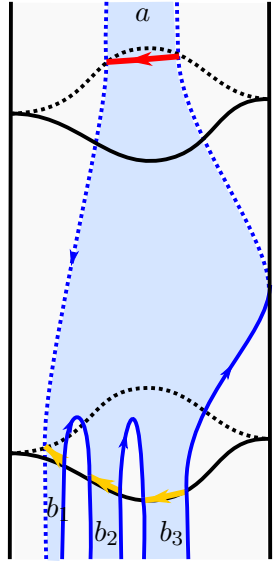}

\end{center}
\label{differential}
\end{minipage}
\begin{minipage}{3in}
\caption{
An example of a $J$-holomorphic disk with boundary on $\bbR\times \Lambda$.
The arrows on the Reeb chords indicate the orientation of the Reeb chords,
while the arrows on the disk boundary indicate the orientation inherited from the unit disk boundary with counterclockwise orientation through $u$.}

\end{minipage}
\end{figure}
Let
$\widetilde{\calM}(a; b_1, \dots, b_m)$ denote the quotient of  $\calM(a; b_1, \dots, b_m)$ by vertical translation of $\bbR_t$.
When $\dim \widetilde{\calM}(a; b_1, \dots, b_m)=0$, the disk $u \in \calM(a; b_1, \dots, b_m)$ is called \textbf{rigid}.
The gradings of corresponding Reeb chords satisfy
$$|a| - |b_1|- \cdot \cdot \cdot -|b_m|=1. $$
For the image of the boundary segment from $q_i$ to $q_{i+1}$ under $u$, one can close it up on $\bbR \times \Lambda$
in a particular way as described in \cite[Section 3.2]{EHK} 
and take the homology class of this curve in $H_1(\Lambda)$, denoted by $\tau_i$.
Here we use $q_0=q_{m+1}=p$.
Moreover, if $\Lambda$ is spin, all the relevant  moduli spaces of $J$-holomorphic disks
admit a coherent orientation. 
Hence, one can associate a sign $s(u)$ to  each rigid $J$-holomorphic disk $u$.
In this way, associate the rigid $J$-holomorphic disk $u$ with a monomial 
$$w(u) =s(u) \tau_0 b_1 \tau_1 \cdots b_{m} \tau_{m}.$$
We call the homology classes $\tau_i$, for $i=1,\dots, m$, the \textbf{coefficients} of $w(u)$.
The {\bf differential} on Reeb chords  is defined by counting rigid $J$-holomorphic disks:
$$\partial a = \displaystyle{\sum_{\dim \widetilde{\calM}(a; b_1, \dots, b_m)=0} \sum_{u \in \calM(a; b_1, \dots, b_m)}  w(u)}.$$
Let $\partial t_i= \partial t_i^{-1}=0$ for $i=1,\dots, M$ and extend the differential to $\calA(\Lambda; \bbF[H_1(\Lambda)])$ through the Leibniz rule
$$\partial(xy) = (\partial x) y + (-1)^{|x|} x(\partial y).$$
An implicit condition for $J$-holomorphic disks is the positive energy constraint.
For a Reeb chord $a$, define the action of $c$ by 
$$\fraka(c) = \int_c \alpha,$$ which is the length of the Reeb chord $c$.
The energy $E(u)$ of a $J$-holomorphic disks $u \in \calM(a; b_1, \dots, b_m)$ satisfies
$$E(u)= \fraka(a) - \fraka(b_1) - \cdots - \fraka(b_m).$$
Therefore, to make each $J$-holomorphic disk endow positive energy, we have 
$$\fraka(b_1) + \cdots + \fraka(b_m) < \fraka(a).$$

There is an equivalent definition from the combinatorial perspective. 
Project $\Lambda$ onto the $xy$-plane to get the Lagrangian projection $\pi_{xy}(\Lambda)$ of $\Lambda$.
After possibly perturbing  $\Lambda$, we can assume that there is a $1-1$ correspondence between the double points of $\pi_{xy}(\Lambda)$ and the Reeb chords of $\Lambda$.
Suppose $\Lambda$ is an $M$-component Legendrian link.
Decorate the diagram with an orientation and a set of \textbf{minimal base points} $\{\ast_1,\dots,\ast_M\}$, i.e.,
\begin{itemize}
\item there is exactly one point in $\{\ast_1,\dots,\ast_M\}$ on each component of $\Lambda$ and
\item the set $\{\ast_1,\dots,\ast_M\}$ does not include any end points of Reeb chords of $\Lambda$.
\end{itemize}
The graded algebra $\calA(\Lambda, \ast_1,\dots , \ast_M)$ is a non-commutative unital graded algebra 
over a field $\bbF$ generated by $\{c_1,\dots , c_m, t_1, t_1^{-1}, \dots , t_M, t_M^{-1}\}$ with relations 
$\{t_i t_i^{-1}=1|\  i =1,\dots, M\}$, where $c_1,\dots , c_m$ are double points of $\pi_{xy}(\Lambda)$ and 
$t_1, \dots, t_M$ correspond to the base points $\ast_1,\dots , \ast_M$.
The grading is defined the same as above.
For the unit disks $D_{m+1}$ as defined above, consider $\Delta(a; b_1, \dots , b_m)$, the space of orientation-preserving smooth immersions up to parametrization 
$$u : (D_{m+1}, \partial D_{m+1}) \to (\bbR^2, \pi_{xy}(\Lambda))$$  
with the following properties:
\begin{itemize}
\item $u$ can be extended to the unit disk $\overline{D_{m+1}}$ continuously;
\item $u(p)=a$ and the neighborhood of $a$ in the image of $u$ is a single positive quadrant (see Figure \ref{crossing});
\item $u(q_i)=b_i$ and the neighborhood of $b_i$ in the image of $u$ is a single negative quadrant for $1\le i \le m$ (see Figure \ref{crossing}).
\end{itemize}

\begin{figure}[!ht]
\begin{center}
\labellist
\small
\pinlabel $+$ at 45 25 
\pinlabel $+$ at 15 25 
\pinlabel $-$ at 30 40
 \pinlabel $-$ at 30 10
\endlabellist
\includegraphics[width=1 in]{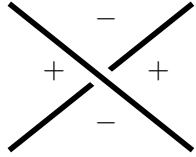}
\caption{At each crossing, the quadrants labeled with $+$ sign are called {\bf positive quadrants} while the other two are called {\bf negative quadrants}.}
\label{crossing}
\end{center}
\end{figure}

If, when traversing $\overline{\partial D_{m+1}}$ counterclockwise from $a$, one encounters Reeb chords and base points in a sequence 
$s_1, \dots, s_l$,
then we associate $u$ with a monomial 
$w(u)=s(u) w(s_1) \cdots w(s_l)$, where
\begin{itemize}
\item  $s(u)$ is the sign associated to the disk $u$ induced from moduli space coherent orientation;
\item if $s_i$ is a Reeb chord $b_j$, then $w(s_i) = b_j$;
\item if $s_i$ is a base point $\ast_j$ for some  $j=1,\dots, M$, then $w(s_i) = t_j$ if the orientation of the boundary agrees with the orientation of the link
and $w(s_i) = t_j^{-1}$ if is otherwise.
\end{itemize}

Define the {\bf differential} on generators as follows:
$$
\begin{array}{rl}
\partial a &= \displaystyle{\sum_{|a| - |b_1|- \cdots -|b_m|=1}\  \sum_{u \in \Delta(a; b_1, \dots, b_m)}  w(u)};\\
&\\
\partial t_j & = \partial t_j^{-1} = 0, \hspace{0.2in} j=1,\dots, M.\\
\end{array}
$$
This can be extended to the whole DGA through the Leibniz rule.

For all the definitions of DGAs $(\calA, \partial)$ above, the differential $\partial$ has degree $-1$
and satisfies  $\partial^2=0$ (\cite{Che}, \cite{ENS}).
Up to stable tame isomorphism, the Legendrian contact homology DGA  is an invariant of $\Lambda$ under Legendrian isotopy. 
In this sense of equivalence, the combinatorial definition does not depend on the choice of base points
\cite{NR}.

However, the homology of the DGA is hard to compute in general.
Let us introduce augmentations of a DGA and use that to deduce  linearized contact homology, which is much easier to compute.
Let $(\calA, \partial)$  be a DGA  over a field $\bbF$ of  a Legendrian link $\Lambda$ with base points.
A \textbf{graded augmentation} of $\calA$ is a DGA map 
$$\epsilon: (\calA, \partial) \to (\bbF, 0),$$
where $(\bbF, 0)$ is a chain complex that is $\bbF$ in degree $0$ and is $0$ in other degrees.
In other words, a graded augmentation
 is an algebra map $\epsilon: \calA \to \bbF$ such that $\epsilon(1)=1$, $\epsilon \circ \partial=0$ and 
$\epsilon(a)=0$ if $|a| \neq 0$.

Given a graded augmentation $\epsilon$,
define  $\calA^{\epsilon}:=\calA \otimes \bbF/(t_i=\epsilon(t_i))$.
Notice that 
the differential
$\partial$ descends to $\calA^{\epsilon}$ since $\partial(t_i)=0$.
Elements in $\calA^{\epsilon}$ are summands of words of Reeb chords.
Let  $C$ be a free $\bbF$-module generated by  Reeb chords.
We can decompose $\calA^{\epsilon}$ in terms of word length as $\calA^{\epsilon}=\displaystyle{\bigoplus_{n \ge 0} C^{\otimes n}}$.
Let $\calA_+^{\epsilon}$  be the part of $\calA^{\epsilon}$ containing the words with length at least $1$, i.e. 
$\calA_+^{\epsilon}= \displaystyle{\bigoplus_{n \ge 1} C^{\otimes n}}$.
Consider a new differential $\partial^{\epsilon} : \calA^{\epsilon} \to  \calA^{\epsilon}$:
$$\partial^{\epsilon} := \phi_{\epsilon} \circ \partial \circ \phi^{-1}_{\epsilon},$$
where
 $\phi_{\epsilon}: \calA^{\epsilon} \to  \calA^{\epsilon}$ is an automorphism defined by 
$\phi_{\epsilon}(a) = a + \epsilon(a)$.
Observe that $\partial^{\epsilon}$ preserves $\calA_+^{\epsilon}$ and does not decrease the minimal length of a word.
Thus, it descends to a differential on $\calA_+^{\epsilon}/(\calA_+^{\epsilon})^2 \cong C$.
The homology of $(C, \partial^{\epsilon})$ is called  \textbf{linearized contact homology} of $\Lambda$ with respect to $\epsilon$, denoted by $LCH_{*}^{\epsilon}(\Lambda)$.
The chain complex $(C, \partial^{\epsilon})$ is called {\bf linearized contact homology chain complex}.

\subsection{Exact Lagrangian cobordisms}\label{sectioncob}
We now review the DGA map induced by an exact Lagrangian cobordism \cite{EHK}
with coefficients in general fields (following the orientation convention of \cite{EESorientation}).
In other words, an exact  Lagrangian cobordism $\Sigma$ from $\Lambda_-$ to $\Lambda_+$ gives a DGA map  $$ \phi_{\Sigma}: \calA(\Lambda_+; \bbF[H_1(\Sigma)]) \to  \calA(\Lambda_-; \bbF[H_1(\Sigma)]).$$
As required in Section \ref{aug},
we restrict to the case where $\Lambda_+$ and $\Lambda_-$ are Legendrian knots with a single base point, denoted by $\ast_+$ and $\ast_-$, respectively.
We modify the DGA map such that the coefficients only depend on the base points but not depend on the cobordism,
i.e., we get a DGA map $$\phi_{\Sigma}: \calA(\Lambda_+, \ast_+) \to  \calA(\Lambda_-, \ast_-).$$

\begin{defn}\label{cobdef}
Suppose  $\Lambda_{\pm}$ are Legendrian submanifolds in $(\bbR^3, \ker \alpha)$, where $\alpha = dz -y dx$.
An \textbf{exact Lagrangian cobordism} $\Sigma$ from $\Lambda_-$ to $\Lambda_+$ is  a $2$-dimensional surface in $\big(\bbR \times \bbR^3, \omega= d(e^t \alpha)\big)$
 (see Figure \ref{lagcob}) such that 
for some big number $N>0$,
\begin{itemize}
\item $\Sigma \cap \big((N, \infty) \times \bbR^3\big) = (N, \infty) \times \Lambda_+ ,$
\item $\Sigma \cap \big((-\infty, -N) \times \bbR^3\big) =(-\infty, -N)  \times \Lambda_- $ and 
\item $\Sigma \cap \big([-N, N] \times \bbR^3\big)$ is compact.
\end{itemize}
Moreover, there exists a smooth function $g: \Sigma \to \bbR$ such that  $$e^t \alpha\mid_{T{\Sigma}} = dg$$
 and $g$ is constant when $t\leq -N$ and $t\ge N$. 
The function $g$ is called a \textbf{primitive} of $\Sigma$.
\end{defn}

For a spin exact Lagrangian cobordism $\Sigma$ from $\Lambda_-$ to $\Lambda_+$, the Legendrian submanifolds $\Lambda_{\pm}$ inherit induced spin structures.
Hence $\Lambda_{\pm}$ have
$\bbF[H_1(\Lambda_{\pm})]$ coefficients DGAs $(\calA(\Lambda_{\pm}; \bbF[H_1(\Lambda_{\pm})]), \partial)$, respectively,
as described in Section \ref{DGA}.
Ekholm, Honda and K{\'a}lm{\'a}n in \cite{EHK} showed that 
an exact Lagrangian cobordism $\Sigma$ induces a DGA map from  $\calA(\Lambda_+)$ to  $\calA(\Lambda_-)$ with $\bbF[H_1(\Sigma)]$ coefficients. 
In order to see that, first, we need to view the DGAs of $\Lambda_{\pm}$ as DGAs with $\bbF[H_1(\Sigma)]$ coefficients.
Notice that the inclusion $H_1(\Lambda_{\pm}) \hookrightarrow H_1(\Sigma)$ induces a canonical inclusion map $\bbF[H_1(\Lambda_{\pm})] \hookrightarrow \bbF[H_1(\Sigma)]$ of the group ring coefficients,
which makes it nature to consider the DGAs  of $\Lambda_{\pm}$ with $\bbF[H_1(\Sigma)]$ coefficients.
Specifically, the new DGA $\calA\big(\Lambda_{\pm}; \bbF[H_1(\Sigma)]\big)$ is generated by Reeb chords of $\Lambda_{\pm}$  and elements in $H_1(\Sigma)$ over $\bbF$.
The differential is defined by the original differential in $\calA\big(\Lambda_{\pm}; \bbF[H_1(\Lambda_{\pm})]\big)$ composed with the inclusion map $H_1(\Lambda_{\pm}) \hookrightarrow H_1(\Sigma)$.

Second, construct a DGA map with $\bbF[H_1(\Sigma)]$ coefficients.
Consider an almost complex structure $J$ that is compatible with the symplectic form $\omega$ and is cylindrical on both ends. 
In other words, $J$ matches the cylindrical almost complex structures on both cylindrical ends.
Fix a generic choice of such an almost complex structure $J$.
For Reeb chords $a$ of $\Lambda_+$ and $b_1, \dots, b_m$ of $\Lambda_-$,
 define $\calM(a; b_1, \dots, b_m)$ to be the moduli space of the $J$-holomorphic disks:
$$u: (D_{m+1}, \partial D_{m+1}) \to (\bbR\times \bbR^3, \Sigma)$$
such that
\begin{itemize}
\item $D_{m+1}$ is a $2$-dimensional unit disk with $m+1$ boundary points $p, q_1, \dots , q_m$ removed 
and the points $p, q_1, \dots , q_m$ are arranged  in a counterclockwise order;
 \item  $u$ is asymptotic to $[N,  \infty) \times a$ at $p$;
 \item  $u$ is asymptotic to $( -\infty, -N] \times b_i$ at $q_i$.
\end{itemize} 

When $\dim \calM(a; b_1, \dots, b_m)=0$, the disk $u \in \calM(a; b_1, \dots, b_m)$ is called \textbf{rigid}.
The gradings of corresponding Reeb chords satisfy
$$|a| - |b_1|- \cdot \cdot \cdot -|b_m|=0. $$
For the image of the boundary segment from $q_i$ to $q_{i+1}$, one can close up in a similar way as the one in the definition of the DGA differential and take the homology class in $H_1(\Sigma)$, denoted by $\tau_i$.
If $\Sigma$ is spin, all the relevant moduli spaces of $J$-holomorphic disks admit a coherent orientation.
In particular, each rigid $J$-holomorphic disk obtains a sign, denoted by $s(u)$.
Associate a monomial $w(u)$ to the $J$-holomorphic disk $u$ as
$$w(u) = s(u)\tau_0 b_1 \tau_1 \cdots b_{m} \tau_{m}.$$
The homology classes $\tau_i$ for $i=1,\dots, m$, are called the \textbf{coefficients} of $w(u)$.
The DGA map is defined by counting rigid $J$-holomorphic disks with boundary on $\Sigma$:
$$\phi (a) = \displaystyle{\sum_{\dim \calM(a; b_1, \dots, b_m)=0}\  \sum_{u \in \calM(a; b_1, \dots, b_m)}  w(u)}.$$
We can extend  the morphism to $\calA(\Lambda_+; \bbF[H_1(\Sigma)])$ by setting $\phi(t) =t$ for any generator $t$ in $H_1(\Sigma)$ and applying the Leibniz rule.

In order to modify the coefficients of the DGA map $\phi$,
let us consider $H_1(\Sigma)$ more precisely. 
To simplify the description, we restrict $\Sigma$ to $[-N, N]\times \bbR^3$, denote by $\Sigma$ as well.
According to Poincar{\'e} duality, $H^1(\Sigma) \cong H_1(\Sigma, \Lambda_+ \cup \Lambda_-)$.
In particular,
for any loop $\alpha$ in $\Sigma$ with ends on $ \Lambda_+ \cup \Lambda_-$, which is an element in $H_1(\Sigma, \Lambda_+ \cup \Lambda_-)$, 
there is an element $\theta_{\alpha}$ in $H^1(\Sigma)$
such that for any oriented loop $\gamma$ on $\Sigma$, the intersecting number of $\alpha$ and $\gamma$ is $\theta_{\alpha}(\gamma)$.
Thus, in order to know the homology class of a curve $\gamma$ in $H_1(\Sigma)$, 
we only need to count the intersection number of each generator curve of $H_1(\Sigma, \Lambda_+ \cup \Lambda_-)$ with $\gamma$.

\begin{figure}[!ht]
\labellist
\small
\pinlabel $\Lambda_+$  at 150 140
\pinlabel $\Lambda_-$  at 160 25
\pinlabel $\Sigma$  at 150 85
\pinlabel $\ast_+$  at 50 135
\pinlabel $\ast_-$  at 40 0
\pinlabel $\alpha$  at 52 80
\endlabellist
\includegraphics[width=1.5in]{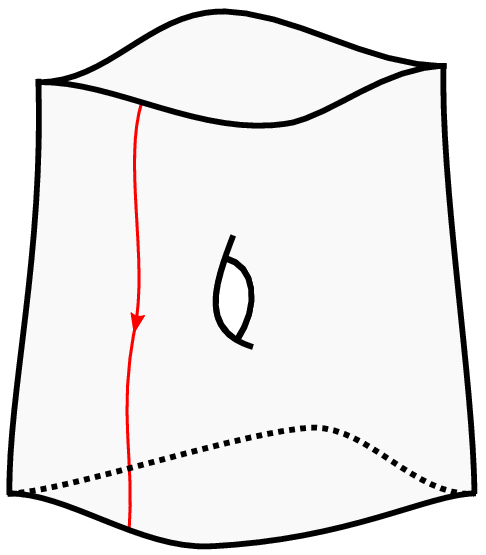}
\caption{Curve $\alpha$ on a cobordism.}
\label{cob}
\end{figure}

Consider  a connected exact Lagrangian cobordism $\Sigma$ from a Legendrian knot $\Lambda_-$ to a Legendrian knot $\Lambda_{+}$ (see Remark \ref{connect} for the reason that we assume that $\Sigma$ is connected).
Choose base points $\ast_+$ and $\ast_-$ for $\Lambda_+$ and $\Lambda_-$, respectively.
There exists a curve $\alpha$ on $\Sigma$ from $\ast_+$ to $\ast_-$ 
with exactly one intersection with $\Lambda_+$ and $\Lambda_-$, respectively.
An example is shown in Figure \ref{cob}.
Let $V^*$ denote the subgroup of $H^1(\Sigma)$ that is generated by the Poincar{\'e} dual of curve  $\alpha$.
The dual space $V$ in $H_1(\Sigma)$ is isomorphic to $\bbZ$.

Now we can modify the DGA map $\phi$ described above to be a map from $\calA(\Lambda_+; \bbF[V])$ to $\calA(\Lambda_-; \bbF[V])$.
First, restrict the generators of $\calA(\Lambda_{\pm})$ to Reeb chords of $\Lambda_{\pm}$ and a basis of $V$.
Second,
project the coefficients $\tau_i$ of the monomial $w(u)$ from $H_1(\Sigma)$ to $V$.
Therefore, the DGA map works in $\bbF[V]$ coefficients.
Indeed, the definitions of $\calA(\Lambda_{\pm}; \bbF[V])$ match the definition of $\calA(\Lambda_{\pm}, \ast_{\pm})$, respectively.
Hence a connected exact Lagrangian cobordism $\Sigma$ induces a DGA map with $\bbF[V]$ coefficients from the DGA of $\Lambda_+$ with a single base point to the DGA of $\Lambda_-$ with a single base point: 
$$\phi: (\calA(\Lambda_{+}, \ast_{+}), \partial)\to  (\calA(\Lambda_{-}, \ast_{-}), \partial).$$

This DGA map does depend on the choice of the curve $\alpha$ that connecting the two base points.

\section{The Augmentation Category}\label{aug}

\subsection{$A_{\infty}$ categories}\label{Acat}
In this section, we give a lightning review of $A_{\infty}$ algebras and $A_{\infty}$ categories  following \cite{NRSSZ}. 
See \cite{keller, GJ} for a more detailed introduction.
\begin{defn}[{\cite[Section 3.1]{keller}}]
An {\bf $A_{\infty}$ algebra} over a field $\bbF$ is a $\bbZ$-graded vector space $A$ endowed with degree $2-n$ maps $m_n: A^{\otimes n} \to A$ such that
$$\displaystyle{\sum_{r+s+t=n}(-1)^{r+st} m_{r+1+t}(1^{\otimes r} \otimes m_s \otimes 1^{\otimes t})=0}.$$
\end{defn}

The most important things we need among these complicated relations are
\begin{itemize}
\item $m_1$ is a differential on $A$ (i.e., $m^2_1=0$) and
\item $m_2$ is associative after passing to the homology with respect to $m_1$. 
\end{itemize}

An $A_{\infty}$ algebra can be achieved nicely through the following  construction.
Let $\overline{T}(C) =\displaystyle{ \bigoplus_{n \ge 1} C^{\otimes n}}$ be a graded vector space over $\bbF$ equipped with a {\bf co-differential} $b$,  i.e.,
\begin{itemize}
\item b has degree $1$, $b^2=0$ and
\item $b= \oplus b_n$, where $b_n$ is a map  $C^{\otimes n}  \to C$, satisfies the co-Leibniz rule
$$\Delta b = (1\otimes b + b \otimes 1) \Delta,$$
where $\Delta(a_1 \otimes\cdot \cdot\cdot \otimes a_n)=\displaystyle{\sum_{i=1}^n}(a_1 \otimes\cdots\otimes a_i)\otimes(a_{i+1} \otimes \cdots \otimes a_n)$.
\end{itemize}

Let $C^{\vee}:= C[-1]$ and $s: C \to C^{\vee}$ be the canonical degree $1$ identification map $a \mapsto a$.
Define maps $m_n: (C^{\vee})^{\otimes n} \to C^{\vee}$ such that the following diagram commutes for all $n$.

\diag{C^{\otimes n} \arr^{b_n}\ard_{s^{\otimes n}} & C\ard^{s}\\
(C^{\vee})^{\otimes n} \arr_{m_n} & C^{\vee}}
Then $C^{\vee}$ is an $A_{\infty}$ algebra with $m_n$ as $A_{\infty}$ operations \cite{Sta}.
One can check that the degree of $m_n$ is $2-n$.

\begin{eg}\label{mn}
If a Legendrian contact homology DGA $(\calA(\Lambda), \partial)$ has an augmentation $\epsilon$,
the conjugated differential
$\partial^{\epsilon}$ is a differential of $\calA^{\epsilon}_+= \displaystyle{\bigoplus_{n \ge 1} C^{\otimes n}= \overline{T}(C)}$,
where $C$ is the vector space over a field $\bbF$ generated by Reeb chords of $\Lambda$.
We define $\delta^{\epsilon}$ to be the adjoint of $\partial^{\epsilon}$ on $\overline{T}(C^*)=\displaystyle{\bigoplus_{n \ge 1} (C^*)^{\otimes n}}$, where
$C^*$ is the dual of $C$.
More specifically, 
$$\delta^{\epsilon}(b_m^* \otimes \cdots \otimes b_1^*) = \displaystyle{\sum_{a} \textrm{Coefficient}_{b_1b_2\cdots b_m} (\partial^{\epsilon}(a))}.$$
It is not hard to check that $\delta^{\epsilon}$ is a co-differential of $\overline{T}(C^*)$. 
Hence one can use the construction above to construct an $A_{\infty}$ algebra $(C^*)^{\vee}$.
\end{eg}

\begin{defn}\cite{CKESW}
An {\bf $A_{\infty}$ category} over a field $\bbF$ is a category where, for any two objects $\epsilon_1$ and  $\epsilon_2$, 
the morphism is a graded vector space $Hom(\epsilon_1, \epsilon_2)$. 
Moreover, for any objects $\epsilon_1, \epsilon_2, \dots , \epsilon_{n+1}$, there exists a degree $2-n$ map
$$m_n: Hom(\epsilon_n,\epsilon_{n+1}) \otimes \cdots \otimes Hom(\epsilon_1,\epsilon_{2}) \to Hom(\epsilon_1,\epsilon_{n+1})$$
satisfying
$$\displaystyle{\sum_{r+s+t=n}(-1)^{r+st} m_{r+1+t}(1^{\otimes r} \otimes m_s \otimes 1^{\otimes t})=0}.$$
\end{defn}

As noticed before, the first $A_{\infty}$ operation $m_1$ is a differential for $Hom(\epsilon_1, \epsilon_2)$ with degree $1$. 
Denote its cohomology by $H^*Hom(\epsilon_1, \epsilon_2)$.
Moreover, we have that $m_2$ descends to an associative map on the cohomology level:
$$m_2: H^*Hom(\epsilon_2, \epsilon_3) \otimes H^*Hom(\epsilon_1, \epsilon_2) \to H^*Hom(\epsilon_1, \epsilon_3)$$
for any objects $\epsilon_1, \epsilon_2, \epsilon_3$.

An {\bf $A_{\infty}$ morphism} between two $A_{\infty}$ categories $f: \calA \to \calB$ maps
the object $\epsilon$ of $\calA$ to $f(\epsilon)$ of $\calB$ and
 for any objects $\epsilon_1, \epsilon_2, \dots , \epsilon_{n+1}$ of $\calA$, 
there exists a map 
$$f_{n}: Hom(\epsilon_n,\epsilon_{n+1}) \otimes \cdot \cdot \cdot \otimes Hom(\epsilon_1,\epsilon_{2}) \to Hom(f(\epsilon_1),f(\epsilon_{n+1}))$$
satisfying the $A_{\infty}$ relations \cite{keller}.
In particular, the first map $f_1$, called {\bf the category map on the level of morphisms}, maps
the morphism $Hom(\epsilon_1, \epsilon_2)$ of $\calA$ to the morphism $Hom(f(\epsilon_1) , f(\epsilon_2))$ of $\calB$.
From the  $A_{\infty}$ relations, we know that
\begin{itemize}
\item the functor $f_{1}$, the category map on the level of morphisms, commutes with $m_1$ and thus $f_{1}$ descends to a map on cohomology:
$$f^*: H^*Hom(\epsilon_1, \epsilon_2) \to H^*Hom(f(\epsilon_1), f(\epsilon_2));$$
\item for any $a\in Hom(\epsilon_2, \epsilon_3) $ and $b\in Hom(\epsilon_1, \epsilon_2)$,
we have $$f^*(m_2([a], [b])) = m_2(f^*[a], f^*[b]),$$
i.e, the composition map $m_2$ commutes with $f^*$ when passing to the cohomology level.
\end{itemize}
An  $A_{\infty}$ morphism between two $A_{\infty}$ categories $f: \calA \to \calB$ induces a functor on the cohomology categories, $\tilde{f}: H^* \calA \to H^* \calB$.
It behaves the same as $f$ on the object level.
On the level of morphisms $\tilde{f}=f^*$.
The functor $\tilde{f}$ is {\bf faithful} if $f^*$ is injective and is {\bf fully faithful} if $f^*$ is an isomorphism for any morphism in $H^*\calA$.

\subsection{The augmentation category}\label{augcat}
In this section, we briefly review the augmentation category $\calA ug_+(\Lambda)$ following \cite{NRSSZ}.

Let $\Lambda$ be an oriented Legendrian knot in $(\bbR^3, \ker \alpha)$ endowed with a single base point
$\ast$.
Denote its Legendrian contact homology DGA  by $(\calA, \partial)$.
Given a field $\bbF$,
the objects of the augmentation category $\calA ug_+(\Lambda)$ are augmentations of $(\calA, \partial)$ to $\bbF$,
$$\epsilon: \calA \to \bbF.$$
In order to describe the morphism $Hom_+(\epsilon_1, \epsilon_2)$ for any two objects $\epsilon_1$ and $\epsilon_2$, 
we need to study the DGA of a $2$-copy of $\Lambda$, denoted by $\Lambda^{(2)}$.

By the Weinstein tubular neighborhood theorem, we can identify a neighborhood of $\Lambda$ with
a neighborhood of the zero section in the $1$-jet space $J^1(\Lambda) = T^*(\Lambda)\times \bbR$ through a contactomorphism.
The contact form in $J^1(\Lambda)$ is $\alpha = dz-pdq$, 
where $q$ is the coordinate on $\Lambda$ and $p$ is the coordinate in the cotangent direction.
For any $C^1$ small function $f: \Lambda \to \bbR$, 
the $1$-jet  $j^1f=\{(q, f'(q), f(q))| \ q \in \Lambda\}$ is a Legendrian knot in $J^1(\Lambda)$
and thus is a Legendrian knot in $\bbR^3$.
Now choose a particular Morse function $f: \Lambda \to (0, \delta)$ such that
\begin{itemize}
\item $\delta$ is smaller than the minimum length of Reeb chords of $\Lambda$,
\item
the Morse function $f$ has exactly $1$ local maximum point at $x$ and $1$ local minimum point at $y$,
and 
\item
around the base point $\ast$, three points $\ast, x, y$ show up in order when traveling along the link (see Figure \ref{base}).
\end{itemize}
Decorate $j^1f$ with a base point in the same location 
and with the same orientation as $\Lambda$.
Now $\Lambda \cup j^1f$ is a $2$-copy of $\Lambda$, denoted by $\Lambda^{(2)}$. 
Label $\Lambda^{(2)}$ from top (higher $z$ coordinate) to bottom (lower $z$ coordinate) by $\Lambda^1$ and $\Lambda^2$.
An example of the $2$-copy of the trefoil with a single base point is shown in Figure \ref{trefoil}.
\begin{figure}[!ht]
\begin{minipage}{3in}
\vspace{0.3in}

\begin{center}
\labellist
\pinlabel $\ast$ at 45 5
\pinlabel $x$ at 75 15
\pinlabel $y$  at 105 25 
\endlabellist
\includegraphics[width=2in]{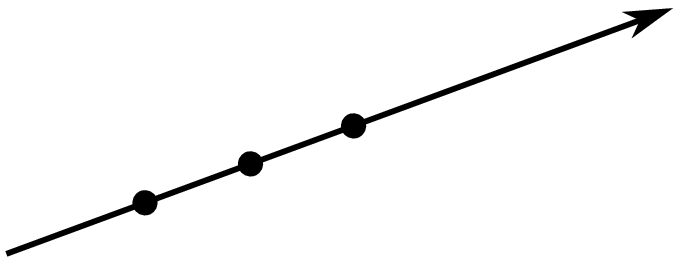}
\vspace{0.3in}
\caption{A neighborhood of the base point $\ast$ on $\Lambda$. The arrow indicates the orientation of $\Lambda$.}
\label{base}
\end{center}

\end{minipage}
\begin{minipage}{3.3in}
\begin{center}
\labellist
{\color{red}
\pinlabel $\Lambda_1$ at 110 100
}
{\color{blue}
\pinlabel $\Lambda_2$ at 110 70
}
 
\endlabellist
\includegraphics[width=2.5in]{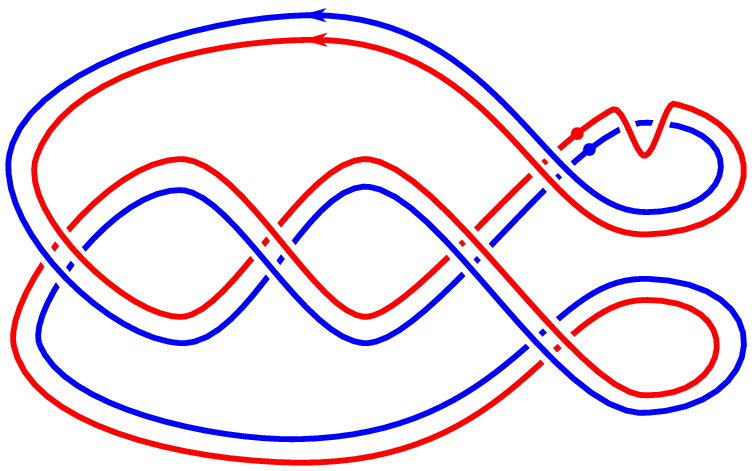}
\caption{The Lagrangian projection of a $2$-copy of the trefoil with a single base point.}
\label{trefoil}
\end{center}
\end{minipage}
\end{figure}

The Legendrian contact homology DGA $(\calA(\Lambda^{(2)}), \partial^{(2)})$ of $\Lambda^{(2)}$ can be recovered from the data carried by the DGA $(\calA(\Lambda), \partial)$ of $\Lambda$.
Recall that $\calA(\Lambda)$ is generated  the set  $\calR$ of Reeb chords $ \{a_1, \dots , a_m\}$ and    the set $\calT= \{t, t^{-1}\}$ that corresponds to the base point  as stated in Section \ref{DGA}.
Similarly, divide the set of generators of $\calA(\Lambda^{(2)})$ into two parts $\calR^{(2)} \cup\calT^{(2)}$.
It is obvious that  $\Lambda^{(2)}$ has two base points and  thus
write $\calT^{(2)}$ as $\{(t^1)^{\pm 1}, (t^2)^{\pm1}\}$.
As for the set of Reeb chords $\calR^{(2)}$,
we divide it into four parts
  $\calR^{(2)}= \displaystyle{\bigcup_{i,j = 1,2}\calR^{ij}}$,
where $\calR^{ij}$ is the set of Reeb chords to $\Lambda_i$  from $\Lambda_j$.
Observe that Reeb chords of $\Lambda^{(2)}$ come from two sources.
\begin{itemize}
\item
Each critical point $x$ or $y$ of  the Morse function $f$ gives one Reeb chord in $\calR^{12}$, 
denoted by $x^{12}$ and $y^{12}$ respectively.
We call this type of Reeb chords  {\bf Morse Reeb chords}.
\item
Each Reeb chord $a_l$ of $\Lambda$ gives four Reeb chords of $\Lambda^{(2)}$, 
denoted by $a_l^{ij}  \in \calR^{ij}$, where  $i,j= 1,2$ and $ l=1, \dots, m$.
We call these Reeb chords {\bf non-Morse Reeb chords}.

\end{itemize}

It is obvious that $a^{ii}$ and $t^{i}$, for  $i =1,2$, inherit the grading from $a$ and $t$ in $\calA(\Lambda)$ respectively. 
We can choose a family of capping paths such that $|a^{ij}|=
|a|$ for any Reeb chord $a$ of $\Lambda$.
Under this choice of capping paths $\gamma$,
one can show that $CZ(\gamma_{x^{12}})= Ind_f(x)$ for any Morse Reeb chord $x^{12}$
through a similar computation as in \cite{EESduality}.
Hence we have $|x^{12}|=0$ and $|y^{12}|=-1$.

In order to describe the differential $\partial^{(2)}$, we encode the generators in matrices.
Let $A_l$, $X_k,\ Y_k,\ \Delta_k$, for  $1\leq l\leq m$, be $2 \times 2$ matrices:
$$A_l= 
\begin{pmatrix}
a_l^{11}& a_l^{12}\\
a_l^{21}& a_l^{22}\\
\end{pmatrix},
\hspace{0.2in}
X=\begin{pmatrix}
1 & x^{12}\\
0 & 1\\
\end{pmatrix},
\hspace{0.2in}
Y=\begin{pmatrix}
0 & y^{12}\\
0 & 0\\
\end{pmatrix},
\hspace{0.2in}
\Delta=\begin{pmatrix}
t^1 & 0\\
0 & t^2\\
\end{pmatrix}.
$$
The differential $\partial^{(2)}$ is defined on generators as follows by applying entry-by-entry to these matrices:
$$
\begin{array}{rcl}
\vspace{0.1in}
\partial^{(2)}A_l & = &\Phi(\partial a_l)+ YA_l- (-1)^{|a_l|}A_lY\\
\vspace{0.1in}
\partial^{(2)}X & = &\Delta^{-1}Y\Delta X - XY\\
\vspace{0.1in}
\partial^{(2)}Y & = &Y^2\\

\partial^{(2)}\Delta & = &0,\\
\end{array} $$
where $\Phi: \calA \to Mat(2, \calA(\Lambda^{(2)}))$ is a ring homomorphism given by
$\Phi(a_l)=A_l$ and $\Phi(t)=\Delta X$.

Given two augmentations $\epsilon^1$ and $\epsilon^2$ of $(\calA, \partial)$, 
we get an augmentation $\epsilon$
of $(\Lambda^{(2)}, \partial^{(2)})$ by sending 
$a_l^{ii} \mapsto \epsilon^i(a_l)$,
 $t^{ii}\mapsto \epsilon^i(t)$
 and everything else to $0$.
Thus $\partial^{(2)}_{\epsilon}=\phi_{\epsilon}\circ \partial^{(2)} \circ  \phi_{\epsilon}^{-1}$
is a differential of $\calA^{(2)} = \calA(\Lambda^{(2)})/ (t^{ii}=\epsilon(t^{ii}))$.
Both the morphism $Hom_+(\epsilon_1, \epsilon_2)$ and the first $A_{\infty}$ operation $m_1$ are  defined from $(\calA^{(2)}, \partial_{\epsilon}^{(2)})$ through the construction stated in Section \ref{Acat}.
For $i,j=1,2$, let $C^{ij}$ denote the free graded  $\bbF$ algebra generated by $\calR^{ij}$, 
which is a sub algebra of $\calA^{(2)}$.
Notice that $C^{12}$ and $C^{21}$ are closed under $\partial^{(2)}_{\epsilon}$\vspace{0.01in}
since $\epsilon$ vanishes on the components in $C^{11}$ and $C^{22}$ of the image of $\partial^{(2)}_{\epsilon}$.
Hence $C^{12}$ and $C^{21}$ are sub chain complexes of $(\calA^{(2)}, \partial_{\epsilon}^{(2)})$.
Define the morphism $Hom_+(\epsilon_1,\epsilon_2)$ between objects $\epsilon_1$ and $\epsilon_2$ to be $(C^{12})^{\vee}$.
To simplify the notation, we write $(a_l^{12})^{\vee}$ as $a_l^{\vee}$, $(x^{12})^{\vee}$ as $x^{\vee}$
and $(y^{12})^{\vee}$ as $y^{\vee}$. 
Therefore, their gradings satisfy $|a_l^{\vee}|=|a_l|+1$, $|x^{\vee}|=1$ and $|y^{\vee}|=0$. 
The first $A_{\infty}$ operation $m_1$ is defined by the adjoint of $\partial^{(2)}_{\epsilon}$, i.e.,
for any Reeb chord $c\in \calR$, 
$$m_1(c^{\vee})=\displaystyle{\sum_{a\in{\calR}} \textrm{Coefficient}_c (\partial^{(2)}_{\epsilon} a)a^{\vee}}.$$ 
As we noted before, $m_1$ is a differential for $Hom_+(\epsilon_1, \epsilon_2)$. 
The corresponding cohomology is denoted by $H^*Hom_+(\epsilon_1, \epsilon_2)$.
Similarly, define $(C^{21})^{\vee}$ to be $Hom_-(\epsilon_2, \epsilon_1)$.
Take the cohomology of $Hom_-(\epsilon_2, \epsilon_1)$ with respect to $m_1$,
 denoted by $H^*Hom_-(\epsilon_2, \epsilon_1)$.

\begin{rmk}
One may find the convention of $Hom_-(\epsilon_2, \epsilon_1)$ not natural.
However, the notations are consistent in the sense that both $Hom_{+}(\epsilon, \epsilon')$ and 
 $Hom_{-}(\epsilon, \epsilon')$ are generated by Reeb chords from the component with the augmentation $\epsilon'$ to the component with the augmentation $\epsilon$.  
\end{rmk}

The $Hom_+(\epsilon_1, \epsilon_2)$ space and the $Hom_-(\epsilon_1, \epsilon_2)$ space are closely related.
Recall that the generators of $Hom_+(\epsilon_1, \epsilon_2)$  naturally correspond to the Reeb chords in $\calR^{12}$, which consist of non-Morse Reeb chords and Morse Reeb chords.
Note that the lengths of Morse Reeb chords are smaller than the lengths of non-Morse Reeb chords.
Due to  the positive energy constraint,
there does not exist
any holomorphic disk that has a positive puncture at a Morse Reeb chord and a negative puncture at a non-Morse Reeb chord.
Therefore,
the graded sub-vector space of $Hom_+(\epsilon_1, \epsilon_2)$  generated by non-Morse Reeb chords is closed under $m_1$,
and thus is a sub-chain complex.
Indeed, this sub-chain complex agrees with $(Hom_-(\epsilon_1, \epsilon_2), m_1)$.
From \cite{Caitlin}, for a Legendrian knot $\Lambda$ with a single base point, any two augmentations $\epsilon_1,$ and $\epsilon_2$ agree on the generator $t$ that corresponds to the base point.
As a result, by  \cite[Proposition 5.2]{NRSSZ},
 the quotient chain complex  that is generated by $\{x^{\vee}, y^{\vee} \}$ is the Morse co-chain complex induced by the Morse function $f$.
Therefore following long exact sequence holds:
\begin{equation}\label{lesknot}
{\xymatrixcolsep{1pc}
\xymatrix{
\cdots \arr & H^{i-1}(\Lambda) \arr & H^iHom_-(\epsilon_1,\epsilon_2) \arr & H^iHom_+(\epsilon_1,\epsilon_2) \arr & H^{i}(\Lambda) \arr & \cdots }}.
\end{equation}
Furthermore, given that both $Hom_+(\epsilon_1,\epsilon_2)$ and $Hom_-(\epsilon_1,\epsilon_2)$ are vector spaces over the field $\bbF$,
combining the Universal Coefficient Theorem with Sabloff Duality in \cite[Section 5.1.2]{NRSSZ},
we have
\begin{equation}\label{sabdual}
H^k Hom_-(\epsilon_1,\epsilon_2) \cong H^{-k}(Hom_-(\epsilon_1, \epsilon_2)^{\dagger})
\cong H^{2-k} Hom_+(\epsilon_2, \epsilon_1).
\end{equation}
For a chain complex $C$,
the chain complex $C^{\dagger}$ is obtained by 
dualizing the underlying vector space and differential of $C$ and then negating the gradings.

For the other $A_{\infty}$ operators $m_n$, one needs to consider an $n$-copy of $\Lambda$, 
denoted by $\Lambda^{(n)}$.
Construct a  DGA $(\calA^{n}, \partial^{(n)}_{\epsilon})$ of $\Lambda^{(n)}$
that is analogous to $(\calA^{2}, \partial^{(2)}_{\epsilon})$. 
Define $m_n$ to be the adjoint of $\partial^{n}_{\epsilon}$ as in Example \ref{mn}.
See \cite{NRSSZ} for more details.

By \cite{NRSSZ}, the augmentation category described above
does not depend on the choice of the Morse function $f$.
Moreover,  up to $A_{\infty}$ category equivalence,
the augmentation category is invariant of Legendrian knot  under Legendrian isotopy.

A key property of $\calA ug_+(\Lambda)$ is that $\calA ug_+(\Lambda)$ is a {\bf strictly unital} $A_{\infty}$ category,
 with the {\bf units} given by
$$e_{\epsilon}=  -y^{\vee} \in Hom_+(\epsilon,\epsilon),$$
i.e.,
\begin{itemize}
\item $m_1(e_{\epsilon})=0$;
\item for any $\epsilon_1, \epsilon_2$ and any $c\in Hom_+(\epsilon_1, \epsilon_2)$, 
$m_2(c, e_{\epsilon_1}) = m_2(e_{\epsilon_2}, c)=c$;
\item any higher composition involving $e_{\epsilon}$ is $0$.
\end{itemize}

As a result, the corresponding cohomology category $H^*\calA ug_+(\Lambda)$ is a unital category,
which makes it natural to talk about the equivalence relation of objects in $\calA ug_+(\Lambda)$.
\begin{defn}  
Two objects $\epsilon_1$ and $\epsilon_2$ are {\bf equivalent} in $\calA ug_+(\Lambda)$ if they are 
isomorphic in $H^* \calA ug_+(\Lambda)$,
i.e. there exist $[\alpha] \in H^0 Hom_+(\epsilon_1, \epsilon_2)$ and $[\beta] \in H^0 Hom_+(\epsilon_2, \epsilon_1)$
such that $m_2([\alpha],[\beta])=[e_{\epsilon_2}] \in H^0 Hom_+(\epsilon_2, \epsilon_2)$ and 
$m_2([\beta],[\alpha])=[e_{\epsilon_1}] \in H^0Hom_+(\epsilon_1, \epsilon_1)$.

\vspace{.1in}
\begin{center}
\begin{minipage}{5in}
\xymatrixcolsep{3pc}
\xymatrix{\epsilon_1 \ar@/_/[rr]_{\beta} \ar@(ul,dl)_{e_{\epsilon_1}} & & \epsilon_2\ar@/_/[ll]_{\alpha} \ar@(ur,dr)^{e_{\epsilon_2}}}
\end{minipage}
\end{center}

\end{defn}
By \cite{NRSSZ}, for a Legendrian knot with a single base point, two augmentations are equivalent if and only if they are isomorphic as DGA maps.

Suppose $\Sigma$ is a connected exact Lagrangian cobordism from a Legendrian knot $\Lambda_-$ to a Legendrian knot $\Lambda_+$.
It induces a DGA map $\phi$ from the DGA $(\calA(\Lambda_+), \partial)$ with a single base point  to a DGA $(\calA(\Lambda_-), \partial)$ with a single base point.
By \cite[Proposition 3.29]{NRSSZ},
this $DGA$ map $\phi$ induces a unital $A_{\infty}$ category morphism $f$ from $\calA ug_+(\Lambda_-)$ to 
$\calA ug_+(\Lambda_+)$.
The category map sends an augmentation $\epsilon_-$ of $\Lambda_-$ to $\epsilon_+=  \epsilon_-\circ \phi$, which is an augmentation of $\Lambda_+$.
The family of maps $\{f_n\}$ is constructed through a family of $DGA$ morphisms of $n$-copies:
$$
\begin{array}{rcll}
f^{(n)}: (\calA^{(n)}(\Lambda_+), \partial^{(n)}) & \mapsto & (\calA^{(n)}(\Lambda_-), \partial^{(n)}) \\
\Delta & \mapsto & \Delta &\\
Y & \mapsto & Y &\\
X & \mapsto & \Delta^{-1} \cdot \Phi_- \circ f(t)\\
\Phi_+(a) & \mapsto & \Phi_-\circ f(a), & a\in\calA(\Lambda_+).\\
\end{array}
$$
Let $\epsilon_-$ be the augmentation of $(\calA^{(n)}(\Lambda_+)), \partial^{(n)})$ 
that sends $a^{ii}_l \mapsto \epsilon_-^i(a_l)$, $t^{ii}\mapsto \epsilon_-^i(t)$ and everything else to $0$.
Define the map $f_{n-1}$ to be the adjoint of $f_{\epsilon_-}^{(n)}$, where
$$f_{\epsilon_-}^{(n)} = \phi_{\epsilon_-} \circ f^{(n)} \circ \phi_{\epsilon_-}^{-1}.$$
In particular,
$f_1$ can be written as 
\begin{equation}\label{f_1}
\begin{array}{rcll}
f_1: Hom_+(\epsilon_-^1, \epsilon_-^2)& \to & Hom_+(\epsilon_+^1, \epsilon_+^2) & \vspace{.05in}\\\
y_-^{\vee} & \mapsto & y_+^{\vee}& \vspace{.05in}\\
c^{\vee} & \mapsto & \displaystyle{ \sum_{a \in \calA(\Lambda_+)} \textrm{Coeff}_c(f^{(2)}_{\epsilon_-}(a)) a^{\vee}}, & c\in\calA(\Lambda_-)\vspace{.05in}\\\
x_-^{\vee} & \mapsto & x_+^{\vee} + \displaystyle{ \sum_{a \in \calA(\Lambda_+)} \textrm{Coeff}_t(f^{(2)}_{\epsilon_-}(a)) a^{\vee}} .& \\
\end{array}
\end{equation}
When computing $\textrm{Coeff}_b(f^{(2)}_{\epsilon_-}(a))$, where $b$ is either a Reeb chord $c \in \calA(\Lambda_-)$ or $t\in \calT$,
one consider all the terms of  $f(a)$ including $b$.
If a term of $f(a)$ including $b$ can be written as  ${\bf p} b{\bf q}$, where $\bf p$ and $\bf q$ are words of pure Reeb chords of $\Lambda_-$,
this term contributes $\textrm{Coeff}_{{\bf p} b {\bf q}}(f(a)) \epsilon_-^1({\bf p}) \epsilon^2_-({\bf q})$ to 
$\textrm{Coeff}_b(f^{(2)}_{\epsilon_-}(a))$.
Therefore 
we have
$$\textrm{Coeff}_b(f^{(2)}_{\epsilon_-}(a))= \sum_{\bf p \  q} \textrm{Coeff}_{{\bf p} b {\bf q}}(f(a)) \epsilon_-^1({\bf p}) \epsilon^2_-({\bf q}).$$

\begin{rmk}
According to \cite[Proposition 3.29]{NRSSZ}, the condition for a DGA map to induce a unital $A_{\infty}$ category morphism is that the DGA map is compatible with the weak link gradings in the sense of \cite[Definition 3.19]{NRSSZ}. 
In our case where
 both $\Lambda_+$ and $\Lambda_-$ are single component Legendrian knots with a single base point,
this condition is trivially satisfied.
\end{rmk}

\vspace{0.3in}

\section{Floer theory for Lagrangian Cobordisms}\label{FT}
In this section, we  give a brief introduction to the Floer theory of a pair of exact Lagrangian cobordisms following \cite{CDGG}.
Let $\Sigma^i$, for $i= 1,2$, be exact Lagrangian cobordisms from $\Lambda^i_{-}$ to $\Lambda^i_{+}$  in 
$\big(\bbR \times \bbR^3, d(e^t \alpha)\big)$, where $ \alpha= dz-ydx.$
A schematic picture is shown in Figure \ref{paircob}.
The union of the cobordisms $\Sigma^1\cup \Sigma^2$ is cylindrical over 
$\Lambda^1_+ \cup \Lambda^2_+ $ (resp. $\Lambda^1_- \cup \Lambda^2_-$) on the positive end (resp. negative end).
If we view $\Sigma^1 \cup \Sigma^2$ as a Lagrangian cobordism from 
the Legendrian link $\Lambda_-^1 \cup \Lambda_-^2$ to the Legendrian link $\Lambda_+^1 \cup \Lambda_+^2$,
we obtain a chain complex generated by Reeb chords of $\Lambda_-^1 \cup \Lambda_-^2$ 
and $\Lambda_+^1 \cup \Lambda_+^2$.
On the other hand,
if we lift the exact Lagrangian cobordism $\Sigma^1\cup\Sigma^2$ to be a Legendrian manifold in $\bbR \times \bbR^3 \times \bbR$,
we have its Legendrian contact homology DGA,
which is generated by double points of $\Sigma^1\cup \Sigma^2$.
One can construct the Cthulhu chain complex $Cth(\Sigma^1, \Sigma^2)$ as a mix of the two chain complexes above.
It is generated by some Reeb chords on the cylindrical ends and intersection points of $\Sigma^1$ and $\Sigma^2$. 
Moreover, this chain complex has trivial cohomology,
i.e. $H^*Cth(\Sigma^1, \Sigma^2)=0$.

\begin{figure}[!ht]
\labellist
\small
{\color{blue}
\pinlabel $\Lambda^2_-$ at 120 100
\pinlabel $\Lambda^2_+$ at 120 300
\pinlabel $\Sigma^2$ at 100 200
}
{\color{red}
\pinlabel $\Lambda^1_-$ at 350 140
\pinlabel $\Lambda^1_+$ at 350 320
\pinlabel $\Sigma^1$ at 370 230
}
{\color{black}
\pinlabel $t$ at -5 380
\pinlabel $N$ at -5 260 
\pinlabel $-N$ at -15 60
}
\endlabellist
\includegraphics[width=3in]{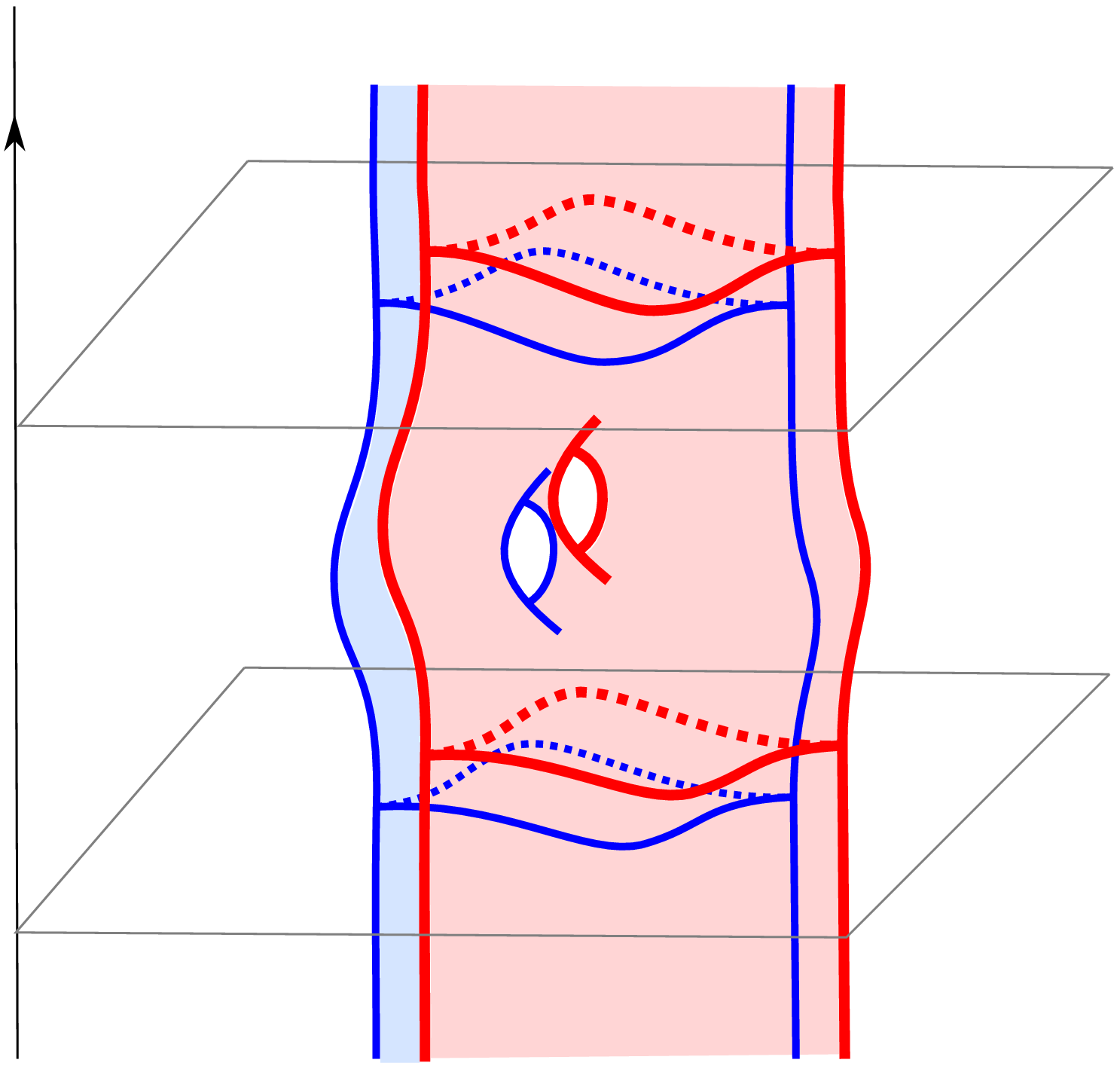}
\caption{Pair of Lagrangian cobordisms in $\big(\bbR\times \bbR^3, d(e^t \alpha)\big)$.}
\label{paircob}
\end{figure}

For the simplicity in defining grading, we  assume that $\Sigma^i$, for $i= 1,2$,  has trivial Maslov number throughout this paper.

\subsection{The graded vector space} \label{GVS}
Assume that both $\Sigma^1$ and $\Sigma^2$ are cylindrical outside of $[-N, N] \times \bbR^3$, where $N$ is a positive number.
The underlying vector space
is a direct sum of three parts: 
$$Cth(\Sigma^1, \Sigma^2) =C(\Lambda_+^1, \Lambda_+^2) \oplus CF(\Sigma^1, \Sigma^2) \oplus C(\Lambda_-^1, \Lambda_-^2). $$
The top level
$C(\Lambda_{+}^1, \Lambda_{+}^2)$  (resp. bottom level $C(\Lambda_{-}^1, \Lambda_{-}^2)$)  is an $\bbF$-module generated by Reeb chords to $\Lambda^1_{+}$ (resp. $\Lambda^1_{-}$) from $\Lambda^2_{+}$ (resp. $\Lambda^2_{-}$) that are lying on the slice of $t=N$ (resp. $t=-N$).  
The middle level
$CF(\Sigma^1, \Sigma^2)$ is an $\bbF$-module generated by intersection points of  $\Sigma^1$ and $\Sigma^2$, which are all contained in $(-N, N)\times \bbR^3$.

{\bf{Grading.}}
To define the degree, first fix a capping path $\gamma_c$ for each generator $c$. 
For a Reeb chord $a$ in $C(\Lambda_+^1, \Lambda_+^2)$ or $C(\Lambda_-^1, \Lambda_-^2)$, define the degree $$|a| = CZ(\gamma_a) -1,$$ 
which matches the definition of degree when viewing $a$ as a generator in the Legendrian contact homology DGA of $\Lambda_+^1 \cup \Lambda_+^2$ or $\Lambda_-^1 \cup \Lambda_-^2$.
For an intersection point $x \in CF(\Sigma^1, \Sigma^2)$, define the degree $$|x|= CZ(\gamma_x),$$ following \cite{Seidel}.
One can also see \cite[Section 4.2]{CDGG} for details. 
Note that for a Reeb chord in $C(\Lambda_+^1, \Lambda_+^2)$, its degree in $Cth(\Sigma^1, \Sigma^2)$ will not be necessarily coincide with its degree in $C(\Lambda_+^1, \Lambda_+^2)$.
It is shifted as we will see later.

{\bf{ Action.}}
For $i=1,2$, suppose $g_i$ is a  primitive of the exact Lagrangian cobordism $\Sigma^i$,
and
hence  $g_i$ is constant when $t<-N$ or $t>N$.
Note that primitive functions are  well defined up to a overall shift by a constant number.
Thus we may assume that the primitives $g_i$ are both zero on $\Sigma^i \cup \big((-\infty, -N) \times \bbR^3\big)$ for $i=1,2$.
The action of generators is defined under this choice of primitives.

For Reeb chords $a^{+} \in C(\Lambda_{+}^1,\Lambda_{+}^2 )$  and $a^{-} \in  C(\Lambda_{-}^1,\Lambda_{-}^2 )$, 
define the {\bf action}  $\fraka$ by 
$$\fraka(a^+)=g_2(a^+)-g_1(a^+)+ \displaystyle{\int_{a^+} e^N\alpha}$$ and
$$\fraka(a^-)=g_2(a^-)-g_1(a^-)+ \displaystyle{\int_{a^-} e^{-N}\alpha}=\displaystyle{\int_{a^-} e^{-N}\alpha} .$$
The last part is due to the special choice of primitives.
For double points $x$ of $\Sigma^1  \cup  \Sigma^2$, the action $\fraka(x)$ is defined by $\fraka(x)= g_2(x)-g_1(x)$.

\subsection{The  differential}\label{diff}

\begin{rmk}\label{Assumption}
Throughout this paper, we restrict ourselves to the case where all intersection generators have positive actions since that is the case
 for the special pair of cobordisms constructed in Section \ref{Pair}.
In general, the differential could include one more map from $CF(\Sigma^1, \Sigma^2)$ to $C(\Lambda_-^1, \Lambda_-^2)$, which is called the Nessie map.
However, by \cite[Proposition 9.1]{CDGG}, the positive energy condition of the holomorphic disks counted by the Nessie map requires the corresponding intersections in $CF(\Sigma^1, \Sigma^2)$ to have negative actions.
Therefore, in our special case, we can exclude the Nessie map and get the differential as a upper triangle as below.

\end{rmk}

	With the assumption in Remark \ref{Assumption},
	we define the differential under the decomposition 
	$$Cth(\Sigma^1, \Sigma^2) =C(\Lambda_+^1, \Lambda_+^2) \oplus CF(\Sigma^1, \Sigma^2) \oplus C(\Lambda_-^1, \Lambda_-^2) $$
by a degree $1$ map of the form
	$$d= 
	\begin{pmatrix}
	d_{++}& d_{+0} & d_{+-}\\
	0 & d_{00} & d_{0-}\\
	0 & 0 & d_{--}\\
	\end{pmatrix}.
	$$	

To describe the differential explicitly, we need to study the  holomorphic disks with boundary on $\Sigma^1 \cup \Sigma^2$.
Fix a generic domain dependent almost complex structure $J$ that is compatible with the symplectic form on $\bbR \times \bbR^3$ and the cylindrical ends in the sense of \cite[Section 3.1.5]{CDGG}.
Suppose that the induced cylindrical almost complex structure on the positive end $(\Sigma^1 \cup \Sigma^2)\cap \big([N, \infty) \times \bbR^3\big)$  (resp. the negative end  $(\Sigma^1 \cup \Sigma^2)\cap \big(( -\infty,-N] \times \bbR^3)$\big) is $J_+$ (resp. $J_-$).
The differential $d_{\pm\pm}$ of $C(\Lambda_{\pm}^1, \Lambda_{\pm}^2)$  counts  rigid $J_{\pm}$-holomorphic disks
with boundary on $\bbR \times( \Lambda^1_{\pm}\cup \Lambda^2_{\pm})$, respectively, as described in Section \ref{DGA}.
The corresponding moduli space is denoted by $\calM_{J_{\pm}}(a^{\pm}; {\bf p}^{\pm}, b^{\pm}, {\bf q^{\pm}})$, 
where $a^{\pm}$ and $b^{\pm}$ are Reeb chords to $\Lambda^1_{\pm}$ from $\Lambda^2_{\pm}$  while  ${\bf p}^{\pm}$ and ${\bf q}^{\pm}$  are words of pure Reeb chords of $\Lambda^1_{\pm}$ and $\Lambda^2_{\pm}$, respectively.
We also write $\widetilde{\calM}_{J_{\pm}}(a^{\pm}; {\bf p}^{\pm}, b^{\pm}, {\bf q^{\pm}})$ as the moduli space $\calM_{J_{\pm}}(a^{\pm}; {\bf p}^{\pm}, b^{\pm}, {\bf q^{\pm}})$ module the action of $\bbR$ in the $t$ direction.

\begin{figure}[!ht]
    \labellist
    \pinlabel $a^+$ at 70 325
    \pinlabel $a^-$ at 65 184
    \pinlabel $p_1^-$ at 20 184
    \pinlabel $q_1^-$ at 100 184
    \pinlabel $q_2^-$ at 145 184
    \pinlabel $u\in\calM(a^+;{\bf p^-},a^-,{\bf q^-})$ at 75 166
    {\color{red}
    \pinlabel $\Sigma_1$ at  15 260
     \pinlabel $\Sigma_1$ at  245 260
     \pinlabel $\Sigma_1$ at  15 80
     \pinlabel $\Sigma_1$ at  250 80
     }
     {\color{blue}
    \pinlabel $\Sigma_2$ at  135 260
      \pinlabel $\Sigma_2$ at  370 260
       \pinlabel $\Sigma_2$ at  150 80  
           \pinlabel $\Sigma_2$ at  370 80
           }  
    \pinlabel $a^+$ at 300 325
    \pinlabel $x$ at 295 205
    \pinlabel $p_1^-$ at 250  184
    \pinlabel $q_1^-$ at 340 184
    \pinlabel $q_2^-$ at 380 184
        \pinlabel $u\in\calM(a^+;{\bf p^-},x,{\bf q^-})$ at 320 166

    \pinlabel $x_1$ at 70 135
    \pinlabel $x_2$ at 70 20
    \pinlabel $p_1^-$ at 20 2
    \pinlabel $q_1^-$ at 110 2
    \pinlabel $q_2^-$ at 145 2
    \pinlabel $u\in\calM(x_1;{\bf p^-},x_2,{\bf q^-})$ at 75 -12

    \pinlabel $x$ at 295 135
    \pinlabel $a^-$ at 300 2
    \pinlabel $p_1^-$ at 250  2
    \pinlabel $q_1^-$ at 340 2
    \pinlabel $q_2^-$ at 380 2
    \pinlabel $u\in\calM(x;{\bf p^-},a^-,{\bf q^-})$ at 320 -12

    \endlabellist
	\includegraphics[width=5in]{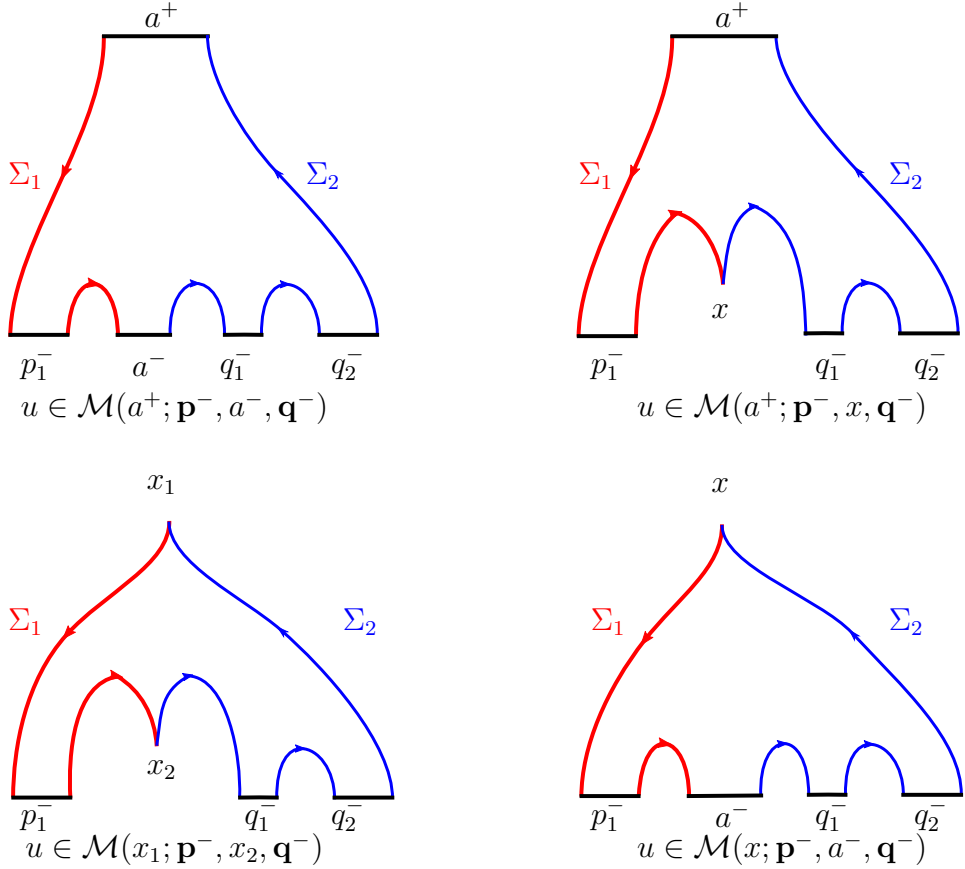}
	\vspace{0.3in}

	\caption{A sketch of the $J$-holomorphic disks in the differential $d$.
	Here $a^{\pm}$ are  Reeb chords to $\Lambda^1_{\pm}$ from $\Lambda^2_{\pm}$, respectively,
and $x$, $x_1$, $x_2$ are double points of $\Sigma^1  \cup  \Sigma^2$.  
	In these examples, $\bf p^-$ is a word of one pure Reeb chord $p^-_1$ of $\Lambda_-^1$ 
	while $\bf q^-$ is a word of two pure Reeb chords $q_1^-q^-_2$ of $\Lambda_-^2$.
	The arrows denote the orientation inherited from the boundary of the unit disk.}	
	\label{Moduli}
\end{figure}

For the remaining maps in the differential, we need to describe  $J$-holomorphic disks with boundary on $\Sigma^1\cup \Sigma^2$, where $J$ is the chosen domain dependent almost complex structure.
The punctures of these $J$-holomorphic disks can be either Reeb chords or  intersection points.
For generators $c_0, c_1,\dots, c_m$ in $Cth(\Sigma^1, \Sigma^2)$,
let $\calM_{J}(c_0; c_1,\dots, c_m)$ denote the moduli space of the $J$-holomorphic disks:
$$u: (D_{m+1}, \partial D_{m+1}) \to (\bbR\times \bbR^3, \Sigma^1 \cup \Sigma^2)$$
with the following properties:
\begin{itemize}
\item $D_{m+1}$ is a $2$-dimensional unit disk with $m+1$ boundary points $q_0 , q_1, \dots , q_m$ removed and
the points $q_0, q_1, \dots , q_m$ are arranged in a counterclockwise order. 
 \item If $c_0$ is a Reeb chord, the image of $u$ is asymptotic to $[N, \infty) \times c_0$ near $q_0$.
   If $c_0$ is an intersection point, then $\displaystyle{\lim_{z \to q_0} u(z)=c_0}$
and $u$ maps the incoming segment (resp. outgoing segment) of the boundary to $\Sigma^2$ (resp. $\Sigma^1$).
 \item If, for $i >0$, $c_i$ is a Reeb chord, the image of $u$ is asymptotic to $( -\infty, -N] \times c_i$ near $q_i$.
 If $c_i$ is an intersection point, then $\displaystyle{\lim_{z \to q_i} u(z)=c_i}$
and $u$ maps the incoming segment (resp. outgoing segment) of the boundary to $\Sigma^1$ (resp. $\Sigma^2$).
\end{itemize}

The four types of Moduli spaces used in the differential $d$ of the Cthulhu chain complex are shown in Figure \ref{Moduli}.
For any one of these four moduli spaces $\calM$, 
we say a disk $u\in \calM$ is  {\bf rigid} if $\dim \calM =0$.
All the moduli spaces of holomorphic disks introduced above 
admit a coherent orientation since both $\Sigma^1$ and $\Sigma^2$ are spin.
Therefore, one can associate each rigid holomorphic disk $u \in \calM$  with a sign
and thus can count the number of rigid holomorphic disks in $\calM$ with sign.

Let $a^{\pm}_i$'s be Reeb chords to $\Lambda^1_{\pm}$ from $\Lambda^2_{\pm}$
and $x_i$'s be double points of $\Sigma^1  \cup  \Sigma^2$. 
The bold letters $\bf p^{\pm}$, ${\bf q}^{\pm}$ are words of pure Reeb chords of $\Lambda^1_{\pm}$ and $\Lambda^2_{\pm}$,  respectively. 
Here we assume that, for $i=1,2$,
 $\epsilon^{i}_{-}$ is an augmentation of $\calA(\Lambda^i_-)$ and $\epsilon^i_+$ is the augmentation of $\calA(\Lambda^i_+)$ induced by $\Sigma_i$. 
The differential is defined as follows:

	$$
	\begin{array}{lll}
	\displaystyle{d_{++}(a^{+}_i)} &=&\displaystyle{ \sum_{\dim \widetilde{\calM}_{J_{+}}(a^+_j; {\bf p}^+, a^+_i, {\bf q}^+) =0} |\widetilde{\calM}_{J_{+}}(a^+_j; {\bf p}^+, a^+_i, {\bf q}^+)| \epsilon^1_+({\bf p}^+) \epsilon^2_+({\bf q}^+) a_j^+};\\

	\displaystyle{d_{--}(a^{-}_i)} &=& \displaystyle{\sum_{\dim \widetilde{\calM}_{J_-}(a^{-}_j;{\bf p}^{-}, a^{-}_i, {\bf q^{-}}) =0} |\widetilde{\calM}_{J_-}(a^{-}_j;{\bf p}^{-}, a^{-}_i, {\bf q^{-}})| \epsilon^1_-({\bf p}^{-}) \epsilon^2_-( {\bf q^{-}}) a_j^-};\\
	\\

	         \displaystyle{d_{00} (x_i)} &=&\displaystyle{ \sum_{\dim \calM_{J}(x_j; {\bf p}^-,  x_i, {\bf q}^-) =0} |\calM_{J}(x_j; {\bf p}^-,  x_i, {\bf q}^-) | \epsilon^1_-({\bf p}^{-}) \epsilon^2_-( {\bf q^{-}}) x_j} ;\\
         \\
         \displaystyle{d_{0-}(a^-_i)} &=& \displaystyle{\sum_{\dim \calM_{J}(x_j; {\bf p}^-, a^-_i, {\bf q^-}) =0} |\calM_{J}(x_j; {\bf p}^-, a^-_i, {\bf q^-})| \epsilon^1_-({\bf p}^{-}) \epsilon^2_-( {\bf q^{-}})  x_j} ;\\
         \\
         \displaystyle{d_{+0}(x_i)} &=& \displaystyle{\sum_{\dim \calM_{J}(a^+_j;  {\bf p^-}, x_i, {\bf q^-}) =0} |\calM_{J}(a^+_j; {\bf p}^-, x_i, {\bf q^-})| \epsilon^1_-({\bf p}^{-}) \epsilon^2_-( {\bf q^{-}}) a^+_j} ;\\
         \\
         \displaystyle{d_{+-}(a^-_i)} &=&\displaystyle{ \sum_{\dim \calM_{J}(a^+_j;  {\bf p^-}, a^-_i, {\bf q^-}) =0} |\calM_{J}(a^+_j;  {\bf p^-}, a^-_i, {\bf q^-})| \epsilon^1_-({\bf p}^{-}) \epsilon^2_-( {\bf q^{-}})  a^+_j} ,\\
         \\
         \end{array}
         $$
where  $|\calM|$ denotes the number of rigid holomorphic disks in the moduli space $\calM$ counted with sign. 
Note that the definition of differential depends on the choice of augmentations $\epsilon_1^-$ and $\epsilon_2^-$, whose existence  are essential to the Floer theory.

A holomorphic disk counted by the differential must satisfy the  rigidity condition and  the positive energy condition.
We will describe these conditions in details.
	
{\bf The Rigidity Condition.} Let us interpret the condition $\dim \calM_J(c_1; {\bf p}, c_2, {\bf q})=0$ in terms of $|c_i|$ for $i=1,2$,
where $c_i$ can be either a Reeb chord or an intersection point
while $\bf p$ and $\bf q$ are words of pure Reeb chords in degree $0$.
Instead of deriving a formula for the dimension of a moduli space, we use the idea of the wrapped Floer homology to find the relation between $|c_1|$ and $|c_2|$.	
Recall that  both $\Sigma^1$ and $\Sigma^2$ are cylindrical outside of $[-N, N]\times \bbR^3$. 
Consider a non-decreasing function $\sigma(t): \bbR_{\ge 0} \to \bbR_{\ge 0}$  
such that $\sigma'(t)=0$ when $t \le N$ and $\sigma'(t)=1$ when $t \ge N'$, 
where $N'$ is a number bigger than $N$.  
Note that $X_H = - \sigma(|t|) \partial_z$  is a Hamiltonian vector field with its time-$s$ flow denoted by $\Phi_H^s$.	
Flow $\Sigma^1$ through $X_H$ and get a  new cobordism $\Phi_H^s(\Sigma^1)$, which is another exact Lagrangian cobordism according to Section \ref{Pair}. 
Observe that  $\Phi_H^s(\Sigma^1)$ wraps $\Sigma^1$ on both ends in the negative Reed chord direction.
Hence for a large enough number $s$, each Reeb chord $c$ to $\Lambda^1_{+}$ (resp. $\Lambda^1_{-}$) from $\Lambda^2_{+}$ (resp. $\Lambda^2_{-}$)  corresponds to a transversally double point $\check{c}$  of  $  \Phi_H^s(\Sigma^1)  \cup \Sigma^2$ in $N< t < N'$ 
(resp. $-N'< t < -N$).
Moreover,
if $c$ is a Reeb chord in $C(\Lambda^1_{+}, \Lambda^2_{+})$, we have
$$|\check{c}|= CZ(\gamma_{\check{c}})= CZ(\gamma_{c})+1=|c|+2.$$
If $c$ is Reeb chord  in $C(\Lambda^1_{-},  \Lambda^2_{-})$, 
$$|\check{c}|= CZ(\gamma_{\check{c}})= CZ(\gamma_{c})=|c|+1.$$
Each double point $x$ of $\Sigma^1 \cup \Sigma^2$  naturally corresponds to a double point
$\check{x}$
of $ \Phi_H^s(\Sigma^1)  \cup \Sigma^2$ in $-N < t < N$
with gradings satisfying $|\check{x}|=|x|$.

\begin{rmk}
The difference in grading correspondence  between Reeb chords in $C(\Lambda^1_{+}, \Lambda^2_{+})$ and Reeb chords in $C(\Lambda^1_{-}, \Lambda^2_{-})$ can be understood better in a special case where $\Sigma^1$ is a push off of $\Sigma^2$ through a positive Morse function $F: \Sigma^2 \to \bbR_{> 0}$.
In other words, in a Weinstein neighborhood of $\Sigma^2$, the cobordism $\Sigma^1$ is the graph of $dF$ for some positive Morse function $F: \Sigma^2 \to \bbR_{> 0}$.
In this case,
 the cobordism
$\Phi_h^s(\Sigma^1)$ is a push off of $\Sigma^2$ through another Morse function $\tilde{F}$ as well.
By the canonical Floer theory \cite{Floer},
we can choose a family of  capping paths so that $CZ(\gamma_x)=Ind_{\tilde{F}}(x)$ for any intersection point $x$ of $\Phi_h^s(\Sigma^1)$ and $\Sigma^2$.
Similarly, 
for any Morse Reeb chord $c$ in $\Lambda^1_+\cup \Lambda^2_+$, we can further require that $Ind_{f_+}(c)= CZ(\gamma_c)$, where $f_{+}=F\big\vert_{\Lambda^2_+}$.
Notice that
 $Ind_{\tilde{F}}(\check{c})= Ind_{f_+}(c)+1$.
Therefore
$$|\check{c}|= CZ(\gamma_{\check{c}})= Ind_{\tilde{F}}(\check{c})= Ind_{f_+}(c)+1= CZ(\gamma_{c})+1=|c|+2.$$
For a Morse Reeb chord $c$  in $\Lambda^1_-\cup \Lambda^2_-$, 
the indices satisfy $Ind_{\tilde{F}}(\check{c})= Ind_{f_-}(c)$, where $f_{-}=F\big\vert_{\Lambda^2_-}$.
Hence $|\check{c}|= CZ(\gamma_{\check{c}})= Ind_{\tilde{F}}(\check{c})= Ind_{f_-}(c)= CZ(\gamma_{c}) =|c|+1.$
A schematic figure is shown in Figure \ref{wrapind}.

\begin{figure}[!ht]
\labellist
\small
{\color{red}
\pinlabel $\Sigma_1$ at -10 85
\pinlabel $\Phi_h^s(\Sigma_1)$ at 165 30
\pinlabel $\Sigma_1$ at -10  175
\pinlabel $\Phi_h^s(\Sigma_1)$ at 170 190
}
{\color{blue}
\pinlabel $\Sigma_2$ at -10 55
\pinlabel $\Sigma_2$ at -10 145
\pinlabel $\Sigma_2$ at 170 55
\pinlabel $\Sigma_2$ at 175 145
}
\pinlabel $a^+$ at 100 65
\pinlabel $a^-$ at 40 65
\pinlabel $a^+$ at 280 65
\pinlabel $a^-$ at 232 65
\pinlabel $\check{a}^+$ at  325 65
\pinlabel $\check{a}^-$ at 195 65
\pinlabel $t$ at 100 15
\pinlabel $t$ at 285 15

\endlabellist
\includegraphics[width=3.9in]{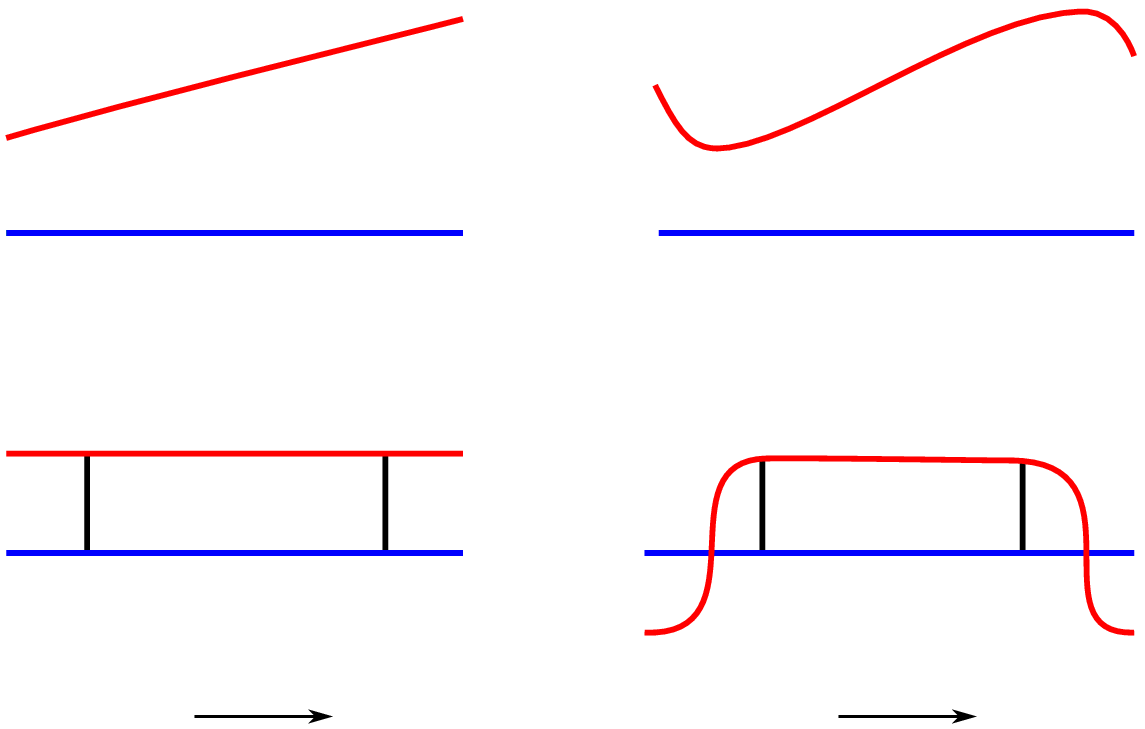}
\caption{The top two are front projections while the bottom two are Lagrangian projections.
The indices satisfy
$|\check{a}^+| =|a^+|+2$ and $|\check{a}^-|=|a^-|+1.$}
\label{wrapind}
\end{figure}
\end{rmk}

So far, we have shown that  generators of $Cth(\Sigma^1, \Sigma^2)$ can be identified  with  intersection  points of $\Phi_H^s(\Sigma^1)$ and $\Sigma^2$, 
which are generators of $Cth(\Phi_H^s(\Sigma^1), \Sigma^2)$.
Moreover, by \cite[Proposition 8.2]{CDGG}, the Cthulhu chain complexes $Cth(\Sigma^1, \Sigma^2)$  and 
$Cth(\Phi_H^s(\Sigma^1), \Sigma^2)$ are identified on the level of complexes as well.
Note that the generators of $Cth(\Phi_H^s(\Sigma^1), \Sigma^2)$ do not contain any Reeb chords and hence we have $Cth(\Phi_H^s(\Sigma^1), \Sigma^2) = \big(CF(\Phi_H^s(\Sigma^1), \Sigma^2), d_{00}\big)$.
Lift $\Phi_H^s(\Sigma^1)\cup \Sigma^2$ to be a Legendrian submanifold $L$ in $\bbR\times \bbR^3\times \bbR$.
Note that $\big(CF(\Phi_H^s(\Sigma^1), \Sigma^2), d_{00}\big)$ is the dual of the linearized contact homology chain complex of $L$ as introduced in Section \ref{DGA}.
Hence if $\dim \calM(c_1; {\bf p}, c_2, {\bf q})=0$, 
there are rigid holomorphic disks that have a positive puncture at $\check{c}_1$ and
a negative puncture at $\check{c}_2$, which implies  $|\check{c}_1|-|\check{c}_2|=1$.
We can get the corresponding grading relation between $c_1$ and $c_2$.
In particular, let $a^{\pm}$ be Reeb chords in $C(\Lambda^1_{\pm}, \Lambda^2_{\pm})$ and $x_1$,$x_2$ be intersection points of $\Sigma^1$ and $\Sigma^2$ while ${\bf p}$ and ${\bf q}$ are words of pure Reeb chords in degree $0$ of $\Lambda_-^1$ and $\Lambda_-^2$, respectively.
 We have that
\begin{itemize}
\item if $\dim \calM_{J}(x_1; {\bf p},  x_2,  {\bf q})=0$, then $|c_1|-|c_2|=1$;
\item if $\dim \calM_{J}(x_1; {\bf p},  a^- , {\bf q})=0$, then $|c_1|-|a^- |=2$;
\item if $\dim \calM_{J}(a^+; {\bf p},  x_1,  {\bf q})=0$, then $|a^+|-|x_1 |=-1$;
\item if $\dim \calM_{J}(a^+; {\bf p},  a^-,  {\bf q})=0$, then $|a^+|-|a^- |=0$.
\end{itemize}

As a result, the Cthulhu chain complex can be written as 
$$Cth^{k}(\Sigma^1, \Sigma^2) =C^{k-2}(\Lambda_+^1, \Lambda_+^2) \oplus CF^{k}(\Sigma^1, \Sigma^2) \oplus C^{k-1}(\Lambda_-^1, \Lambda_-^2). $$
Under this decomposition, the differential	
	$$d= 
	\begin{pmatrix}
	d_{++}& d_{+0} & d_{+-}\\
	0 & d_{00} & d_{0-}\\
	0 & 0 & d_{--}\\
	\end{pmatrix}
	$$
has degree $1$ as we expected.

{\bf The Positive Energy Condition.} We interpret the positive energy condition of a holomorphic disk 
$u \in \calM(c_0; c_1,\dots, c_m)$ in terms of the action of $c_i$, where $i=0,\dots,m$.
Following \cite{EkhSFTZ2}, we define the energy $E(u)$ of a holomorphic disk $$u: (D^2, \partial D^2) \to( \bbR \times \bbR^3, \Sigma^1  \cup  \Sigma^2),$$  by
	$E(u)=E_{\omega}(u)+E_{\alpha}(u)$,
	where the $\omega$-energy 
	$$E_{\omega}(u) = \displaystyle{\int_{u^{-1}([-N,N] \times \bbR^3)} u^* (\omega)+\int_{u^{-1}((-\infty,-N) \times \bbR^3)} u^* (e^{-N}d\alpha) + \int_{u^{-1}((N, \infty) \times \bbR^3)} u^* (e^{N}d\alpha)}.$$
	Note that
	Write $u=(t,v)$, where $t:D^2 \to \bbR$ and $v: D^2 \to \bbR^3$.
	Define the $\alpha$-energy  by
	$$\displaystyle{E_{\alpha}(u) = \sup_{\phi_-}\ \left( \int_{u^{-1}((-\infty,-N) \times \bbR^3)} \ \phi_-(t)\  d t \wedge (v^* \alpha)\right) 
	+ \sup_{\phi_+}\ \left( \int_{u^{-1}((N, \infty) \times \bbR^3)} \ \phi_+(t)\  d t \wedge (v^* \alpha)\right)},$$
	where $\phi_+$ and $\phi_-$ range over all  compact supported smooth functions such that
	$$\displaystyle{\int_{-\infty}^{-N} \phi_-(t) \ dt=e^{-N}}  \textrm{  and  } 
	\displaystyle{\int_{N}^{\infty}\phi_+(t) \ dt=e^N},$$
	respectively.
By Stokes' Theorem, for any holomorphic disk $u \in \calM(c_0; c_1,\dots, c_m)$, 
we have
$$E_{\omega}(u) = \fraka(c_0) - \displaystyle{\sum_{l=1}^m}\fraka(c_l).$$
A holomorphic disk has positive $\omega$-energy, i.e. $E_{\omega}(u)>0$, which implies
that
$\fraka(c_0) > \displaystyle{\sum_{l=1}^m}\fraka(c_l)$.

Under the assumption in Remark \ref{Assumption}, the differential of the Cthulhu chain complex $d$ is of the form of an upper triangle.
By \cite[Section 6]{CDGG}, we have  $d^2=0$ and thus $d$ is a differential map.
Moreover, from \cite[Section 8]{CDGG}, the induced cohomology 
$H^*Cth(\Sigma^1, \Sigma^2)$ is an invariant under compactly supported Hamiltonian isotopies.
Push $\Sigma^1$ along the negative $z$ direction until $\Sigma^1$ is far below $\Sigma^2$
and then, there is no Reeb chord to $\Sigma^1$ from $\Sigma^2$ nor intersection point between $\Sigma^1$ and $\Sigma^2$.
It is obvious that  the cohomology is trivial, i.e. $H^*Cth(\Sigma^1, \Sigma^2)=0$.

\vspace{0.3in}

\section{Main Result}

In Section \ref{Pair}, we perturb an exact Lagrangian cobordism using a Morse function and obtain a pair of Lagrangian cobordisms. 
 In Section \ref{LES}, we apply the Floer theory to this pair of cobordisms and get the long exact sequence in Theorem \ref{thm1}.
In Section \ref{geointer},
we describe the rigid holomorphic disks counted by $d_{+-}$, which is a part of the differential map of the Cthulhu chain complex,
 in terms of holomorphic disks with boundary on $\Sigma$ and Morse flow lines.
This is useful when identifying $f_1$, i.e. the category map on the level of morphisms,  with $d_{+-}$.
In Section \ref{aside}, we extend the method in Section \ref{geointer} to describe the differential $d$ of the Cthulhu chain complex  and recover the long exact sequences in \cite{CDGG}.
Finally, we use the identification  in Section \ref{geointer} between $f_1$ and $d_{+-}$ to prove Theorem \ref{thm5} in  Section \ref{injection}.

\subsection{Construction of the pair of cobordisms}\label{Pair}

First let us describe the neighborhood of a Lagrangian cobordism.
Let $\Sigma$ be an exact Lagrangian cobordism from $\Lambda_-$ to $\Lambda_+$ in $\big(\bbR \times \bbR^3, d(e^t \alpha)\big)$, where $\Lambda_-$ and $\Lambda_+$ are Legendrian links.
By the Weinstein Lagrangian neighborhood theorem, there is a symplectomorphism
$$\psi: \mbox{nbhd}(\Sigma) \subset \big( (\bbR \times \bbR^3), d(e^t \alpha)\big) \to \big(T^*\Sigma, d\theta \big),$$
where $\theta$ is the negative Liouville form $\theta = - \displaystyle{\sum p_i dq_i }$
 of $T^*\Sigma$
with coordinates $((q_1,q_2), (p_1,p_2))$.
Specifically,
on the $(\pm \infty-)$boundary $\bbR\times \Lambda_{\pm}$, the symplectomorphism $\psi$ is given by a composition of two symplectomorphisms $\psi_1 \circ \psi_0$.
As mentioned before, there is a contactomorphism from a tubular neighborhood of $\Lambda_{\pm}$ in $\bbR^3$
to a neighborhood of the zero section of $J^1(\Lambda_{\pm})$. 
Composing with the identity map on $\bbR_t$, we get a symplectomorphism 
$\psi_0$ from the neighborhood of $\bbR \times \Lambda_{\pm}$ in $\bbR \times \bbR^3$ to 
$\bbR\times J^1(\Lambda_{\pm})$.
The second part $\psi_1$ is given by
$$
\begin{array}{rl}
\psi_1: \mbox{nbhd}(\Sigma) \subset \big( (\bbR \times J^1(\Lambda_{\pm})), d(e^t \alpha)\big) &\to \big(T^*(\bbR_{> 0} \times \Lambda_{\pm}), d\theta \big)\vspace{0.1in}\\
(t, (q,p, z)) &\mapsto \big((e^t, q),(z, e^t p)\big)
\end{array}
$$

For a Morse function $F: \Sigma \to \bbR_{\ge 0}$ such that the determinant of the Hessian matrix is small enough, the graph of $dF$ is a  Lagrangian submanifold in $T^*\Sigma$. 
Pull it back to $\bbR \times \bbR^3$ and denote $\psi^{-1}(\mbox{graph}(dF))$ by $\Sigma'$.

Now we show that $\Sigma'$ is an exact Lagrangian submanifold as well.
Notice that $V_{(\bf{q},\bf{p})}:= dF|_{\bf{q}}$ is a Hamiltonian vector field in $T^*\Sigma$ since $\iota_{V} d\theta =- d\tilde{F}$,
where $\tilde{F}= F\circ \pi$  and $\pi$ is the natural projection $\pi: T^*(\Sigma) \to \Sigma$.
In order to extend $\psi^{-1}_{*}(V)$ to be a Hamiltonian vector field in $\bbR \times \bbR^3$,
we choose a smoothly cut off function $\gamma: T^*(\Sigma) \to \bbR$ such that 
$\gamma({\bf q}, {\bf p})=1$ in a tubular neighborhood of the zero section containing the graph of $dF$
and $\gamma({\bf q}, {\bf p})=0$ outside of a slightly bigger tubular neighborhood of the zero section.\vspace{0.01in}
Pull the Hamiltonian vector field of $\gamma \cdot \tilde{F}$ back through $\psi$ and  extend to a  Hamiltonian vector field $X_H$ in $\bbR\times \bbR^3$.
For a suitable neighborhood of $\Sigma$ in $\bbR \times \bbR^3$, we have
$$\displaystyle{\iota_{X_H} d(e^t \alpha) \bigg|_{nbhd(\Sigma)} = \psi^*(\iota_{V} d\theta) = \psi^*(-d\tilde{F})= d(-\tilde{F}\circ \psi)\bigg|_{nbhd(\Sigma)}}.$$
Hence its Hamiltonian $H= -\tilde{F} \circ \psi$ around $\Sigma$.
Denote the time $s$ flow of $X_H$ by $\phi_H^s$ and thus
$\Sigma' = \phi_H^1(\Sigma)$.
We can compute the $1$-form on $\Sigma'$:
\begin{equation}\label{primitive}
\begin{array}{rl}
\displaystyle{{\phi_H^1}^* e^t\alpha} &= \displaystyle{e^t\alpha + \int_0^1   \frac{d}{ds}\ {\phi_H^s}^*( e^t\alpha)\ ds}\\
&\\

&= \displaystyle{e^t\alpha + \int_0^1  { \phi_H^{s}}^*(\iota_{X_H} d (e^t\alpha) + d(\iota_{X_H} e^t\alpha))\ ds}\\
&\\

&=\displaystyle{e^t\alpha+ \int_0^1  { \phi_H^{s}}^*(dH + d(e^t\alpha(X_H))) \ ds}\\
&\\
&=  \displaystyle{ e^t\alpha + d(\int_0^1  (H + e^t\alpha(X_H)) \circ \phi_H^{s} \ ds)}.\\
\end{array}
\end{equation}
Thus $\Sigma'$ is exact.
Moreover,
if $\Sigma$ has  a primitive $g$,
then  $\Sigma'$ has a primitive  $$g+\displaystyle{\int_0^1  (H + e^t\alpha(X_H)) \circ \phi_H^{s} \ ds}.$$

We are going to construct a  particular Morse function for $\Sigma$ such that the image of the Morse function has cylindrical ends as well and thus is an exact Lagrangian cobordism.
Suppose $\Sigma$ is cylindrical outside of $[-N+\delta, N-\delta]\times \bbR^3$, where $0<\delta<1$.
Choose Morse functions $g_{\pm}: \Lambda_{\pm} \to \displaystyle{ (0,{1}/{2})}$ and $G: \Sigma  \cap  \big( [-N, N]\times \bbR^3\big)\to (0,1)$ such that 
$$G \big\rvert_{ \Sigma \cap\{ t \in [-N, -N+\delta] \cup [N-\delta,N]\}} = e^t.$$
Define a smooth non-decreasing function $\rho: \  \bbR_{> 0} \to [0,1]$ such that $\rho(s)=0$ for $s\le 1$ and $\rho(s) =1$ for $s \ge e^{\delta}$.

For $0< \eta  < e^{-2}$,  define a Morse function $F^{\eta} : \Sigma \to \bbR_{> 0}$ to be 
$$
F^{\eta}(t, q):=  
\begin{cases}
\eta^{2N} g_-(q) s, & \mathrm{if }\   t < -N;\\
(\rho(e^N s)(\eta^{N} - \eta^{2N} g_-(q)) + \eta^{2N} g_-(q))s, & \mathrm{if }\  -N \le t \le -N+\delta;\\
\eta^{N} G, & \mathrm{if} \ -N+\delta < t < N-\delta;\\
(\eta^{N} + \rho(e^{-N+\delta} s) \eta^{N} g_+(q))s, & \mathrm{if} \ N-\delta \le t \le N;\\
(\eta^N +\eta^N g_+(q))s,  & \mathrm{if}\  t>N,\\
\end{cases}
$$
where $s=e^t$.
One can check that $F^{\eta}$ has the following properties:
\begin{itemize}
\item
 The Morse function $F^{\eta}$ is increasing with respect to $t$ when $t\le -N+\delta$ or $t \ge N-\delta$.
 This implies that the  critical points of $F^{\eta}$ and the  critical points of $G$ are in $1-1$ correspondence and are all contained in $\Sigma \cap \big( [-N, N] \times \bbR^3\big)$.
 \item 
The Morse function $F^{\eta}$ is bounded by $ 2 \eta^N e^N $ on $\Sigma \cap \big([-N, N]\times \bbR^3\big)$.
\item Write $F^{\eta}\big\vert_{\{ N\} \times \Lambda_{+}}$ as $f^{\eta}_{+}e^{N}$ and $F^{\eta}\big\vert_{\{ -N\} \times \Lambda_{-}}$ as $f^{\eta}_{-}e^{-N}$ , respectively.
The graph of  $dF^{\eta}$ on $(-\infty,-N) \times \Lambda_-$ is the same as $(-\infty,-N) \times graph(df^{\eta}_-)$ and the graph of  $dF^{\eta}$ on $(N, \infty) \times \Lambda_+$ is the same as $(N, \infty) \times graph(df^{\eta}_+)$.
\end{itemize}

Push $(\Sigma, \Lambda_+, \Lambda_-)$ off through $F^{\eta}$ and obtain a copy of $(\Sigma, \Lambda_+, \Lambda_-)$, labeled by $(\Sigma^1, \Lambda^1_+, \Lambda^1_-)$.
Label the original $(\Sigma, \Lambda_+, \Lambda_-)$ by $(\Sigma^2, \Lambda^2_+, \Lambda^2_-)$.
Thus $\Sigma^1$ is a push off of $\Sigma^2$ through $F^{\eta}$ and 
 $\Lambda^1_+$ (resp. $\Lambda^1_-$) is a push off  of $\Lambda^2_+$ (resp. $\Lambda^2_-$) through $f^{\eta}_+$ (resp. $f^{\eta}_-$). 

\subsection{The long exact sequence}\label{LES}
Now we apply the Floer theory to the pair of cobordisms $\Sigma^1 \cup \Sigma^2$ constructed in Section \ref{Pair} and get a long exact sequence.
Combining the long exact sequence with the augmentation category map induced by the exact Lagrangian cobordism,
we obtain an obstruction to the existence of the exact Lagrangian cobordisms.

Recall that the grading for generators in the Cthulhu chain complex depends on the choice of capping paths.
According to the canonical Floer theory \cite{Floer}, we can choose a family of capping paths such that
the Conley-Zehnder  index 
of any double point  $x$ of $\Sigma^1\cup \Sigma^2$ satisfies $CZ(\Gamma_x) = Ind_{F^{\eta} (x)}$. 
Now we apply the Floer theory to the pair of Lagrangian cobordisms $\Sigma^1\cup \Sigma^2$ and have the following theorem.

\begin{thm}
Let $\Sigma^i$, for $i=1,2$, be the cobordisms from  $\Lambda^i_-$ to $\Lambda^i_+ $
as constructed in Section \ref{Pair}. 
Suppose $\epsilon^i_-$ is an augmentation of $\calA(\Lambda^i_-)$
and $\epsilon^i_+$ is the  augmentation of $\calA(\Lambda^i_+)$ induced by $\Sigma^i$.
Fix a suitable domain dependent almost  complex structure on $\bbR \times \bbR^3$ that is compatible with the symplectic form and cylindrical ends.
For $\eta$ small enough,
the Cthulhu chain complex is
$$Cth^{k}(\Sigma^1, \Sigma^2) =C^{k -2}(\Lambda_+^1, \Lambda_+^2) \oplus CF^{k }(\Sigma^1, \Sigma^2) \oplus C^{k -1}(\Lambda_-^1, \Lambda_-^2).$$
Under this decomposition, the differential is
	$$d= 
	\begin{pmatrix}
	d_{++}& d_{+0} & d_{+-}\\
	0 & d_{00} & d_{0-}\\
	0 & 0 & d_{--}\\
	\end{pmatrix}.
	$$	
Moreover,
\begin{enumerate}
\item the map $d_{00}$ is the Morse co-differential induced by $F^{\eta}$, i.e.,
 the chain complex $(CF^{ k}(\Sigma^1, \Sigma^2), d_{00})$ is the Morse co-chain complex $(C^{k}_{Morse} F^{\eta}, d_{F^{\eta}})$ induced by $F^{\eta}$;
\item  the chain complex 
 $(C^{k -2}(\Lambda_+^1, \Lambda_+^2) , d_{++})=(Hom^{k-1}_+(\epsilon^1_+, \epsilon^2_+), m_1)$
 while the chain complex
$(C^{k -1}(\Lambda_-^1, \Lambda_-^2), d_{--})=(Hom^{k }_+(\epsilon^1_-, \epsilon^2_-), m_1)$.
\end{enumerate}
\end{thm}

\begin{proof}

First, we need to show that all the intersection points $x \in CF^{k}(\Sigma^1, \Sigma^2)$ have positive action,
which is the condition for the differential to have the form above by Remark \ref{Assumption}.

Let $g_i$ be a primitive of $\Sigma^i$ for $i=1,2$.
According to the computation (\ref{primitive}), we have
$$g_1 =  g_2 + \displaystyle{\int_0^1  (H + e^t\alpha(X_H)) \circ \phi_H^{s} \ ds},$$
where $H= -\tilde{F^{\eta}} \circ \psi$.
It is not hard to check that $g_1=g_2$ on $\Sigma \cap \big((-\infty, -N)\times \bbR^3\big)$.
Therefore we can assume $g_1=g_2=0$ on $\Sigma \cap \big((-\infty, -N)\times \bbR^3\big)$
 and use $g_1$ and $g_2$ as primitives to define action.
The action of  each intersection point $x$ is 
$$\fraka(x) = g_2(x)-g_1(x)=\displaystyle{\int_0^1  (H + e^t\alpha(X_H)) \circ \phi_H^{s} \ ds}.$$
Notice that the vector field $X_H$ vanishes at the intersection point $x$.
Hence
$$\fraka(x)= -\displaystyle{\int_0^1 H\circ \phi_H^s \ ds= -H = F^{\eta} \circ \psi(x)}>0.$$

Next, we are going to show that for $\eta$ small enough, 
the rigid holomorphic disks that contribute to $d_{00}$ do not include any pure Reeb chords as negative punctures. 
Let $x$ and $y$ be two double points of $\Sigma^1 \cup \Sigma^2$. 
For any rigid holomorphic disk $u \in \calM(x; {\bf p}, y, {\bf q}) $, where ${\bf p}$ and ${\bf q}$ are words of pure Reeb chords of $\Lambda_-^1\cup \Lambda_-^2$,
we have the energy estimate:
$$E_{\omega}(u) \le \fraka(x)-\fraka(y) -\fraka({\bf p})-\fraka({\bf q}).$$
Since $u$ has positive energy, we have
$$\fraka({\bf p})+\fraka({\bf q}) \leq F^{\eta}(\psi(x))-F^{\eta}(\psi(y)) \leq F^{\eta}(\psi(x)).$$
Therefore, for $\eta$ small enough such that the maximum of the Morse function $F^{\eta}$ is smaller than the minimum action of  pure Reeb chords of $\Lambda_-^1$ and $\Lambda_-^2$, 
the moduli space that contributes to $d_{00}$ is of the form $\calM(x;y)$. 
By \cite[Lemma 6.11]{EESduality}, the boundary of a rigid  holomorphic disk with two punctures at intersection points converge to a rigid Morse flow line,  which implies $d_{00} = d_{F^{\eta}}$.
Furthermore, the gradings satisfy $|x|=CZ(\gamma_x) = Ind_{F^{\eta} (x)}$. 
Therefore $(CF^{ k}(\Sigma^1, \Sigma^2), d_{00}) = (C^{k}_{Morse} F^{\eta}, d_{F^{\eta}})$.

Recall that  there is a natural identity map with degree $1$ from
$C^{k-1}(\Lambda_{\pm}^1,\Lambda_{\pm}^2)$ to 
$Hom^{k}_+(\epsilon_{\pm}^1,\epsilon_{\pm}^2)$, respectively.
Moreover, the definitions of $d_{\pm \pm}$ and $m_1$ match as well.
 Hence we have $(C^{k -1}(\Lambda_-^1, \Lambda_-^2), d_{--}) = (Hom^{k }_+(\epsilon^1_-, \epsilon^2_-), m_1)$
while $(C^{k -2}(\Lambda_+^1, \Lambda_+^2) , d_{++})= (Hom^{k-1}_+(\epsilon^1_+, \epsilon^2_+), m_1)$.
\end{proof}



For the rest of the paper, we fix a small enough $\eta$ and  write $F^{\eta}$, $f^{\eta}_+$, $f^{\eta}_-$ as ${F}$, $f_+$, $f_-$, respectively.
According to the Floer theory in Section 4, we have $H^{k}(Cth(\Sigma^1, \Sigma^2), d)=0$, where
$$Cth^{k}(\Sigma^1, \Sigma^2) =Hom_+^{ k-1} (\epsilon^1_+, \epsilon^2_+) \oplus C_{Morse}^{ k}F \oplus Hom_+^ {k}(\epsilon^1_-, \epsilon^2_-) ,$$
and
$$d= 
	\begin{pmatrix}
	m_1 & d_{+0} & d_{+-} \\
	0 & d_F & d_{0-} \\
	0 &  0 & m_1\\
	\end{pmatrix}.
$$
Consider the chain map $\Psi= d_{+-}+d_{0-}$:
 $$\Psi: \ (Hom_+^ {k}(\epsilon^1_-, \epsilon^2_-), m_1) \to \left( Hom_+^{ k} (\epsilon^1_+, \epsilon^2_+)\oplus C_{Morse}^{ k+1}F,\ 
  d'=\begin{pmatrix}
	m_1 & d_{+0}\\
	0  &d_F \\
	\end{pmatrix}
  \right). $$
Notice that the mapping cone of $\Psi$ has trivial homology. 
Therefore,
$$H^{k}(Hom_+(\epsilon^1_-,\epsilon^2_-)) \cong H^{k} Cone(d_{+0}).$$
Hence we have the following long exact sequence:

{\xymatrixcolsep{0.8pc}
\xymatrix{\cdots\arr& H^{k}(C_{Morse} F, d_{F}) \arr & H^{k}Hom_+(\epsilon^1_+,\epsilon^2_+) \arr & H^{k}Hom_+(\epsilon^1_-,\epsilon^2_-) \arr & H^{k+1}(C_{Morse} F, d_{F})\arr &\cdots}}
Moreover, notice that  in the construction in Section \ref{Pair}, the gradient flows of $F$ flow in from the bottom and out of the  top. Hence we have
$$H^{k}(C_{Morse} F, d_F) =H^{k}(\Sigma, \Lambda_-).$$

\begin{cor}\label{lescor}
Let $\Sigma$ be an exact Lagrangian cobordism with Maslov number $0$ from $\Lambda_-$ to $\Lambda_+$.
For $i=1,2$, if
$\epsilon^i_-$ is an augmentation of $\calA(\Lambda_-)$
and $\epsilon^i_+$ is the augmentation of $\calA(\Lambda_+)$ induced by $\Sigma$,
then we have the following long exact sequence:
\begin{equation}\label{les1}
{\xymatrixcolsep{1pc}
\xymatrix{
\cdots  \arr& H^{k}(\Sigma, \Lambda_-)  \arr& H^k Hom_+(\epsilon^1_+,\epsilon^2_+)  \arr& H^{k}Hom_+(\epsilon^1_-,\epsilon^2_-)  \arr& H^{k+1}(\Sigma, \Lambda_-) \arr& \cdots }}
\end{equation}
\end{cor}

\begin{rmk}
When the Maslov number of $\Sigma$ is $d$, which is not $0$, the method above works as well. 
The only difference is that the grading  of generators in the Cthulhu chain complex is defined mod $d$.
Thus, the long exact sequence (\ref{les1}) holds with gradings mod $d$.
\end{rmk}

If $\epsilon^1_- = \epsilon^2_-= \epsilon_-$, by \cite[Section 5.2]{NRSSZ},  we have the identification
$$H^k Hom_+(\epsilon, \epsilon) \cong LCH^{\epsilon}_{1-k}(\Lambda),$$
where $LCH^{\epsilon}_k(\Lambda)$ is  the linearized contact homology of $\Lambda$.
The long exact sequence (\ref{les1}) can be rewritten in terms of linearized contact homology:
\diag{\cdots \arr & H^{k}(\Sigma, \Lambda_-) \arr & LCH^{\epsilon_+}_{1-k}(\Lambda_+) \arr& LCH^{\epsilon_-}_{1-k}(\Lambda_-)  \arr& H^{k+1}(\Sigma, \Lambda_-)  \arr& \cdots }

Furthermore, if $\Lambda_-$ is empty, then $\Sigma$ is an exact Lagrangian filling of $\Lambda_+$
and $\epsilon_+$ is an augmentation of $\calA(\Lambda_+)$ induced by the Lagrangian filling.
The long exact sequence (\ref{les1}) gives 
$$H^{k}(\Sigma) \cong H^kHom_+(\epsilon_+,\epsilon_+)\cong LCH^{\epsilon_+}_{1-k}(\Lambda_+),$$
which is the Seidel isomorphism (following \cite{NRSSZ}).
This theorem was conjectured by Seidel \cite{Seidel} and was proved by Dimitroglou Rizell \cite{DR}.

If $\Lambda_+$, instead, is empty and $\calA(\Lambda_-)$ has an augmentation $\epsilon_-$, 
the long exact sequence (\ref{les1}) tells us that
$$H^kHom_+(\epsilon_-,\epsilon_-) \cong H^{k+1}(\Sigma, \Lambda_-) =
\begin{cases}
\bbF & \textrm{if} \ k=1,\\
\bbF^{1-\chi(\Sigma)} & \textrm{if } k=0,\\
0 & \textrm{otherwise.} 
\end{cases} $$
However, by  Sabloff duality (\ref{sabdual}),
$$\dim H^0Hom_+(\epsilon_-,\epsilon_-) = \dim H^2Hom_-(\epsilon_-, \epsilon_-) = 
\dim H^2Hom_+(\epsilon_-, \epsilon_-)=0.$$
This is a contradiction since the unit $e_{\epsilon_-}= -y^{\vee}$ is always in $H^0Hom_+(\epsilon_-, \epsilon_-)$ and is not $0$.
Thus if $\Lambda_+$ is empty, then $\Lambda_-$ does not admit any augmentation. 
This result was previously known  by \cite{CDGG, DRcap}.

\begin{rmk}{\label{connect}}
For the rest of the paper, we will focus on the case where $\Lambda_+$ and $\Lambda_-$ are single component knots.
Given the fact that there does not exist a compact Lagrangian manifold in $\bbR\times\bbR^3$ and $\Lambda_-$ does not admit a cap (since $\Lambda_-$ has an augmentation),
we know that any cobordism $\Sigma$ from $\Lambda_-$ to $\Lambda_+$ must be connected.
\end{rmk}
Combining the long exact sequence (\ref{les1}) with the augmentation category map induced by the exact Lagrangian cobordism $\Sigma$, we have the following theorem.

\begin{thm}\label{main1}
Let $\Sigma$ be an exact Lagrangian cobordism  with Maslov number $0$ from a Legendrian knot $\Lambda_-$ to a Legendrian knot $\Lambda_+$.
For $i=1,2$,
assume $\epsilon^i_-$ is an augmentation of  the $\calA(\Lambda_-)$ with a single base point 
and $\epsilon^i_+$ is the  augmentation of $\calA(\Lambda_+)$ induced by $\Sigma$.
Then 
the map $$i^0: H^0Hom_+(\epsilon^1_+, \epsilon^2_+) \to H^0Hom_+(\epsilon^1_-, \epsilon^2_-)$$
in the long exact sequence (\ref{les1}) is an isomorphism.
Moreover, we have that
\begin{equation}\label{hom}
H^*Hom_+(\epsilon^1_+,\epsilon^2_+) \cong H^*Hom_+(\epsilon^1_-,\epsilon^2_-) \oplus \bbF^{-\chi (\Sigma)}[1],
\end{equation}
where $\bbF^{-\chi (\Sigma)}[1]$ denotes the vector space $\bbF^{-\chi (\Sigma)}$ in degree $1$
and $\chi(\Sigma)$ is the Euler characteristic of the surface $\Sigma$.
\end{thm}

\begin{proof}

By Remark \ref{connect}, we have
$$ H^{k}(\Sigma, \Lambda_-) = 
\begin{cases}
\displaystyle{\bbF^{-\chi (\Sigma)} }& \mathrm{if} \ k=1\\
0 & \mbox{if else.} 
\end{cases}
$$

The long exact sequence (\ref{les1})
shows that for $k>1$ or $k<0$,
the map $$i^k: H^kHom_+(\epsilon^1_+,\epsilon^2_+) \to H^kHom_+(\epsilon^1_-,\epsilon^2_-)$$
in the long exact sequence induces an isomorphism 
$$H^kHom_+(\epsilon^1_+,\epsilon^2_+) \cong H^kHom_+(\epsilon^1_-,\epsilon^2_-).$$
When $k=0$ or $1$, we have

\hspace{-0.3in}
{\xymatrixcolsep{.6pc}\xymatrix{
0 \arr &H^0Hom_+(\epsilon^1_+,\epsilon^2_+)\arr^{i^0}& H^0Hom_+(\epsilon^1_-,\epsilon^2_-)\arr &\bbF^{-\chi (\Sigma)} \arr &H^1Hom_+(\epsilon^1_+,\epsilon^2_+) \arr & H^1Hom_+(\epsilon^1_-,\epsilon^2_-)\arr & 0 
}}

For the rest of the proof, we need to show
$$\dim \left(H^0Hom_+(\epsilon^1_+,\epsilon^2_+)\right) \ge \dim \left(H^0Hom_+(\epsilon^1_-,\epsilon^2_-)\right).$$ 
Once this inequality holds, the fact that  $i^0 :  H^0Hom_+(\epsilon^1_+,\epsilon^2_+) \to H^0Hom_+(\epsilon^1_-,\epsilon^2_-) $ is injective implies that it is an isomorphism. 
Note that the long exact sequence (\ref{les1}) is over the field $\bbF$.
It follows that $$H^1Hom_+(\epsilon^1_+,\epsilon^2_+) \cong H^1Hom_+(\epsilon^1_-,\epsilon^2_-) \oplus \bbF^{-\chi (\Sigma)}.$$

To prove the inequality, we exchange the positions of $\epsilon^1$ and $\epsilon^2$ in the long exact sequence (\ref{les1}) and get

\xymatrix{\cdots \arr & H^{k}(\Sigma, \Lambda_-) \arr& H^kHom_+(\epsilon^2_+,\epsilon^1_+) \arr & H^{k}Hom_+(\epsilon^2_-,\epsilon^1_-) \arr & H^{k+1}(\Sigma, \Lambda_-) \arr & \cdots,}
which implies
$$H^2Hom_+(\epsilon^2_+,\epsilon^1_+) \cong H^2Hom_+(\epsilon^2_-,\epsilon^1_-).$$ 
 By Sabloff duality (\ref{sabdual}), we have
$$\dim\left(H^0 Hom_-(\epsilon_{\pm}^1, \epsilon_{\pm}^2)\right)=\dim\left(H^2Hom_+(\epsilon_{\pm}^2,\epsilon_{\pm}^1)\right).$$
Thus $\dim \left(H^0Hom_-(\epsilon^1_+,\epsilon^2_+)\right) = \dim\left( H^0Hom_-(\epsilon^1_-,\epsilon^2_-)\right)$.

Since $\Lambda_+$ and $\Lambda_-$ are both Legendrian knots with a single base point, we have the long exact sequence (\ref{lesknot}) for $\Lambda_+$ and $\Lambda_-$:

{\xymatrixcolsep{1.5pc}
\xymatrix{0 \arr & H^0 Hom_-(\epsilon^1_{\pm},\epsilon^2_{\pm}) \arr &  H^0 Hom_+(\epsilon^1_{\pm},\epsilon^2_{\pm}) \arr & H^0(\Lambda_{\pm})\arr^-{\delta_{\pm}}  & H^1 Hom_-(\epsilon^1_{\pm},\epsilon^2_{\pm}) \arr  & \cdots . }
}

From this long exact sequence, we have
$$
\dim \left(H^0 Hom_+(\epsilon^1_{\pm},\epsilon^2_{\pm}\right) = \dim\left(H^0 Hom_-(\epsilon^1_{\pm},\epsilon^2_{\pm})\right)+ \dim(\ker \delta_{\pm}).$$
Thus, to prove $\dim \left(H^0Hom_+(\epsilon^1_+,\epsilon^2_+)\right) \ge \dim\left( H^0Hom_+(\epsilon^1_-,\epsilon^2_-)\right)$,
we only need to show $\dim(\ker\delta_+) \ge \dim(\ker\delta_-).$

Recall that the cobordism $\Sigma$ from $\Lambda_-$ and $\Lambda_+$  induces an $A_{\infty}$ category map 
$$f: \calA ug_+(\Lambda_-) \to \calA ug _+(\Lambda_+)$$
in the way described in Section \ref{augcat}.
In particular, we get the functor $f_1$ of augmentation categories on the level of morphisms:
$$f_1:Hom_+(\epsilon^1_-, \epsilon^2_-) \to 
Hom_+(\epsilon^1_+, \epsilon^2_+).$$

This map descends to the cohomology level as
$f^*: H^*Hom_+(\epsilon_-^1, \epsilon_-^2) \to H^*Hom_+(\epsilon_+^1, \epsilon_+^2)$.
Notice that $f_1$ sends 
$Hom_-(\epsilon_-^1, \epsilon_-^2)$ to $Hom_-(\epsilon_+^1, \epsilon_+^2)$.
Hence $f_1$ induces a map between the cohomology of the quotient chain complexes, denoted by $f^*$ as well.

We have  the following diagram commutes:
{\xymatrixcolsep{1.2pc}
\diag{0 \arr & H^0 Hom_-(\epsilon^1_+,\epsilon^2_+) \arr&  H^0 Hom_+(\epsilon^1_+,\epsilon^2_+) \arr  & H^0(\Lambda_{+})\arr^-{\delta_+}  & H^1 Hom_-(\epsilon^1_+,\epsilon^2_+) \arr  & \cdot\cdot \cdot \\
 0 \arr & H^0 Hom_-(\epsilon^1_-,\epsilon^2_-) \arr  \aru_{f^*}&  H^0 Hom_+(\epsilon^1_-,\epsilon^2_-) \arr \aru_{f^*}& H^0(\Lambda_{-})\arr^-{\delta_-} \aru_{f^*} & H^1 Hom_-(\epsilon^1_-,\epsilon^2_-) \arr \aru_{f^*} & \cdot\cdot \cdot . }
 }

Thus $f^*(\ker{\delta_-}) \subset \ker\delta_+$.
Furthermore, notice that  $f_1$ sends 
the generator $(y^-)^{\vee}\in C_{Morse}^0(\Lambda_{-})$ to the corresponding $(y^+)^{\vee} \in  C_{Morse}^0(\Lambda_{+})$ and $C^{-1}_{Morse}(\Lambda_+)=0.$
Hence $f^*$ is injective on $H^0(\Lambda_-)$, which implies 
$$\dim(\ker\delta_+) \ge \dim(\ker\delta_-).$$
\end{proof}

If $\epsilon^1_- = \epsilon^2_-= \epsilon_-$ and $\epsilon_-$ comes from a Lagrangian filling $L_-$, then $\epsilon_+$ also comes from the filling $L_+$, which is a concatenation of $\Sigma$ and $L_-$.
By Seidel's isomorphism (following \cite{NRSSZ}), we have $Hom^k_+(\epsilon_{\pm},\epsilon_{\pm}) \cong H^k(L_{\pm})$,
which implies that
$$H^kHom_+(\epsilon_+,\epsilon_+) \cong H^kHom_+(\epsilon_-,\epsilon_-) \textrm{ for } k \neq 1$$
and when $k=1$,
$$H^1Hom_+(\epsilon_+,\epsilon_+) \cong H^1Hom_+(\epsilon_-,\epsilon_-) \oplus \bbF^{-\chi (\Sigma)}.$$

Theorem \ref{main1} is a generalization of Seidel's isomorphism.
Equation (\ref{hom}) holds even if  $\epsilon_-$ does not come from a Lagrangian filling or
$\epsilon^1_- $ and $\epsilon^2_-$ are not the same.

If the two augmentations are the same, we can identify the cohomology of $Hom_+$ space with the linearize contact homology by \cite{NRSSZ}:
$$H^kHom_+ (\epsilon, \epsilon) \cong LCH_{1-k}^{\epsilon}(\Lambda).$$
Now we restate Theorem \ref{main1} in terms of linearized contact homology.
\begin{cor}\label{LCH}
Let $\Sigma$ be an exact Lagrangian cobordism with Maslov number $0$ from a Legendrian knot $\Lambda_-$ to a Legendrian knot $\Lambda_+$.
Assume $\epsilon_-$  is an augmentation of  $\calA(\Lambda_-)$ and $\epsilon_+$ is the augmentation of $\calA(\Lambda_+)$ induced by $\Sigma$.
Then
$$LCH_*^{\epsilon_+}(\Lambda_+)\cong LCH_*^{\epsilon_-}(\Lambda_-) \oplus \bbF^{-\chi (\Sigma)}[0],$$
where $\bbF^{-\chi (\Sigma)}[0]$ denotes the vector space $\bbF^{-\chi (\Sigma)}$ in degree $0$.
\end{cor}
Therefore, if there exists an exact Lagrangian cobordism $\Sigma$ from $\Lambda_-$ to $\Lambda_+$,
the Poincar{\'e} polynomial of linearized contact homology of $\Lambda_+$ agrees with that of $\Lambda_-$ in  all  degrees except $0$.
In degree $0$ their coefficients differ by $-\chi(\Sigma)$.
This gives a strong and computable obstruction to the existence of exact Lagrangian cobordisms.
One can check the Poincar{\'e} polynomials of linearize contact homology for any two Legendrian knots with small crossings through the atlas in \cite{CNatlas}.
If they do not satisfy the relation given in Corollary \ref{LCH}, there does not exist an exact Lagrangian cobordism between them.

\begin{figure}[!ht]
\labellist
\small
\pinlabel $6_1$ at  -10 150
\pinlabel $4_1$ at  370 150
\endlabellist
\includegraphics[width=4.2in]{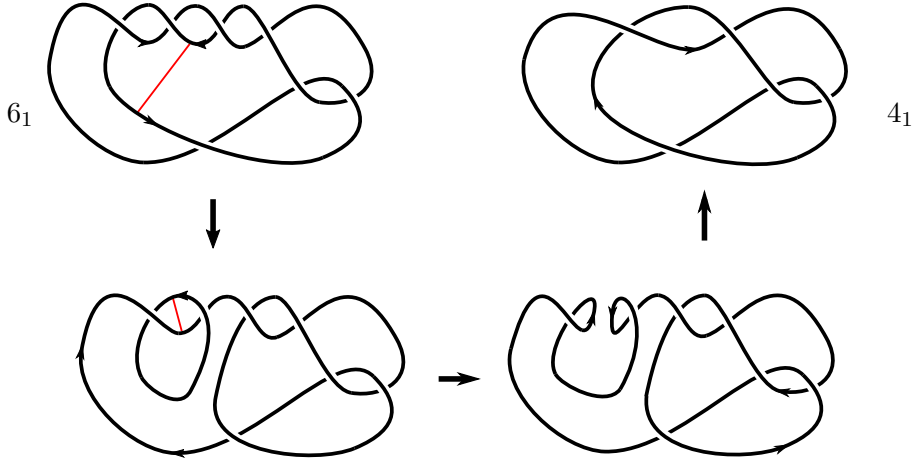}
\caption{The topological cobordism between $6_1$ and $4_1$ can be achieved by  two saddle
moves along the red lines followed by an isotopy.}
\label{topcob}
\end{figure}

For example, 
let $\Lambda_1$ and $\Lambda_2$ be the Legendrian knots with maximum Thurston-Bennequin number of smooth knot type  $4_1$ and $6_1$, respectively (as shown in Figure \ref{4_1vs6_1}).
There is a topological cobordism between $4_1$ and $6_1$ with genus $1$ as shown in
Figure \ref{topcob}.
Moreover, the Thurston-Bennequin numbers of $\Lambda_1$ and $\Lambda_2$ are $-3$ and $-5$, respectively, which satisfy the Thurston-Bennequin number relation (\ref{tb}).
Thus there is a possibility to exist an exact Lagrangian cobordism from $\Lambda_2$ to $\Lambda_1$ with genus $1$.
However, we have the following proposition:
\begin{prop}\label{ex}
There does not exist  an exact Lagrangian cobordism from $\Lambda_2$ to $\Lambda_1$ with Maslov number $0$.
\end{prop}
\begin{proof}
 The Poincar{\'e} polynomials of linearized contact homology for $\Lambda_1$ and $\Lambda_2$ are $t^{-1}+2t$ and $2t^{-1} +3t$, respectively.
As a result of Corollary \ref{LCH}, there does not exist an exact Lagrangian cobordism from $\Lambda_2$ to $\Lambda_1$ with Maslov number $0$.
\end{proof}

\subsection{Geometric description of the differential map.}\label{geointer}

Let $\Sigma$ be an exact Lagrangian cobordism from a Legendrian knot $\Lambda_-$ to a Legendrian knot $\Lambda_+$.
For $i=1,2$, assume $\epsilon^i_-$ is an augmentation of $\calA(\Lambda_-)$ and $\epsilon^i_+$ is the augmentation induced by $\Sigma$.
So far we have two maps from $Hom_+(\epsilon^1_-, \epsilon^2_-)$ to 
$Hom_+(\epsilon^1_+, \epsilon^2_+)$.
One is 
the geometric map $d_{+-}$ in the differential of the Cthulhu chain complex $Cth(\Sigma^1,\Sigma^2)$ defined by counting rigid holomorphic disks with boundary on $\Sigma^1 \cup \Sigma^2$.
The other map is the augmentation category map induced by $\Sigma$ on the level of morphisms $f_1$, defined algebraically in Section \ref{augcat}.
In this section, we will show that, with a choice of Morse function $F$ on the cobordism $\Sigma$, the maps $d_{+-}$ and $f_1$ are the same.
To do that, we describe the two maps separately and then compare their images on each generator of $Hom_+(\epsilon^1_-, \epsilon^2_-)$.

In order to describe $d_{+-}$, we want to interpret  rigid holomorphic disks with boundary on $\Sigma^1 \cup \Sigma^2$ in terms of  rigid holomorphic disks with boundary on $\Sigma$ together with  negative gradient flows of a Morse function. 
This is  analogous to a result  in \cite{EESduality}, which
gives a correspondence between  rigid holomorphic disks with boundary on a $2$-copy of a Legendrian submanifold $L$ and  rigid holomorphic disks with boundary on $L$ together with  negative gradient flows of a Morse function.
Now let us describe the result in \cite{EESduality}  with details.

Let $L$ be a Legendrian submanifold in the contact manifold $(P\times \bbR, \ker(dz-\theta))$, where $(P, d\theta)$ is an  exact symplectic $2n$-dimensional manifold.
Instead of considering  holomorphic disks in the symplectization of $P\times \bbR$ with boundary on $\bbR\times L$, according to  \cite{DR},
we can consider  holomorphic disks in $P$  with boundary on $\pi(L)$,
where $\pi$ is the projection $P \times \bbR \to P$.
See \cite[Section 2.2.3]{EESduality} for the detailed definition of holomorphic disks with boundary on $\pi(L)$.
As the points on $\pi(L)$ and points on $L$ are naturally corresponded except that the double points of $\pi(L)$ correspond to the Reeb chords of $L$, 
we refer the holomorphic disks as {\bf $J$-holomorphic disks with boundary on $L$} as in \cite{EESduality}, where
 $J$ is a generic almost complex structure on $P$.
Choose a Morse-Smale pair $(f, g)$, where $f$ is a Morse function $L \to \bbR$ and $g$ is a Riemannian metric on $L$, such that $(f, g, J)$ is {\bf adjusted to $L$} in the sense of \cite[Section 6.3]{EESduality}.
Push $L$ off through the Morse function $f$ and get a $2$-copy of $L$, denoted by $2L$.
In order to describe rigid holomorphic disks with boundary on $2L$, 
we need to introduce the generalized disks determined by $(f, g, J)$.
A {\bf generalized disk} is  a pair $(u, \gamma)$, where
\begin{itemize}
\item  $u\in \calM$ is a $J$-holomorphic disk with boundary on $L$; 
\item $\gamma$ is a negative gradient flow of $f$ with one end on the boundary of $u$ and the other end at a critical point $p$ of the Morse function $f$;
\item the boundary of $u$ and $\gamma$ intersect transversely.
\end{itemize}
The point $p$ is called a {\bf negative Morse puncture} if the flow line $\gamma$ flows toward $p$,
and is called a {\bf positive Morse puncture} if $\gamma$ flows away from $p$.
The formal  dimension  $ \dim(u, \gamma)$ is defined by
$$
\dim(u, \gamma)=
\begin{cases}
\displaystyle{\dim \calM +1 +Ind_{f}(p)-n },& \mathrm{if} \ p  \textrm{ is a positive Morse puncture,}\\
{\dim \calM +1 - Ind_{f}}(p), & \textrm{if } p \textrm{ is a negative Morse puncture.}
\end{cases}  
$$
The generalized disk $(u, \gamma)$ is called {\bf rigid} if $\dim(u, \gamma)=0$.

The rigid holomorphic disks with boundary on a $2$-copy of $L$ can be described as below in terms of whether their punctures are Morse Reeb chords or non-Morse Reeb chords (as defined in Section \ref{augcat}).
\begin{lem}[{\cite[Theorem 3.6]{EESduality}\label{Legcrsp}}]
Let $(f, g, J)$ be a pair described above that is adjusted to  the Legendrian submanifold $L$.
Push $L$ off through the Morse function $f$ and get a $2$-copy $2L$.
There are bijective correspondences below:
\begin{itemize}
\item Rigid holomorphic disks with boundary on $2L$ that have one positive puncture and one negative puncture at non-Morse mixed Reeb chords  and the other punctures at pure Reeb chords
are in $1-1$ correspondence with   rigid holomorphic disks with boundary on $L$ as shown in Figure \ref{Legdiskcrsp} $(a)$.
\item Rigid holomorphic disks with boundary on $2L$ that have exactly one puncture at a Morse Reeb chord 
are in $1-1$ correspondence with rigid generalized disks $(u, \gamma)$  determined by $(f, g, J)$ as shown in Figure \ref{Legdiskcrsp} $(b)$.
\item Rigid holomorphic disks with boundary on $2L$ that have two punctures at Morse Reeb chords are in $1-1$ correspondence with 
rigid negative gradient flows of the Morse function $f$.
\end{itemize}
\end{lem}
\begin{figure}[!ht]
\labellist
\small
\pinlabel $2L$ at  150 120
\pinlabel $L$ at  210 120
\pinlabel $u$ at 250 70
\pinlabel $(a)$ at 180 0
\endlabellist
\includegraphics[width=4in]{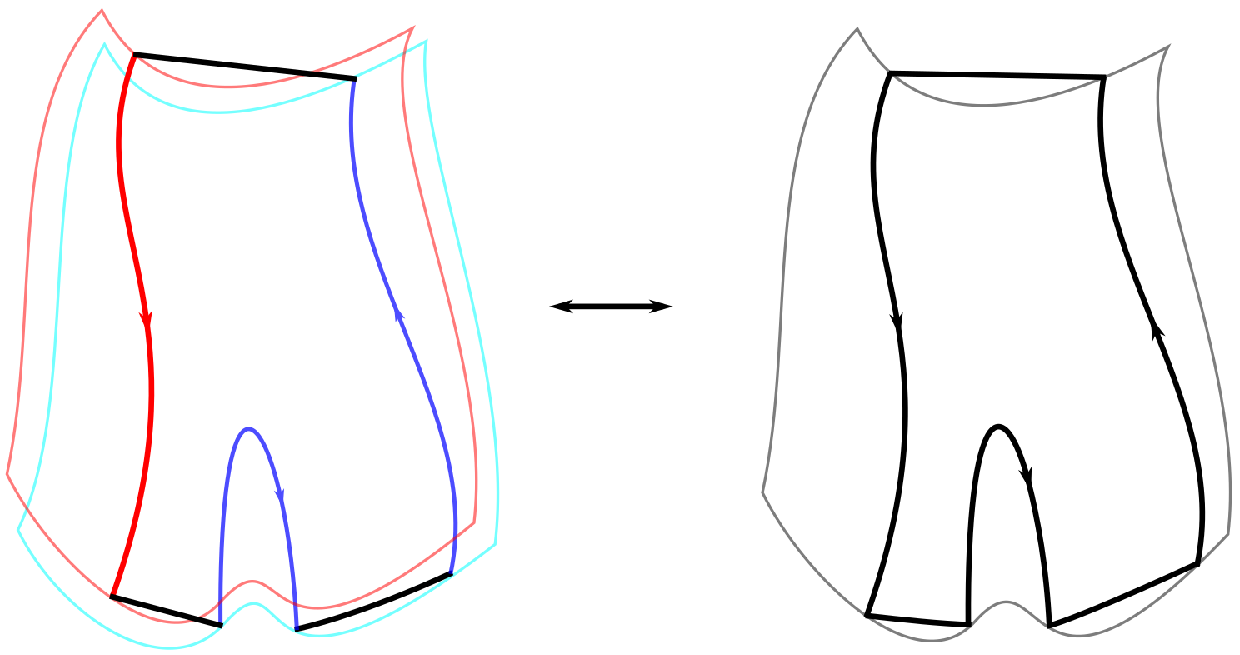}

\vspace{0.1in}

\labellist
\small

\pinlabel $2L$ at  235 170
\pinlabel $L$ at  320 170
\pinlabel $\gamma$ at 393 70
\pinlabel $u$ at 510 120
\pinlabel $(b)$ at 270 0
\endlabellist
\includegraphics[width=4in]{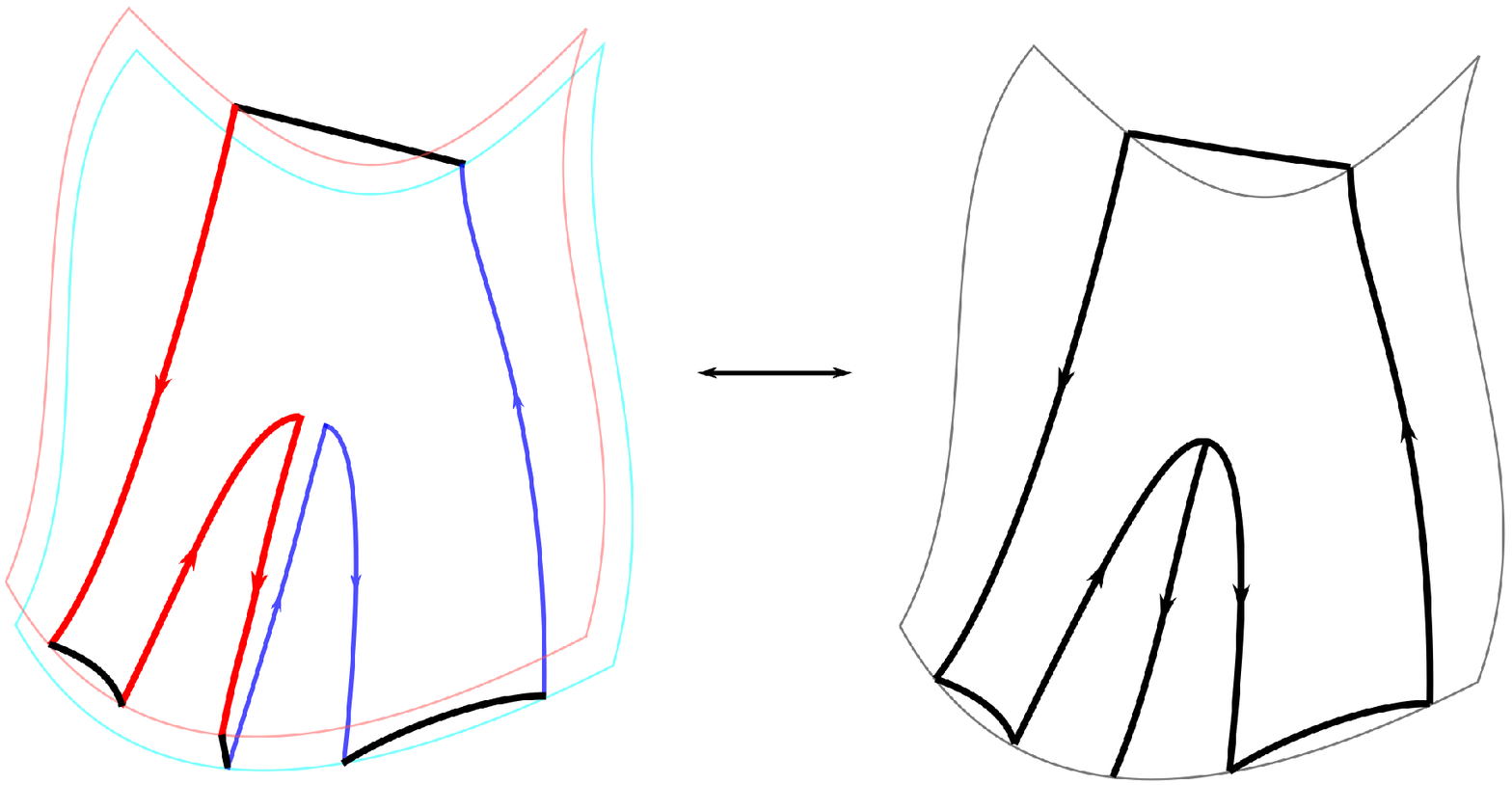}

\vspace{0.1in}

\caption{ Schematic picture of the correspondences in Lemma \ref{Legcrsp}.
The arrows indicate the orientation of  holomorphic disks and the negative gradient flow line.}
\label{Legdiskcrsp}
\end{figure}

In order to get an analogous description for a $2$-copy of $\Sigma$, 
we need a result in \cite{EHK} to relate  rigid holomorphic disks with boundary on a cobordism $\Sigma$ 
to   rigid holomorphic disks with boundary on some Legendrian submanifold $L_{\Sigma}$.

Let us first construct the Legendrian  submanifold $L_{\Sigma}$.
Suppose $\Sigma$ is an exact Lagrangian submanifold in $\big(\bbR \times \bbR^3, d(e^t \alpha)\big)$ from $\Lambda_-$ to $\Lambda_+$.
Assume it is cylindrical outside of $[-N+\delta, N-\delta]\times \bbR^3$, where $\delta$ is  a small positive number.
Under the symplectomorphism 
$$
\begin{array}{rl}
\psi:  \big(\bbR \times \bbR^3, d(e^t \alpha)\big) &\to \big(T^*(\bbR_{>0}\times \bbR), d\theta \big)\vspace{0.1in}\\
(t, x, y, z) &\mapsto \big(( e^t, x ), ( z, e^t y)\big),
\end{array}
$$
the cobordism $\Sigma$ can be viewed as a cobordism in $\big(T^*(\bbR_{>0} \times \bbR), d\theta\big)$, where $\theta$ is the negative Liouville form of the cotangent bundle.
Let $a_-= e^{-N}$ and $a_+=e^{N}$.
There exists a small number $\epsilon>0$ such that $\Sigma$ is cylindrical outside of 
 $T^*\big([a_-+\epsilon, a_+-\epsilon] \times \bbR \big)$.
Chopping off the ends of $\Sigma$, we get a cobordism in $T^*\big( [a_-, a_+] \times \bbR\big)$ with the canonical symplectic form.
Lift it to be a Legendrian submanifold $\overline{\Sigma}$ in the $1$-jet space $J^1\big([a_-, a_+] \times \bbR\big)= T^*\big([a_-, a_+] \times \bbR\big) \times \bbR$.
Near the positive boundary $J^1\big((a_+-\epsilon, a_+] \times \bbR\big)$, the Legendrian $\overline{\Sigma}$ can be parametrized as 
$$
(  s,\ x_{+}(q),\ z_{+}(q),  \ s y_{+}(q),  \ s z_{+}(q)+ B_{+}) = j^1(s z_{+}(q)+B_{+})
$$
for some  constant number $B_+$,
where $s=e^t$ and $(t,q) \in \Sigma\cap \big((\log(a_+-\epsilon),\log a_+] \times \bbR^3\big)=(\log(a_+-\epsilon),\log a_+]  \times \Lambda_{+}$.
Here $s z_{+}(q)+B_{+}$ may not be a function of $(s, x)$. 
However, consider $\{(x_{+}(q), z_{+}(q))| q \in \Lambda_+\}$, which is the front projection of $\Lambda$ to $xz$-plane.
The cusps divide the front diagram of $\Lambda_+$ into pieces.
Note that on each piece $z_{+}(p)$ is a perfect function of $x_{+}(p)$ and at each cusp, the two functions from different pieces match at the cusp.
Therefore, we can write the parametrization as $j^1(s z_{+}(q)+B_{+})$.
Similarly, near the negative boundary $J^1\big( [a_-, a_-+\epsilon) \times \bbR \big)$, the Legendrian $\overline{\Sigma}$ can be parametrized as 
$$
( s,\ x_{-}(q),  \ z_{-}(q),  \ s y_{-}(q), \ s z_{-}(q)+ B_{-}) = j^1(s z_{-}(q)+B_{-}),
$$
where $(s,q) \in [a_-, a_-+\epsilon)\times \Lambda_{-}$ and $B_{-}$ is a constant number.

However, notice that $\overline{\Sigma}$ does not have any Reeb chords.
Therefore, we consider the {\bf Morse Legendrian} $\overline{\Sigma^{Mo}}$,
which is a Legendrian submanifold in $J^1\big([a_-, a_+] \times \bbR\big)$ that 
agrees with $\overline{\Sigma}$ on $J^1\big((a_-+\epsilon, a_+-\epsilon) \times \bbR \big)$.
But
near the $(\pm)$-boundary, the Morse Legendrian can be parametrized as $j^1\big((A_{\pm}\mp(s-a_{\pm})^2)  z_{\pm}(q)\big)$, i.e.
\begin{equation}\label{par}
 \big(s, \ x_{\pm}(q),  \ \mp2(s-a_{\pm}) z_{\pm}(q),\  (A_{\pm}\mp(s-a_{\pm})^2) y_{\pm}(q), \  (A_{\pm}\mp(s-a_{\pm})^2)  z_{\pm}(q)\big),
\end{equation}
where $A_{\pm}$ are positive numbers.
The key property of the Morse Legendrian is that the Reeb chords of $\overline{\Sigma^{Mo}}$ on the $(\pm)$-boundary are in bijective correspondence with  Reeb chords of $\Lambda_{\pm}$, respectively.

There are isotopies from $s z_{\pm}(q)+B_{\pm}$ to $(A_{\pm}\mp(s-a_{\pm})^2)  z_{\pm}(q)$\vspace{0.02in}, respectively,
which induce a diffeomorphism
from $\overline{\Sigma}$ to
the Morse Legendrian $\overline{\Sigma^{Mo}}$.
Extend $\overline{\Sigma^{Mo}}$ to be a Legendrian submanifold $L_{\Sigma}$ in $J^1(\bbR_{> 0} \times \bbR)$ by adding 
$$j^1\big((A_{+}-(s-a_{+})^2)  z_{+}(q)\big)$$ with $(s, q) \in (a_+,\infty)\times \Lambda_{+}$ to the positive boundary and adding $$j^1\big((A_{-}+(s-a_{-})^2)  z_{-}(q)\big)$$ with $(s, q) \in (0,a_-)\times \Lambda_{-}$ to the negative boundary.
In other words, when $s<a_-+\epsilon$ or $s>a_+-\epsilon$, we can parametrize $L_{\Sigma}$ as 
(\ref{par}).
Note that
$$L_{\Sigma} \cap J^1\big( (a_-+\epsilon, a_+-\epsilon) \times \bbR \big)= \overline{\Sigma^{Mo}}.$$
Moreover, according to \cite{EK}, there is a natural bijective correspondence between rigid holomorphic disks with boundary on $L_{\Sigma}$ and rigid holomorphic disks with boundary on $\overline{\Sigma^{Mo}}$.
Combining with a result in \cite{EHK}, we know that rigid holomorphic disks with boundary on an exact Lagrangian cobordism $\Sigma$ are in $1-1$ correspondence with rigid holomorphic disks with boundary on $L_{\Sigma}$ that have positive (resp. negative) punctures at Reeb chords lying the slice $s=a_+$ (resp. $s= a_-$).
The proof of this result can be applied directly to the case of immersed exact Lagrangian submanifolds with cylindrical ends, 
where we only consider the rigid holomorphic disks with punctures on Reeb chords but no double points.
Hence we have the following result for a $2$-copy of $\Sigma$, denoted by  $\Sigma\cup \Sigma'$.
\begin{lem}\label{cobtoLeg}
Let $\Sigma$ and $\Sigma'$ be exact Lagrangian cobordisms from $\Lambda_-$ to $\Lambda_+$ and
from $\Lambda'_-$ to $\Lambda'_+$, respectively.
The Morse Legendrian $L_{\Sigma \cup \Sigma'}$ constructed above is a union of $L_{\Sigma}$ and $L_{\Sigma'}$.  
Moreover,
 rigid holomorphic disks with boundary on $\Sigma\cup \Sigma'$ that have positive (resp. negative) punctures at  Reeb chords  of $\Lambda_+ \cup \Lambda_+'$ (resp. $\Lambda_- \cup \Lambda_-'$) are in 1-1 correspondence with  rigid holomorphic disks with boundary on $L_{\Sigma} \cup L_{\Sigma'}$ that have positive (resp. negative) punctures at the Reeb chords lying in the slice $s=a_+$ (resp. $s=a_-$).  
\end{lem}

Note that the rigid holomorphic disks with boundary on $\Sigma\cup \Sigma'$ considered in Lemma \ref{cobtoLeg} are not all the rigid holomorphic disks since we did not talk about holomorphic disks with punctures at double points.
The disks we considered are the ones counted by $d_{+-}$.

In order to apply Lemma \ref{Legcrsp} to $L_{\Sigma \cup \Sigma'}$ and get the analog correspondences for exact Lagrangian cobordisms, we need to view $L_{\Sigma'}$ as the $1$-jet of a function  
$$\tilde{F}: L_{\Sigma} \to \bbR$$
 in the neighborhood of $L_{\Sigma}$ and show that $\tilde{F}$ is Morse.
To describe the function easily, pull it back to be a function $\Sigma \to \bbR$, denoted by $\tilde{F}$ as well.
Note that $\tilde{F}=F$ on $\Sigma \cap T^*\big( [a_-+\epsilon, a_+-\epsilon] \times \bbR\big)$.

Now let us focus on the part $s \in (a_+-\epsilon, \infty)$. 
Denote $L_{\Sigma} \cap J^1\big((a_+-\epsilon, \infty) \times \bbR \big)$ by $\partial_+(L_{\Sigma})$
and denote $\Sigma \cap T^*\big( (a_+-\epsilon, \infty) \times \bbR \big)$ by $\partial_+(\Sigma)$.
One can check that the Reeb chords from $\partial_+(L_{\Sigma})$ to $\partial_+(L_{\Sigma'})$ are in bijective correspondence with the Reeb chords from $\Lambda_+$ to $\Lambda_+'$ by
a property of Morse Legendrian.
As a result, the only critical points of $\tilde{F}$ on $\partial_+(\Sigma)$ are $(s,q)$, where $s= a_+$ and $f'_+(q)=0$.

Let $\pi_1$ and $\pi_2$ be the natural projections as follows:
$$
\xymatrixcolsep{0.5pc}
\xymatrixrowsep{3pc}
\xymatrix{ & {J^1(\bbR_{> 0}  \times \bbR_x)} \ar[ld]_{\pi_1} \ar[rd]^{\pi_2} & \\
{J^1(\bbR_{> 0})} &  & {J^1(\bbR_x})\\}
$$
First project $\partial_+(L_{\Sigma})$ and $\partial_+(L_{\Sigma'})$ to $J^1(\bbR_x)$.
We have
$$\pi_2(\partial_+(L_{\Sigma}))=(x_{+}(q), (A_+-(s-a_+)^2) y_{+}(q), (A_+-(s-a_+)^2)  z_{+}(q))$$ and
$$\pi_2(\partial_+(L_{\Sigma'}))=(x'_{+}(q), (A_+-(s-a_+)^2) y'_{+}(q), (A_+-(s-a_+)^2)  z'_{+}(q)).$$
Thus for fixed $s \in(a_+-\epsilon, \infty)$, we have $\tilde{F}(s,q)= (A_+-(s-a_+)^2)f_+(q)$, where $f_+= F\big\vert_{\{a_+\}\times \Lambda_+}$.
Second, project $\partial_+(L_{\Sigma})$ and $\partial_+(L_{\Sigma'})$ to $J^1(\bbR_{> 0})$. 
We have
$$\pi_1(\partial_+(L_{\Sigma}))=(s,-2(s-a_+) z_{+}(q), (A_+-(s-a_+)^2)  z_{+}(q)),$$ and
$$\pi_1(\partial_+(L_{\Sigma'}))=(s,-2(s-a_+) z'_{+}(q), (A_+-(s-a_+)^2)  z'_{+}(q)).$$
For a fixed $q\in \Lambda_+$,  the only non-degenerate singularity of $\tilde{F}(s,q)$ is $a_+$.
In particular, it is a local maximum since $z'_{+}(q)> z_{+}(q)$, which comes from the fact that $f_+>0$ as constructed in Section \ref{Pair}.
Therefore, we have 
$$Ind_{\tilde{F}}(a_+, q) =Ind_{f_+}(q)+1.$$
Similarly, on the negative side, denote $F\big\vert_{\{a_-\}\times \Lambda_-}$ as $f_-$.
The critical points of $\tilde{F}$ on $\Sigma \cap T^*\big((-\infty, a_-+\epsilon) \times \bbR\big)$ agree with the critical points of $f_-$ that lie in the slice $s=a_-$. 
Moreover, the indices satisfy $Ind_{\tilde{F}}(a_-, q) =Ind_{f_-}(q)$.
Hence $\tilde{F}$ is a Morse function.

Choose a Riemannian metric $g$ on $\Sigma$  and a generic almost complex structure $J$ on $\bbR\times \bbR^3$ that is adjusted to cylindrical ends such that 
the pair $(\tilde{F}, g, J)$ is adjusted to $L_{\Sigma}$.
Now we can apply Lemma \ref{Legcrsp} to the $2$-copy $L_{\Sigma} \cup L_{\Sigma'}$.

Define a {\bf generalized disk} to be a pair $(u, \gamma)$ consisting of a 
$J$-holomorphic disk $u$ with boundary on $\Sigma$ as defined in Section \ref{sectioncob} 
and a negative gradient flow line $\gamma$ of $\tilde{F}$ with one end on the boundary of $u$ and one end at a critical point $p$ of $\tilde{F}$ such that the boundary of $u$ intersect transversely  with  the negative gradient flow $\gamma$.
The point $p$ is called a {\bf negative Morse puncture} if the flow line $\gamma$ flows toward $p$,
and is called a {\bf positive Morse puncture} if $\gamma$ flows away from $p$.
The formal  dimension $\dim(u, \gamma)$ is defined by
\begin{equation}\label{dim}
\dim(u, \gamma)=
\begin{cases}
\displaystyle{\dim \calM +1 +Ind_{f}(p)-2 },& \mathrm{if} \ p  \textrm{ is a positive Morse puncture,}\\
{\dim \calM +1 - Ind_{f}}(p), & \textrm{if } p \textrm{ is a negative Morse puncture.}
\end{cases}  
\end{equation}

The generalized disk $(u, \gamma)$ is called {\bf rigid} if $\dim(u, \gamma)=0$.
We have the following result that is analogous to Lemma \ref{Legcrsp}.
\begin{thm}\label{cobcrsp}
Let $\Sigma$ be an exact Lagrangian cobordism in $\big(\bbR \times \bbR^3, d(e^t \alpha)\big)$ from $\Lambda_-$ to $\Lambda_+$ and
is cylindrical outside of $[-N+\delta, N-\delta] \times \bbR^3$. 
Let $F: \Sigma \to \bbR_{> 0}$ be a positive Morse function.
Push $\Sigma$ off through $F$ and get a new cobordism $\Sigma'$.

Denote $F\big\vert_{\{N\}\times\Lambda_{+}}$  by $f_+$ and $F\big\vert_{\{-N\}\times\Lambda_{-}}$  by $f_-$.
Define a new Morse function $\tilde{F}: \Sigma \to \bbR$ satisfying the following properties:
\begin{itemize}
\item The Morse function $\tilde{F}= F$ on $\Sigma \cap \big([-N+\delta, N-\delta] \times \bbR^3\big)$.
\item On $\Sigma \cap \big( (N-\delta, \infty)\times \bbR^3\big)$, all the critical points of  $\tilde{F}$ lie in $\Sigma \cap ( \{N\}\times \bbR^3)= \{N\}\times \Lambda_+$ and agree with the critical points of $f_+$. Moreover, at each critical point $c$, we have $Ind_{\tilde{F}}c= Ind_{f_+}c+1$;
\item On $\Sigma \cap \big(( -\infty, -N+\delta)\times \bbR^3\big)$, all the critical points of  $\tilde{F}$ lie in $\Sigma \cap ( \{-N\}\times \bbR^3)= \{-N\} \times \Lambda_-$ and agree with the critical points of $f_-$.
Moreover,  at each critical point $c$, we have $Ind_{\tilde{F}}c= Ind_{f_-}c$.
\end{itemize}
The Riemannian metric $g$ and almost complex structure $J$ are chosen as above.
Then we can describe the rigid holomorphic disks with boundary on $\Sigma \cup \Sigma'$ 
that have punctures on Reeb chords of $\Lambda_+\cup \Lambda_+'$ and $\Lambda_-\cup \Lambda_-'$
in terms of whether the Reeb chords are Morse or non-Morse as defined in Section \ref{augcat}.
\begin{enumerate}
\item Rigid holomorphic disks with boundary on $\Sigma \cup \Sigma'$ that have two punctures 
at  non-Morse mixed Reeb chords 
are in $1-1$ correspondence with
 rigid holomorphic disks with boundary on $\Sigma$. See Figure \ref{cobdiskcrsp} $(a)$. 
\item Rigid holomorphic disks with boundary on $\Sigma \cup \Sigma'$ that have exactly one puncture  
at a Morse Reeb chord
are in $1-1$ correspondence with
 rigid generalized disks $(u, \gamma)$ determined by $(\tilde{F}, g, J)$. See Figure \ref{cobdiskcrsp} $(b)$.
\item Rigid holomorphic disks with boundary on $\Sigma \cup \Sigma'$ with two punctures at Morse Reeb chords 
are in $1-1$ correspondence with
rigid negative gradient flows of the Morse function $\tilde{F}$ from a critical point on $\Lambda_+$ to a critical point on $\Lambda_-$.
\end{enumerate}

\end{thm}

\begin{figure}[!ht]
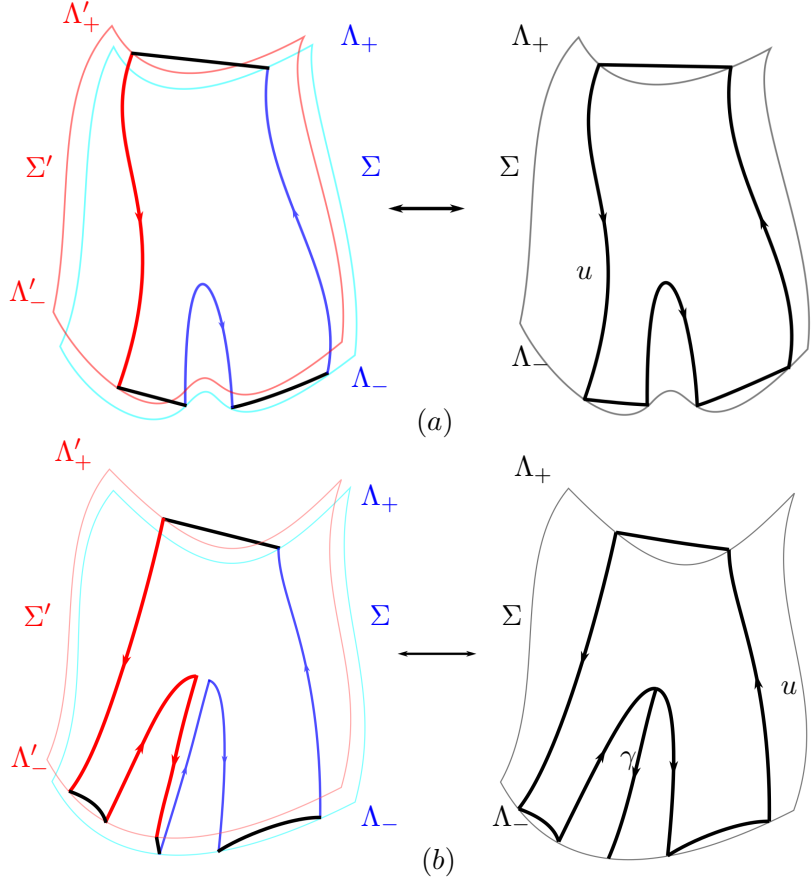

\labellist
\small
{\color{red}
\pinlabel $\Lambda_+'$ at 15 190
\pinlabel $\Sigma'$ at -5 120
\pinlabel $\Lambda_-'$ at -10 60
}
{\color{blue}
\pinlabel $\Sigma$ at  150 120
\pinlabel $\Lambda_+$ at 145 180
\pinlabel $\Lambda_-$ at 150 20
}
\pinlabel $\Sigma$ at  215 120
\pinlabel $\Lambda_+$ at 225 180
\pinlabel $\Lambda_-$ at 225 30
\pinlabel $u$ at 250 70
\pinlabel $(a)$ at 180 0
\endlabellist
\includegraphics[width=4in]{cobmixreeb}

\vspace{0.2in}

\labellist
\small

{\color{red}
\pinlabel $\Sigma'$ at -5 170
\pinlabel $\Lambda_+'$ at 20 280
\pinlabel $\Lambda_-'$ at -10 70
}
{\color{blue}
\pinlabel $\Sigma$ at  230 170
\pinlabel $\Lambda_+$ at 230 250
\pinlabel $\Lambda_-$ at 230 30
}
\pinlabel $\Sigma$ at  320 170
\pinlabel $\Lambda_+$ at 335 270
\pinlabel $\Lambda_-$ at 320 30
\pinlabel $\gamma$ at 400 70
\pinlabel $u$ at 510 120
\pinlabel $(b)$ at 270 0
\endlabellist
\includegraphics[width=4in]{cobcritical}
\vspace{0.2in}

\caption{ Schematic picture of the correspondences in Theorem \ref{cobcrsp}.
The arrows denote the orientation of holomorphic disks and the negative gradient flow line.}
\label{cobdiskcrsp}
\end{figure}
\begin{proof}

According to Lemma \ref{cobtoLeg}, rigid holomorphic disks with boundary on $\Sigma \cup \Sigma'$ that have two punctures at mixed Reeb chords correspond to  rigid holomorphic disks with boundary on $L_{\Sigma} \cup L_{\Sigma'}$ that have positive (resp. negative) boundary at  mixed Reeb chords  lying in the slice $s=a_+$ (resp. $s=a_-$).  
By Lemma \ref{Legcrsp},
these disks with boundary on $L_{\Sigma} \cup L_{\Sigma'}$ are in 1-1 correspondence with  holomorphic disks with boundary on $L_{\Sigma}$ that have positive (resp. negative) boundary at the Reeb chords lying in the slice $s=a_+$ (resp. $s=a_-$)  together with Morse flow lines of $\tilde{F}$.
If it is a rigid Morse flow line of $\tilde{F}$ on $L_{\Sigma}$, it flows from a critical point on $\Lambda_+$ to a critical point on $\Lambda_-$. Pull it back to be a flow line  $\Sigma$ and get the correspondence $(3)$.
If it is a rigid holomorphic disk with boundary on $L_{\Sigma}$ that has positive (resp. negative) boundary at the Reeb chords lying in the slice $s=a_+$ (resp. $s=a_-$), 
by \cite{EHK}, it corresponds to a rigid holomorphic disk with boundary on $\Sigma$, which is the respondence $(1)$.
Otherwise, it it is a rigid generalized disk $(u, \gamma)$ determined by $(\tilde{F}, g, J)$.
From the construction of $\tilde{F}$, one can note that all the critical points of $\tilde{F}$ on $\Sigma\cap(\{N\}\times \bbR^3)$ are of index $1$ or $2$ while all the critical points of $\tilde{F}$ on $\Sigma\cap(\{-N\}\times \bbR^3)$ are of index $0$ or $1$.
By the dimension formula (\ref{dim}), the generalized disk $(u, \gamma)$  is rigid if and only if $u$ is a rigid holomorphic disk.
Each rigid holomorphic disk $u$ with boundary on $L_{\Sigma}$ that has positive (resp. negative) boundary at the Reeb chords lying in the slice $s=a_+$ (resp. $s=a_-$) 
 in turn corresponds to a rigid holomorphic disk with boundary on $\Sigma$.
 Pulling $\gamma$ back to $\Sigma$ and get a rigid generalized disk on $\Sigma$ determined by $(\tilde{F}, g, J)$.
 Hence we get  the correspondence $(2)$.
\end{proof}

Recall that $f_1$ is defined algebraically as follows.
The exact Lagrangian cobordism $\Sigma$ from a Legendrian knot $\Lambda_-$ to a Legendrian knot $\Lambda_+$ induces a DGA map $\phi$ between the DGA's with a single base point
by counting rigid holomorphic disks with boundary on $\Sigma$:
$$\phi: (\calA(\Lambda_+), \partial) \to (\calA(\Lambda_-), \partial),$$
as described in Section \ref{sectioncob}.
This DGA map $\phi$ induces an $A_{\infty}$ category map 
$$f: \calA ug_+(\Lambda_-) \to \calA ug _+(\Lambda_+)$$
in the way described in Section \ref{augcat}.
Restrict the category map on the level of morphisms, we have
$$f_1:Hom_+(\epsilon^1_-, \epsilon^2_-) \to 
Hom_+(\epsilon^1_+, \epsilon^2_+).$$
See calculation (\ref{f_1}) for the explicit formula.

\begin{thm}\label{equal}
With a choice of  Morse function $F:\Sigma \to \bbR$, we have $d_{+-}=f_1$.
\end{thm}
\begin{proof}
We show $d_{+-}=f_1$ by checking their images on generators of $Hom_+(\epsilon^1_{-},\epsilon^2_{-})$.
Recall that
$Hom_+(\epsilon^1_{-},\epsilon^2_{-})$ are generated by  the elements in
$ Hom_-(\epsilon^1_{-},\epsilon^2_{-})$ that correspond to non-Morse Reeb chords and the elements in
$T=\{x_{-}^{\vee},y_{-}^{\vee}\}$ that correspond to Morse Reeb chords, respectively. 

First consider the element $b^{\vee}$ in $Hom_-(\epsilon_-^1, \epsilon_-^2)$.
Notice that Morse Reeb chords  are much shorter than non-Morse Reeb chords.
The energy restriction ensures that $d_{+-}(b^{\vee})$ does not include any element in $T$.
Therefore $d_{+-}$ sends $b^{\vee}$ to $a^{\vee} \in Hom_-(\epsilon^1_+,\epsilon^2_+)$
by counting rigid holomorphic disks $u \in \calM(a^{12}; {\bf p^{11}}, b^{12}, {\bf q^{22}})$  with boundary on $\Sigma^1 \cup \Sigma^2$,
 where ${\bf p^{11}}$ and ${\bf q^{22}}$ are words of pure Reeb chords of $\Lambda^1_-$ and $\Lambda^2_-$, respectively.
According to the correspondence  $(1)$ in Theorem \ref{cobcrsp},
these disks correspond to rigid holomorphic disks $u \in  \calM(a; {\bf p}, b, {\bf q})$ with boundary on $\Sigma$ (as shown in Figure \ref{mixreeb}), 
which are the disks counted in $f_1$.
Notice that both $d_{+-}$ and $f_1$  sends $b^{\vee}$ to $|\calM(a; {\bf p}, b, {\bf q})| \epsilon^1_-({\bf p}) \epsilon^2_-({\bf q}) a^{\vee}$,
where $|\calM(a; {\bf p}, b, {\bf q})|$ is the number of rigid disks in $\calM$ counted with sign. 
Hence the definition of $d_{+-}$ matches the definition of $f_1$ on $Hom_-(\epsilon^1_-,\epsilon^2_-)$.

\begin{figure}[!ht]
\labellist
\small
\pinlabel $a^{12}$ at  68 135
\pinlabel $a$ at  260 135
\pinlabel ${\bf p^{11}}$ at 15 -5
\pinlabel ${\bf p}$ at 210 -5
\pinlabel $b^{12}$ at 65 -5
\pinlabel $b$ at 255 -5
\pinlabel $\bf q^{22}$ at 115 -5 
\pinlabel $\bf q$ at 305 -5
{\color{red} 
\pinlabel $\Sigma_1$ at 10 80
}
{\color{blue}
\pinlabel $\Sigma_2$ at  120 80
}
\pinlabel $\Sigma$ at  205 80

\endlabellist
\includegraphics[width=4in]{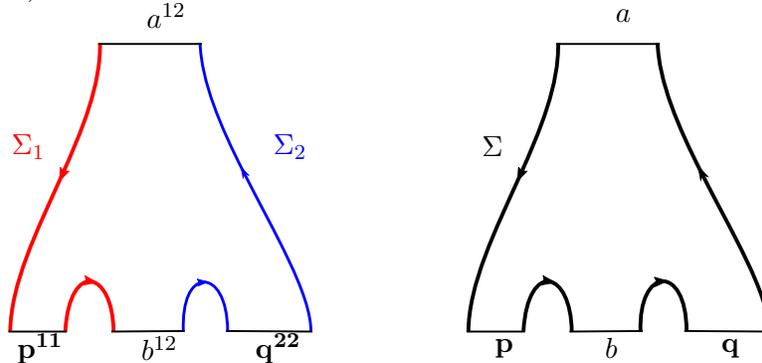}
\caption{ The disk on the left is counted in $d_{+-}$ while the disk on the right is counted in $\phi$.}
\label{mixreeb}
\end{figure}

In order to simplify the map $d_{+-}$,
we can choose a Morse function $F$ such that
the negative gradient flow of $F$ flows from $\ast_{+}$  directly to $\ast_-$ without going through any critical points.
We can further require that the negative gradient flow of  $F$ behave the same in a collar neighborhood of the flow line from $\ast_{+}$ to $\ast_-$ as shown in Figure \ref{strip}.
As $x_{\pm}$ and $y_{\pm}$ sit right besides $\ast_{\pm}$,
the negative gradient flow lines of $\tilde{F}$ flow from $x_{+}$ and $y_{+}$ directly to $x_-$ and $y_{-}$, respectively.

\begin{figure}[!ht]
\labellist
\small
\pinlabel $\Lambda_+$  at 145 155
\pinlabel $\Lambda_-$  at 175 25
\pinlabel $\Sigma$  at 160 85
\pinlabel $\ast_+$  at 41 143
\pinlabel $x_+$ at 55 140
\pinlabel $y_+$ at 69 137
\pinlabel $\ast_-$  at 26 10
\pinlabel $\alpha$  at 25 80
\pinlabel $x_-$ at 40 5
\pinlabel $y_-$ at 54 0
\endlabellist

\includegraphics[width=1.7in]{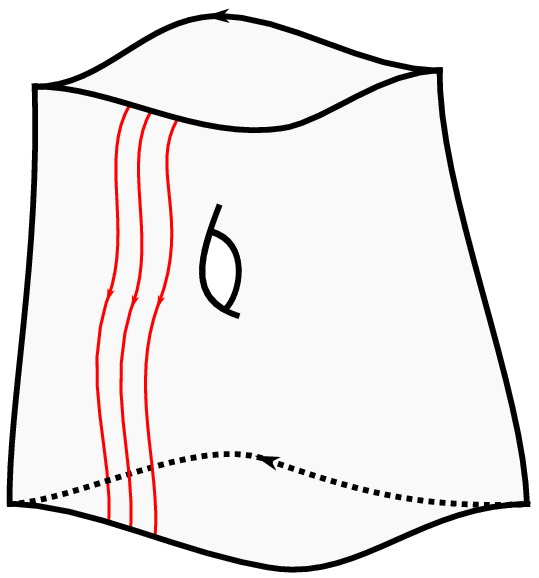}
\caption{An example of  Morse flows of $F$.}
\label{strip}
\end{figure}

For the element $c^{\vee} \in T$, the map $d_{+-}$ counts the rigid holomorphic disks in $\Sigma_1 \cup \Sigma_2$ that have a negative puncture at the Morse Reeb chord $c$.
For the rigid disk that has a positive puncture at a Morse Reeb chord as well, 
according to the  correspondence $(3)$ in Theorem \ref{cobcrsp},
it  corresponds to a rigid Morse flow line of $\tilde{F}$.
The indices of $\tilde{F}$ on $y_-$, $x_-$, $y_+$ and $x_+$ are $0$, $1$, $1$ and $2$ respectively.
Therefore $d_{+-}(x_-^{\vee})$ has $x_+^{\vee}$ as a term and $d_{+-} (y_-^{\vee})$ has $y_+^{\vee}$ as a term.
If the rigid disk that has a positive puncture at a non-Morse mixed Reeb chord $a^{12}$,
we denote it by  $u \in \calM(a^{12}; {\bf p^{11}}, c^{12}, {\bf q^{22}})$.
By the correspondence $(2)$ in Theorem \ref{cobcrsp},
it corresponds to a rigid generalized disk $(u, \gamma)$,
 where $u \in \calM(a; {\bf p}, {\bf q})$ is a holomorphic disk with boundary on $\Sigma$ and
 $\gamma$ is a Morse flow of $\tilde{F}$ that flows towards $c$
(see Figure \ref{critical}).
 Due to the dimension formula (\ref{dim}) of generalized disks,
 no rigid disk has a negative puncture at $y_-$ since $Ind_{\tilde{F}} y_-=0$ but $\dim \calM \ge 0$.
 Hence $d_{+-} (y_-^{\vee}) = y_+^{\vee}$, which matches the definition of $f_1$ on $y_k^-$.

\begin{figure}[!ht]
\labellist
\small
\pinlabel $a^{12}$ at 70 135
\pinlabel $a$ at 260 135
\pinlabel $p^{11}$ at 15 -5
\pinlabel $c^{12}$ at 65 -5
\pinlabel $q^{22}$ at 115 -5
\pinlabel $p$ at 210 -5
\pinlabel $c$ at 255 -5
\pinlabel $q$ at 310 -5
{\color{red}
\pinlabel $\Sigma_1$ at 10 80
}
{\color{blue}
\pinlabel $\Sigma_2$ at 115 80
}
\pinlabel $\Sigma$ at 205 80

\pinlabel $\gamma$ at 265 35
\endlabellist
\includegraphics[width=4in]{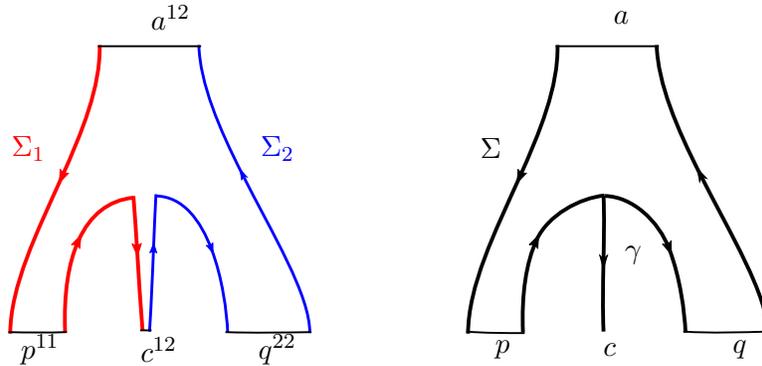}
\caption{The disk on the left is counted in $d_{+-}$ while the disk on the right is counted by $\phi$.}
\label{critical}
\end{figure}

For the element $x_-^{\vee}$, 
we know that $f_1(x_+^{\vee})$ counts the element $a^{\vee}$ if $t$ shows up in the image of the DGA map $\phi(a)$.
In other words,
there exists a rigid holomorphic disk $u \in \calM(a; {\bf p }, {\bf q})$ with boundary on $\Sigma$, where ${\bf p }$ and ${\bf q}$ are words of pure Reeb chords of $\Lambda_-$, such that $u$ has a nontrivial intersection number with  $\alpha$, where $\alpha$ is the curve from the base point $\ast_+$ to $\ast_-$.
Each rigid holomorphic disk $u$ contributes to $f_1(x_+^{\vee})$ a term of $a^{\vee}$ with  coefficient $s(u, \alpha)\epsilon_-^1({\bf p}) \epsilon^2_-(\bf{q})a^{\vee},$ 
where $s(u, \alpha)$ is the intersection number of the boundary of $u$ and $\alpha$.
We can make the Morse function $F$ satisfy the property that the
negative gradient flow line $\gamma$ of $F$ from $x_+$ to $x_-$  is parallel to $\alpha$ and of the same orientation.
For each intersection point $p_i$ of the boundary of $u$ and  $\gamma$, denote the part of $\gamma$ from $p_i$ to $c$ by $\gamma_{i}$.
By the correspondence $(3)$ in Theorem \ref{cobcrsp},
the rigid generalized disk $(u, \gamma_i)$ corresponds to a rigid holomorphic disk in $\calM(a^{12}; {\bf p^{11}}, c^{12}, {\bf q^{22}})$ with boundary on $\Sigma^1\cup\Sigma^2$,
and hence contributes to  $d_{+-}(x_-^{\vee})$ with a term $s(u, \gamma_i)\epsilon_-^1({\bf p^{11}}) \epsilon^2_-(\bf{q^{22}})a^{\vee},$ 
where  
$s(u, \gamma_i)$ is  the sign of the intersection.
Sum over all the intersections of  the boundary of $u$ and $\gamma$,
the rigid holomorphic disk $u$ contributes $s(u, \gamma)\epsilon_-^1({\bf p}) \epsilon^2_-(\bf{q})a^{\vee}$ to  $d_{+-}(x_-^{\vee})$, where $s(u, \gamma)$ is the 
  intersection number of   the boundary of $u$ and $\gamma$.
Therefore, we have 
$d_{+-}=f_1$  on $x_-^{\vee}$.
\end{proof}

\subsection{Aside.}\label{aside}
In this section, we  describe  the differential map of the Cthulhu chain complex in terms of holomorphic disks with boundary on $\Sigma$ and Morse flow lines.
This allows us to recover the long exact sequences in \cite{CDGG}.
The theorem in this section is stated without rigorous proof. 
But it will not be used in the other part of the paper.

In Section \ref{geointer}, we only need  to describe  the rigid disks with boundary on $\Sigma^1 \cup \Sigma^2$ that have punctures at  Reeb chords. 
Hence we only have correspondences for those types of disks. 
However, the method should work for all the rigid holomorphic disks with boundary on $\Sigma^1\cup \Sigma^2$ including the disks counted by $d_{+0}$ and $d_{0-}$.
We state the following theorem without proof.

\begin{thm}
Let $\Sigma$ be an exact Lagrangian cobordism from $\Lambda_-$ to $\Lambda_+$
and $\Sigma^1\cup \Sigma^2$ be a $2$-copy of $\Sigma$ as constructed in Section \ref{Pair}.
For $i=1,2$, assume $\epsilon^i_-$ is an augmentation of $\calA(\Lambda_-)$ and $\epsilon^i_+$ is the augmentation of $\calA(\Lambda_+)$ induced by $\Sigma$.
For $\eta$ small enough,
 the Cthulhu chain complex can be decomposed into five parts:
$$Cth^k(\Sigma^1, \Sigma^2)=
Hom_-^{k-1}(\epsilon^1_+,\epsilon^2_+) \oplus
C_{Morse}^{k-1} f_+ \oplus C_{Morse}^k F  \oplus Hom^k_-(\epsilon^1_-,\epsilon^2_-) \oplus C_{Morse}^k f_-.$$
Under this decomposition, the differential can be written as 
$$d=
\begin{pmatrix}
	m_1& d_{+f_{+}} & d_{+F} & d_{+-} & d_{+f_-} \\
	0 & d_{f_+} & d_{f_+ F} & 0 & d_{f_+ f_-}\\
	0 & 0 & d_F & 0 &  d_{F f_-} \\
	0 &  0 & 0 &  m_1 & d_{-f_-}\\
	0 & 0 & 0 & 0 & d_{f_-}\\
	\end{pmatrix}
$$
Moreover,
\begin{enumerate}
\item The holomorphic disks counted by $d_{+F}$ and $d_{+f_-}$ are in $1-1$ correspondence with rigid generalized disks on $\Sigma$ determined by $(\tilde{F}, g, J)$.
\item The holomorphic disks counted by $d_{+-}$ are in $1-1$ correspondence with rigid holomorphic disks with boundary on $\Sigma$.
\item The holomorphic disks counted by $d_{f_+F}$, $d_{f_+f_-}$ and $d_{Ff_-}$  are in $1-1$ correspondence with  rigid Morse flow lines of $\tilde{F}$. 
\end{enumerate}
\end{thm}

This theorem is similar to the conjectural analytic Lemma in \cite[Lemma 4.11]{Ekh}, which describes the correspondence in the case of  exact Lagrangian fillings.

We next discuss how to recover the three long exact sequences in \cite{CDGG} from this chain complex.
\begin{enumerate}
\item
Decompose the Cthulhu chain complex as
$$Hom_-^{k-1}(\epsilon^1_+,\epsilon^2_+)\oplus (C_{Morse}^{k-1}f_+\oplus C_{Morse}^kF)\oplus Hom_+^{k}(\epsilon^1_-,\epsilon^2_-).$$
Notice that the chain complex 
$$\left(C_{Morse}^{k-1}f_+\oplus C_{Morse}^kF, \ 
\begin{pmatrix}
d_{f_+} & d_{f_+ F}\\
0 & d_F\\
\end{pmatrix}
\right)$$
can be identified with the Morse co-chain complex $(C^k_{Morse} \bar{F}, d_{\bar{F}})$ induced by a Morse function $\bar{F}$, where $\bar{F}$ agrees with $\tilde{F}$ near  $\Lambda_+$ and agrees with $F$ for the rest part.
Hence $$H^*(C^k_{Morse} \bar{F}, d_{\bar{F}})= H^k(\Sigma, \Lambda_+ \cup \Lambda_-).$$
Therefore, we have the following long exact sequence:

{
\xymatrixcolsep{0.8pc}
\xymatrix{\cdots \arr & H^{k}(\Sigma, \Lambda_+\cup \Lambda_-)\arr&  H^{k}Hom_-(\epsilon^1_+,\epsilon^2_+) \arr & H^{k}Hom_+(\epsilon^1_-,\epsilon^2_-) \arr  & H^{k+1}(\Sigma, \Lambda_+\cup \Lambda_-)\arr&\cdots,}
}
which is  Theorem 1.5 of \cite{CDGG}.
\item 
View  the Cthulhu chain complex as a direct sum of $Hom_-^{k-1}(\epsilon^1_+,\epsilon^2_+)$,  
$
C_{Morse}^{k-1}f_+\oplus C_{Morse}^k F \oplus Hom^k_-(\epsilon^1_-,\epsilon^2_-)$ and 
the Morse co-chain complex $C_{Morse}^k f_-$.
We have the following long exact sequence:

{
\xymatrixcolsep{0.8pc}
\xymatrix{\cdots \arr & H^{k-1}( \Lambda_-) \arr &H^{k}(\Sigma, \Lambda_+\cup \Lambda_-) \oplus H^{k}Hom_-(\epsilon^1_-,\epsilon^2_-) \arr &  H^{k}Hom_-(\epsilon^1_+,\epsilon^2_+) \arr&
 H^{k}( \Lambda_-) \arr &\cdots,}
}
which is Theorem 1.6 in \cite{CDGG}.
\item Rewrite the Cthulhu chain complex as
$$
\begin{array}{rl}
Cth^k(\Sigma^1, \Sigma^2) & =
Hom_-^{k-1}(\epsilon^1_+,\epsilon^2_+) \oplus Hom^k_-(\epsilon^1_-,\epsilon^2_-) \oplus
C_{Morse}^{k-1} f_+\oplus C_{Morse}^k F  \oplus C_{Morse}^k f_-\\
&\\
&= Hom_-^{k-1}(\epsilon^1_+,\epsilon^2_+) \oplus Hom^k_-(\epsilon^1_-,\epsilon^2_-) \oplus
C^k_{Morse} \tilde{F}
\end{array}$$
with the differential
$$d=
\begin{pmatrix}
	m_1& * &* \\
	0 & m_1 & *\\
	0 & 0 & d_{\tilde{F}} \\
	\end{pmatrix}.
$$
We have the long exact sequence:

{
\xymatrixcolsep{1pc}
\xymatrix{\cdots \arr & H^{k}Hom_-(\epsilon^1_-,\epsilon^2_-) \arr& H^{k}Hom_-(\epsilon^1_+,\epsilon^2_+) \arr & H^{k+1}(\Sigma, \Lambda_+) \arr & H^{k+1}Hom_-(\epsilon^1_-,\epsilon^2_-) \arr &\cdots,}
}
which is Theorem 1.4 in \cite{CDGG}.
\end{enumerate}

One may get other long exact sequences from the Cthulhu chain complex above. 
One example is : 

{
\xymatrixcolsep{1pc}
\xymatrix{\cdots \arr & H^{k}Hom_-(\epsilon^1_-,\epsilon^2_-) \arr& H^{k}Hom_+(\epsilon^1_+,\epsilon^2_+) \arr & H^{k}(\Sigma) \arr & H^{k+1}Hom_-(\epsilon^1_-,\epsilon^2_-) \arr &\cdots,}
}
which is obtained by decomposing the Cthulhu chain complex as  a direct sum of $Hom_+^{k-1}(\epsilon^1_+,\epsilon^2_+)$,
$Hom^k_-(\epsilon^1_-,\epsilon^2_-)$ and $C^k_{Morse} F\oplus C^k_{Morse} f_-$.

\subsection{Injectivity}\label{injection}
Theorem \ref{equal} implies that $d_{+-}$ is a chain map and hence
it gives the Cthulhu chain complex a stronger algebraic structure. 
In this section, we use these algebraic information to 
 deduce that  the augmentation category map
  induced by the exact Lagrangian cobordism $\Sigma$ 
  is injective on the level of equivalence classes of objects.
  And its induced map
  on the  cohomology category $H^*\calA ug_+$ is faithful.

Notice that $d_{+-} = f_1$ implies that $d_{+-}$ is a chain map and thus induces maps $d_{+-}^k: H^k Hom_+(\epsilon^1_-, \epsilon^2_-) \to H^k Hom_+(\epsilon^1_+, \epsilon^2_+)$ for $k\in \bbZ$.
Then we have the following theorem deduced from the double cone structure of the Cthulhu chain complex.
We would like to thank the referee for pointing out this theorem.

\begin{thm}\label{thm 1}
Let $\Sigma$ be an exact Lagrangian cobordism from a Legendrian knot $\Lambda_-$ to a Legendrian knot $\Lambda_+$ with Maslov number $0$.
For $i=1,2$, assume $\epsilon^i_-$ is an augmentation of $\calA(\Lambda_-)$ and $\epsilon^i_+$ is the  augmentation of $\calA(\Lambda_+)$ induced by $\Sigma$.
With the same choice of Morse function as in Theorem \ref{equal}, we have the following statement.

For fixed $k \in \bbZ$,
the map $$i^k: H^k Hom_+(\epsilon^1_+, \epsilon^2_+) \to H^k Hom_+(\epsilon^1_-, \epsilon^2_-)$$ 
in the long exact sequence (\ref{les1}) is injective (resp. surjective) if and only if
the map $$d_{+-}^k: H^k Hom_+(\epsilon^1_-, \epsilon^2_-) \to H^k Hom_+(\epsilon^1_+, \epsilon^2_+)$$ 
is surjective (resp. injective).
\end{thm}

\begin{proof}
We will first prove that $i^k$ is surjective if and only if $d_{+-}^k$ is injective for fixed $k$.

Consider the Cthulhu chain complex as a mapping cone of $\Phi: Hom_+(\epsilon^1_-,\epsilon^2_-) \to Cone (d_{+0})$, where $\Phi = d_{+-} + d_{0-}$.
The trivial cohomology of the Cthulhu chain complex implies that $\Phi$ induces isomorphisms
$$\Phi^k: H^k Hom_+(\epsilon^1_-,\epsilon^2_-) \to H^kCone (d_{+0}) \textrm{ for } k\in\bbZ .$$
We have the following diagram commutes.
\diag{
\cdots\arr & {H^k Hom_+(\epsilon^1_+,\epsilon^2_+) }\arr^{i_k} & H^k Cone (d_{+0})\arr^{j_k} &{ H^{k+1}(C_{\textrm{Morse}} F)} \arr
 &\cdots\\
  & & H^kHom_+(\epsilon^1_-,\epsilon^2_-)  \aru^{\Phi^k}_{\cong} \ar[ur]_{d^k_{0-}}
} 

Comparing to the long exact sequence (\ref{les1}), we know that  $i^k= (\Phi^k)^{-1}\circ i_k.$
Notice that $d_{+-}$ is a chain map and thus $d_{+0}\circ d_{0-}=0$.
It is not hard to show that $\Phi^k= d^k_{+-}+ d^k_{0-}$.
Thus $$H^kCone (d_{+0}) \cong d^k_{+-}\big( H^k Hom_+(\epsilon^1_-,\epsilon^2_-)\big) \oplus d^k_{0-}\big( H^k Hom_+(\epsilon^1_-,\epsilon^2_-)\big). $$
Since we are working over the field $\bbF$, we have the following relation on dimensions:
$$
\begin{array}{rl}
\dim \big( H^k Hom_+(\epsilon^1_-,\epsilon^2_-)\big) &=\ \dim \big( H^kCone (d_{+0}) \big)\vspace{0.05in}\\
&=\  \dim \big( d^k_{+-}\big( H^k Hom_+(\epsilon^1_-,\epsilon^2_-)\big)\big) + \dim \big(d^k_{0-}\big( H^k Hom_+(\epsilon^1_-,\epsilon^2_-)\big)\big).
\end{array}$$
Thus $ \dim \big(d^k_{+-}\big( H^k Hom_+(\epsilon^1_-,\epsilon^2_-)\big)\big) \leq \dim \big(H^k Hom_+(\epsilon^1_-,\epsilon^2_-)\big)$ and the equality holds if and only if $d^k_{0-}=0$, which is equivalent to the condition that $i^k$  is surjective.
Hence $d^k_{+-}$ is injective if and only if $i^k$  is surjective.

The proof of the statement that $i^k$ is injective if and only if $d_{+-}^k$ is surjective is basically the same if we consider the Cthulhu chain complex as a mapping cone of $\Psi: Cone(d_{0-}) \to Hom_+(\epsilon^1_+, \epsilon^2_+)$, where $\Psi= d_{+0} + d_{+-}$.
\end{proof}

Thanks to Theorem \ref{equal},  we know that $d_{+-}$ agrees with $f_1$. 
Theorem \ref{main1} shows that $i^0$ is both injective and surjective. Therefore we have the following corollary.
\begin{cor}\label{iso}
Let $f^*$ denote the induced map of $f_1$ on cohomology.
Then $f^*$ restricted on $0$ degree cohomology  $$f^*: H^0 Hom_+(\epsilon^1_-,\epsilon^2_-)\to H^0 Hom_+(\epsilon^1_+,\epsilon^2_+)$$
 is an isomorphism.

\end{cor}

\begin{thm}\label{injectivity}
Let  $\Sigma$ be an exact Lagrangian cobordism from a Legendrian knot $\Lambda_-$ to a Legendrian knot $\Lambda_+$ with Maslov number $0$.
The $A_{\infty}$ category map $f: \calA ug_+(\Lambda_-) \to \calA ug_+(\Lambda_+)$ induced by the exact Lagrangian cobordism $\Sigma$  is injective on the level of equivalence classes of objects. 

In other words, for $i=1,2$,
assume $\epsilon^i_-$ is an augmentation of   $\calA(\Lambda_-)$ with a single base point 
and $\epsilon^i_+$ is the  augmentation of $\calA(\Lambda_+)$ with a single base point induced by $\Sigma$.
If  $\epsilon^1_+$ and $\epsilon^2_+$ are equivalent in $\calA ug_+(\Lambda_+)$, then $\epsilon^1_-$ and $\epsilon^2_-$ are equivalent in $\calA ug_+(\Lambda_-)$.
\end{thm}

\begin{proof}
Since $\epsilon^1_+$ and $\epsilon^2_+$ are equivalent in $\calA ug_+(\Lambda_+)$, there exist
$[\alpha_+] \in H^0Hom_+(\epsilon^1_+, \epsilon^2_+)$ and $[\beta_+] \in H^0Hom_+(\epsilon^2_+, \epsilon^1_+)$ such that 
$$[m_2(\alpha_+, \beta_+)] = [e_{\epsilon^2_+}] \in H^0Hom_+(\epsilon^2_+, \epsilon^2_+)$$
 and 
 $$[m_2( \beta_+,\alpha_+)] = [e_{\epsilon^1_+}] \in H^0Hom_+(\epsilon^1_+, \epsilon^1_+).$$
 
Corollary \ref{iso} shows that $f^* :  H^0Hom_+(\epsilon^1_-, \epsilon^2_-) \to H^0Hom_+(\epsilon^1_+, \epsilon^2_+) $ is an isomorphism. 
 Hence there exists $[\alpha_-] \in H^0Hom_+(\epsilon^1_-, \epsilon^2_-)$ such that 
 $$f^*([\alpha_-])=[\alpha_+] \in H^0Hom_+(\epsilon^1_+, \epsilon^2_+).$$ 
Similarly,  there exists $[\beta_-] \in H^0Hom_+(\epsilon^2_-, \epsilon^1_-)$ such that 
 $$f^*([\beta_-])=[\beta_+] \in H^0Hom_+(\epsilon^2_+, \epsilon^1_+).$$
Moreover, we have $$f^*[m_2(\alpha_-, \beta_-)]=
 m_2(f^*([\alpha_-]),f^*([\beta_-]))=m_2([\alpha_+],[\beta_+])=[e_{\epsilon^2_+}] \in H^0Hom_+(\epsilon^2_+, \epsilon^2_+)
.$$

Notice that $f$ sends $y_{-}^{\vee} \in  Hom^0_+(\epsilon^2_-, \epsilon^2_-) $ 
to $y_{+}^{\vee} \in Hom^0_+(\epsilon^2_+, \epsilon^2_+)$
and hence $f^*[e_{\epsilon^2_-}]=[e_{\epsilon^2_+}]$.
By Corollart \ref{iso}, the map $f^* :  H^0Hom_+(\epsilon^2_-, \epsilon^2_-) \to H^0Hom_+(\epsilon^2_+, \epsilon^2_+) $ is an isomorphism. 
Hence
$[m_2(\alpha_-, \beta_-)]=[e_{\epsilon^2_-}] \in  H^0Hom_+(\epsilon^2_-, \epsilon^2_-)$. 
Similarly, we have
$[m_2( \beta_-,\alpha_-)]=[e_{\epsilon^1_-}] \in  H^0Hom_+(\epsilon^1_-, \epsilon^1_-)$. 
Therefore $\epsilon^1_-$ and $\epsilon^2_-$ are equivalent in $\calA ug_+(\Lambda_-)$.
\end{proof}
In addition, the exact Lagrangian cobordism $\Sigma$ described above also induces a category functor  on the cohomology category as described in Section \ref{Acat} $$\tilde{f}: H^*\calA ug_+(\Lambda_-) \to H^*\calA ug_+(\Lambda_+).$$
We have the following  statement.

\begin{thm}\label{cocat}
Let  $\Sigma$ be an exact Lagrangian cobordism from a Legendrian knot $\Lambda_-$ to a Legendrian knot $\Lambda_+$ with Maslov number $0$.
The corresponding cohomology category map $ \tilde{f}: H^* \calA ug_+(\Lambda_-) \to H^* \calA ug_+(\Lambda_+)$ induced by  $\Sigma$  is faithful.
Moreover, if $\chi(\Sigma)=0$, this functor is fully faithful. 
\end{thm}

\begin{proof}
Notice that  the category map $\tilde{f}$ restricted on the level of morphisms is 
$$f^*: H^*Hom_+(\epsilon^1_-, \epsilon^2_-) \to  H^*Hom_+(\epsilon^1_+, \epsilon^2_+).$$
The long exact sequence (\ref{les1}) tells us that $i^k$ are surjective for all the $k \in \bbZ$. 
By Theorem \ref{thm 1}, we know that $f^*$ is injective. Therefore  $\tilde{f}$ is faithful.

In particular, if $\chi(\Sigma)=0$, Theorem \ref{main1} implies that $i^k$ are isomorphisms for all the $k \in \bbZ$.
Therefore, by Theorem \ref{thm 1}, the map $f^*$ is an isomorphism, which implies that $\tilde{f}$ is fully faithful.
\end{proof}

As a result of Theorem \ref{injectivity}, there is an induced map from the equivalence classes of augmentations of 
$\Lambda_-$ to the equivalence classes of augmentations of $\calA ug_+(\Lambda_+)$.
Thus the number of equivalence classes of augmentations of $\calA(\Lambda_-)$ is less than or equal to
the number of  equivalence classes of augmentations of $\calA(\Lambda_+)$. 
However, the equivalence classes of augmentations is difficult to count in general. 
Ng, Rutherford, Shende and Sivek \cite{NRSS} introduced another way to count objects: the homotopy cardinality of $\pi_{\ge 0} \calA ug_+(\Lambda; \bbF_q)^*$, where $\bbF_q$ is a finite field.
This can be computed using  ruling polynomials.

The  {\bf homotopy cardinality} is defined by
$$\pi_{\ge 0} \calA ug_+(\Lambda; \bbF_q)^* = \displaystyle{
\sum_{[\epsilon] \in \calA ug_+(\Lambda; \bbF_q)/ \sim} \frac{1}{|Aut(\epsilon)|} \cdot 
\frac{|H^{-1}Hom_+(\epsilon, \epsilon)| \cdot |H^{-3}Hom_+(\epsilon, \epsilon)| \cdots}
{|H^{-2}Hom_+(\epsilon, \epsilon)| \cdot |H^{-4}Hom_+(\epsilon, \epsilon)| \cdots}}
,$$
where $[\epsilon]$ is the equivalence class of $\epsilon$ in the augmentation category $\calA ug_+(\Lambda)$
and $|Aut(\epsilon)|$ is the number of invertible elements in $H^0Hom_+(\epsilon, \epsilon)$.

\begin{cor}
Let $\Sigma$ be a spin exact Lagrangian cobordism from a Legendrian knot $\Lambda_-$ to a Legendrian knot $\Lambda_+$ 
with Maslov number $0$. Then for any finite field $\bbF_q$, we have
$$\pi_{\ge 0} \calA ug_+(\Lambda_-; \bbF_q)^* \leq \pi_{\ge 0} \calA ug_+(\Lambda_+; \bbF_q)^*.$$
\end{cor}
\begin{proof}
Assume $[\epsilon_-]$ is an equivalence class in $\calA ug_+(\Lambda_-; \bbF_q)$ 
and $[\epsilon_+]$ is the induced equivalence class in $\calA ug_+(\Lambda_+; \bbF_q)$.
Theorem \ref{main1} implies 
$$H^k Hom_+(\epsilon_-,\epsilon_-) \cong H^k Hom_+(\epsilon_+,\epsilon_+) \textrm{ for } k<1.$$
In particular, we have
$H^0 Hom_+(\epsilon_-,\epsilon_-) \cong H^0 Hom_+(\epsilon_+,\epsilon_+)$,
which implies $|Aut(\epsilon_-) |= |Aut(\epsilon_+)|$.
Notice that $\calA ug_+(\Lambda_+; \bbF_q)$ may have more equivalence classes than 
$\calA ug_+(\Lambda_-; \bbF_q)$.
Therefore, we have
$$\pi_{\ge 0} \calA ug_+(\Lambda_-; \bbF_q)^* \leq \pi_{\ge 0} \calA ug_+(\Lambda_+; \bbF_q)^*.$$
\end{proof}

From \cite[Corollary 2]{NRSS}, this cardinality can be related to the ruling polynomial in the following way:
$$\pi_{\ge 0} \calA ug_+(\Lambda; \bbF_q)^*= q^{tb(\Lambda)/2} R_{\Lambda}(q^{1/2} - q^{-1/2}).$$

Recall that a {\bf normal ruling} $R$ is a decomposition of the front projection of $\Lambda$ into embedded disks connected by switches that satisfy some requirements (see details in \cite{Cheruiling}).
The {\bf ruling polynomial} is defined by
$$R_{\Lambda}(z)=\displaystyle{\sum_{R} z^{\# (\textrm{switches}) - \# (\textrm{diskes})}}.$$

\begin{cor}\label{ruling}
Suppose there is a spin exact Lagrangian cobordism from a Legendrian knot $\Lambda_-$ to a Legendrian knot $\Lambda_+$ 
with Maslov number $0$. 
Then the ruling polynomials satisfy:
$$R_{\Lambda_-}(q^{1/2} - q^{-1/2}) \leq q^{-\chi(\Sigma)/2} R_{\Lambda_+}(q^{1/2} - q^{-1/2}),$$
for any $q$ that is a power of a prime number.
\end{cor}

When $\Sigma$ is decomposable, i.e. consists of pinch moves and minimum cobordisms \cite{EHK},
there is a map from the rulings of $\Lambda_-$ to rulings of $\Lambda_+$.
For each pinch move or minimal cobordism, 
any normal ruling of the bottom knot gives a normal ruling of the top knot, as shown in Figure \ref{pinch}.
Moreover, different rulings of the bottom knot give different rulings of the top knot.
Therefore the ruling polynomials of $\Lambda_+$ and $\Lambda_-$ satisfy the relation in Corollary \ref{ruling}.
This corollary shows that the result is  true even if the cobordism is not decomposable.

\begin{figure}[!ht]
\labellist
\pinlabel $\uparrow$ at 40 40
\pinlabel {the pinch move} at 40 -10
\pinlabel {the minimum cobordism} at 180 -10
\endlabellist
\includegraphics[width=3in]{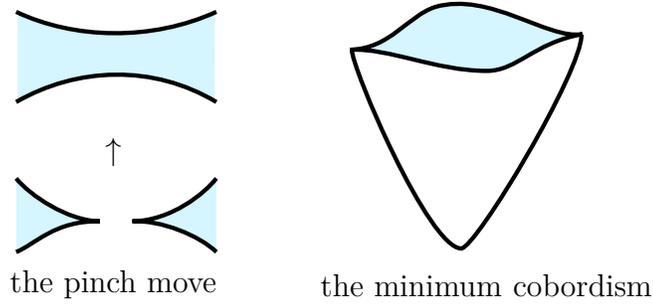}
\vspace{0.2in}
\caption{The relation between rulings of Legendrian submanifolds that are related by a pinch move or a minimum cobordism.}
\label{pinch}
\end{figure}

One can check the atlas in \cite{CNatlas}
 for the ruling polynomials of small crossing Legendrian knots.
This corollary gives a new and easily computable obstruction to the existence of exact Lagrangian cobordisms.
We can use Corollary \ref{ruling} to give a new proof of the follow theorem which is a result in \cite{Chasymmetric} and was reproved by \cite{CNS}.
\begin{thm}[\cite{Chasymmetric}] 
Lagrangian concordance is not a symmetric relation.
\end{thm} 
\begin{proof}
Consider the Legendrian knot $\Lambda$ of smooth knot type $m(9_{46})$ with maximum Thurston-Bennequin number  and the Legendrian unknot $\Lambda_0$ as shown in Figure \ref{9_46}. 
There is an exact Lagrangian concordance from the $\Lambda_0$ to $\Lambda$ by doing a pinch move at the red line in Figure \ref{9_46} and Legendrian isotopy.
However, there does not exist an exact Lagrangian concordance from $\Lambda$ to $\Lambda_0$ since the ruling polynomial of $\Lambda$ is 2  while the ruling polynomial of $\Lambda_0$  is $1$. 
\end{proof}
\vspace{-0.10in}
\begin{figure}[!ht]
\labellist
\pinlabel $\Lambda$ at 130 0
\pinlabel $\Lambda_0$  at 380 0
\endlabellist
\includegraphics[width=4in]{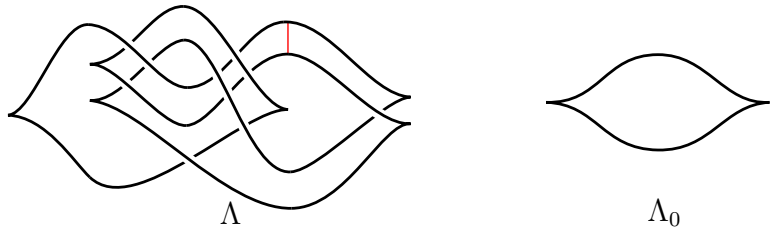}
\caption{Front projections of Legendrian knot $\Lambda$ of knot type $m(9_{46})$ and the Legendrian unknot $\Lambda_0$.}
\label{9_46}
\end{figure}

\vfill

\pagebreak

\bibliographystyle{alpha}

\newcommand{\etalchar}[1]{$^{#1}$}

\end{document}